\documentclass[12pt]{amsart}
\usepackage[utf8]{inputenc}
\usepackage{amsaddr}
\usepackage{comment}
\usepackage{amsmath,amsfonts,amssymb,amsthm}
\usepackage{xcolor} 
\usepackage{bm}
\usepackage{enumitem}
\usepackage{mathrsfs}
\usepackage{multicol}
\usepackage{centernot}
\usepackage{mathtools}
\usepackage{ stmaryrd }
\usepackage{yhmath}
\usepackage{tikz-cd}
\usepackage{url}
\usepackage{hyperref}
\usepackage[margin = 1in]{geometry}
\usepackage{tabularx}
\usepackage{ltablex}
\newcommand{\nind}[3]{#1 & \hspace{.5cm} & #2; #3.} 

\usepackage[citestyle=numeric-comp, backend=biber, maxbibnames=9]{biblatex}
\DeclareFieldFormat[article]{title}{#1} 
\DeclareFieldFormat[inproceedings]{title}{#1} 
\DeclareFieldFormat[inbook]{title}{#1} 
\AtBeginBibliography{\small}

\newtheorem{thm}{Theorem}

\newtheorem{cor}[thm]{Corollary}
\newtheorem{lemma}[thm]{Lemma}
\newtheorem{prop}[thm]{Proposition}

\theoremstyle{definition}
\newtheorem{defn}[thm]{Definition}

\newtheorem{remark}[thm]{Remark}
\newtheorem{non-example}[thm]{Non-Example}

\newtheorem{fact}[thm]{Fact}

\DeclareMathOperator{\End}{End}
\DeclareMathOperator{\Hom}{Hom}
\DeclareMathOperator{\Res}{Res}
\DeclareMathOperator{\Ind}{Ind}
\DeclareMathOperator{\Irr}{Irr}
\DeclareMathOperator{\diag}{diag}
\DeclareMathOperator{\antidiag}{antidiag}
\def\mf{\mathfrak}

\def\phi{\varphi}

\def\wtilde{\widetilde} 
\def\hat{\widehat}

\def\C{\mathbb{C}}
\title{Dual Pairs in $PGL(n,\C)$}

\author{Marisa Gaetz}
\email{mgaetz@mit.edu}
\thanks{The author was supported by the NSF Graduate Research Fellowship Program under Grant Nos.~1745302 and 2141064, and by the Fannie \& John Hertz Foundation.}

\begin{document}

\maketitle

\begin{abstract}
In Roger Howe's seminal 1989 paper ``Remarks on classical invariant theory," he introduces the notion of Lie algebra dual pairs, and its natural analog in the groups context: a pair $(G_1,G_2)$ of reductive subgroups of an algebraic group $G$ is a \textit{dual pair} in $G$ if $G_1$ and $G_2$ equal each other's centralizers in $G$. While reductive dual pairs in the complex reductive Lie algebras have been classified, much less is known about algebraic group dual pairs, which were only fully classified in the context of certain classical matrix groups. In this paper, we classify the reductive dual pairs in $PGL(n,\mathbb{C})$. 
\end{abstract}

\tableofcontents

\section{Introduction} \label{sec:introduction}

Throughout this paper, $U$ will denote a finite-dimensional complex vector space. The goal of this paper is to classify the conjugacy classes of \textit{reductive dual pairs} in $PGL(U)$.

\begin{defn}[{Howe \cite[p.~551]{HoweRemarks}}] \label{defn:dual-pair}
Let $G$ be a complex algebraic group. A pair $(G_1, G_2)$ of algebraic subgroups of $G$ is a \textbf{reductive dual pair} in $G$ if
\begin{enumerate}[label = (\arabic*)]
\item $G_1$ and $G_2$ are reductive in $G$, and 
\item $Z_G(G_1) = G_2$ and $Z_G(G_2) = G_1$,
\end{enumerate} 
where $Z_G(\cdot)$ denotes taking the centralizer in $G$.
\end{defn}

Since we will only be considering reductive dual pairs in this paper, we will just use ``dual pair" to refer to a reductive dual pair.

\subsection{Prior work}

Dual pairs were first introduced in Roger Howe's seminal 1989 paper \cite{HoweRemarks}, in which he primarily focuses on \textit{Lie algebra dual pairs}: a pair of Lie subalgebras $(\mf{g}_1, \mf{g}_2)$ of a Lie algebra $\mf{g}$ is a \textit{dual pair} in $\mf{g}$ if $\mf{z}_{\mf{g}}(\mf{g}_1) = \mf{g}_2$ and $\mf{z}_{\mf{g}}(\mf{g}_2) = \mf{g}_1$ (where $\mf{z}_{\mf{g}}(\cdot)$ denotes the centralizer in $\mf{g}$). Since their initial introduction by Howe, dual pairs in Lie groups and Lie algebras have been widely used and studied \cite{superalgebras,Collingwood, myarxiv, my-liealgebras,cubic,Ginzburg,Savin4,Savin2,KovT,Kov,Cases,Savin3,theta-lifting,ohta-1,Panyushev,rubenthaler,Savin1,partial-ordering,global-theta-lifting,theta-lifting-reps,local-theta-lifting,przebinda-4,przebinda-5,przebinda-6,przebinda-7,przebinda-8,przebinda-9}.

Some of the earliest significant work on dual pairs was done by H.~Rubenthaler in a 1994 paper \cite{rubenthaler}. In this paper, Rubenthaler outlines a classification of the conjugacy classes of reductive dual pairs in complex reductive Lie algebras. The paper \cite{my-liealgebras} uses Rubenthaler's outline to obtain complete and explicit lists of the dual pairs in the complex exceptional Lie algebras, and the paper \cite{myarxiv} derives independent classifications of the dual pairs in the complex classical Lie algebras using elementary methods. Lie algebra dual pairs are also intimately connected to the structure theory of reductive Lie algebras. Namely, certain reductive dual pairs in the exceptional Lie algebras are implicitly used in some of the well-known constructions of these algebras: $(G_2^1, A_2^{2''})$ and $(G_2^1, F_4^1)$ correspond to the Freudenthal-Tits constructions of $E_6$ and $E_8$, respectively \cite{tits-construction}, and $(D_4, D_4)$ corresponds to the triality construction of $E_8$ \cite{allison, barton-sudbery}.

While dual pairs in Lie algebras have been extensively studied and have been classified in the primary case of interest (i.e.,~reductive dual pairs in complex reductive Lie algebras), Lie group dual pairs are not understood in as much generality. In the same paper in which he introduces the notion of dual pairs \cite{HoweRemarks}, Howe classifies the \textit{classical reductive dual pairs} in the complex symplectic group $Sp(2n,\C)$, where a \textit{classical reductive dual pair} is a reductive dual pair in which both members are themselves classical groups. This classification is nicely summarized by T.~Levasseur and J.T.~Stafford in \cite{Cases}, where they go on to use these dual pairs in their study of rings of differential operators on classical rings of invariants.

The classical reductive dual pairs in $Sp(2n,\C)$ turn out to coincide with the \textit{irreducible reductive dual pairs} in $Sp(2n,\C)$, as defined by M.~Schmidt in \cite{partial-ordering}. In the same paper, Schmidt derives complete lists of the conjugacy classes of the reductive dual pairs in the groups of isometries of real, complex, and quaternionic Hermitian spaces (including $O(n,\C)$ and $Sp(2n,\C)$). Classifications of the reductive dual pairs in $GL(n,\C)$, $SL(n,\C)$, and $SO(n,\C)$ are straightforward to derive using similar methods. The paper \cite{myarxiv} explicitly describes classifications for $GL(n, \C)$, $SL(n, \C)$, $Sp(2n, \C)$, $O(n, \C)$, and $SO(n, \C)$ and establishes certain classes of dual pairs in the projective counterparts of these groups (i.e.,~$PGL(n,\C)$, $PSp(2n,\C)$, $PO(n,\C)$, and $PSO(n,\C)$). In particular, \cite{myarxiv} classifies the connected dual pairs in $PGL(n,\C)$ (recounted in Subsection \ref{subsec:connected-dps}).

Except for the complex classical groups and for the isometry groups considered in \cite{partial-ordering} and for the work in \cite{myarxiv}, classifications of reductive dual pairs in Lie groups seem to be unknown, as the literature only mentions a scattered assortment of dual pairs. The majority of the results about Lie group dual pairs concern groups defined over $p$-adic fields and are due to G.~Savin and various collaborators \cite{Savin4, Savin2, Savin3, Savin1}. For example, in \cite{Savin1}, G.~Savin and M.~Woodbury consider certain dual pairs of split algebraic groups of types $D_4$, $D_5$, $D_6$, $E_6$, $E_7$, and $E_8$ defined over a $p$-adic field. For \textit{exceptional dual pairs} of the form $(G_2, H)$ (i.e.,~where one of the groups is of type $G_2$), they construct correspondences between the representations of $G_2$ and of $H$.

\subsection{Further motivation}

In addition to being a natural analog of the well-studied notion of dual pairs in Lie algebras, the topic of dual pairs in Lie groups is made interesting by its potential to play an important role in the study of nilpotent orbits in complex semisimple Lie algebras. To see this, let $\mf{g}$ be a complex semisimple Lie algebra, and let $G_{\text{ad}}$ denote its adjoint group. As was shown by Mal'cev and Kostant \cite{malcev, kostant}, the nonzero nilpotent $G_{\text{ad}}$-orbits in $\mf{g}$ are in one-to-one correspondence with the $G_{\text{ad}}$-conjugacy classes of the $\mf{sl}_2$-subalgebras of $\mf{g}$ \cite[Section 3.4]{Collingwood}. (In this correspondence, the nilpotent element $X \in \mf{g}$ defining the orbit $\mathcal{O}_X$ corresponds to the nilpotent element of a standard $\mf{sl}_2$-triple $\{H,X,Y\}$). 

Let $\phi$ denote the homomorphism of $\mf{sl}_2$ into $\mf{g}$ determined by the standard triple $\{ H,X,Y \}$, and let $G_{\text{ad}}^{\phi}$ denote the centralizer of $\phi (\mf{sl}_2)$ in $G_{\text{ad}}$. The group $G_{\text{ad}}^{\phi}$ is a complex reductive Lie group and is usually disconnected. Since $G_{\text{ad}}^{\phi}$ is potentially disconnected, it cannot be fully understood using based root datum and the structure theory of connected reductive algebraic groups. However, $G_{\text{ad}}^{\phi}$ could be better understood with the help of the theory of Lie group dual pairs. Indeed, since $SL(2,\C)$ is simply connected, there is a unique Lie group homomorphism $\Phi \colon SL(2,\C) \rightarrow G_{\text{ad}}$ such that $\phi = d\Phi$. Moreover, $G_{\text{ad}}^{\phi} = G_{\text{ad}}^{\Phi}$, where $G_{\text{ad}}^{\Phi}$ is the centralizer of $\Phi(SL(2,\C))$ in $G_{\text{ad}}$. Therefore, by the following general fact, $(G_{\text{ad}}^{\phi}, Z_{G_{\text{ad}}}( G_{\text{ad}}^{\phi} ))$ is a reductive dual pair in $G_{\text{ad}}$:

\begin{fact} \label{fact:triple centralizer}
Let $G$ be a group, and $S \subseteq G$ a subset. Then $Z_G(Z_G(Z_G(S))) = Z_G(S)$. In particular, $(Z_G(S), Z_G(Z_G(S)) )$ is a dual pair in $G$.
\end{fact}

More generally, one might hope to better understand the structure of a complex reductive algebraic group $G$ by looking at the groups $G^{\Phi}$ corresponding to various algebraic group homomorphisms $\Phi \colon H \rightarrow G$. Here, $G^{\Phi}$ is a complex reductive algebraic group and is usually disconnected (and hence again not easily understandable). However, as before, Fact \ref{fact:triple centralizer} shows that $(G^{\Phi}, Z_G(G^{\Phi}))$ is a reductive dual pair in $G$. Moreover, we note that all reductive dual pairs in $G$ arise in this way:

\begin{remark} \label{rmk:all pairs}
Let $G$ be a complex reductive algebraic group. Then any reductive dual pair $(G_1, G_2)$ in $G$ can be written in the form $(G^{\Phi}, Z_G(G^{\Phi}))$. Indeed, take $\Phi$ to be the inclusion $G_2 \hookrightarrow G$. We get $G_1 = Z_G(\Phi (G_2))= G^{\Phi}$ and $G_2 = Z_G(G_1) = Z_G(G^{\Phi})$.
\end{remark}

The goal of this paper is to begin to understand how disconnectedness factors into the classification of dual pairs by studying the dual pairs of $PGL(U)$.

\section{Constructing ``single-orbit" dual pairs in \texorpdfstring{$PGL(U)$}{PGL(U)}} \label{sec:constructing-single-orbit}

In this section, we classify the connected dual pairs in $PGL(U)$ (Subsection \ref{subsec:connected-dps}), and construct a class of ``single-orbit" disconnected dual pairs (Subsection \ref{subsec:single-orbit-disconnected-construction}). 

\subsection{Connected dual pairs in $PGL(U)$} \label{subsec:connected-dps}

To start, we state a result about dual pairs in $GL(U)$. While this result is likely well-known, we cite \cite{myarxiv} for an explicit statement and treatment.

\begin{thm}[G.~{\cite[Theorem 6.2]{myarxiv}}] \label{thm:dual-pairs-in-GL}
The dual pairs of $GL(U)$ are exactly the pairs of groups of the form
\begin{equation*} 
\left ( \prod_{i=1}^{r} GL(V_i) , \; \prod_{i=1}^r GL(W_i) \right ),
\end{equation*}
where $U = \bigoplus_{i=1}^{r} V_i \otimes W_i$ is a vector space decomposition of $U$. In particular, any member of a $GL(U)$ dual pair is connected.
\end{thm}

In what follows, $p_V$ will denote the canonical projection $p_V \colon GL(V) \rightarrow PGL(V)$ for a complex vector space $V$. When the vector space is clear from context, we will sometimes just use $p$ to denote this projection.

\begin{prop}[G.~{\cite[Proposition 6.6]{myarxiv}}] \label{prop:GL-PGL-descent}
Let $(G_1, G_2)$ be a dual pair in $GL(U)$. Then $(p (G_1), p(G_2))$ is a dual pair in $PGL(U)$. 
\end{prop}

Proposition \ref{prop:GL-PGL-descent} shows that every dual pair in $GL(U)$ descends to a dual pair in $PGL(U)$ under the canonical projection. In fact, as the following proposition will show, the dual pairs in $GL(U)$ are in bijection with the connected dual pairs in $PGL(U)$. 

\begin{prop}[G.~{\cite[Proposition 6.7]{myarxiv}}] \label{prop:connected-connected}
Let $(\overline{G_1}, \overline{G_2})$ be a connected dual pair in $PGL(U)$, and define $G_1 := p^{-1}(\overline{G_1})$ and $G_2 := p^{-1}(\overline{G_2})$. Then $( G_1 , G_2 )$ is a dual pair in $GL(U)$.
\end{prop}

\subsection{Disconnected dual pairs in $PGL(U)$} \label{subsec:single-orbit-disconnected-construction}

Although all connected dual pairs in $PGL(U)$ arise as the images in $PGL(U)$ of dual pairs in $GL(U)$, we will soon see that not all dual pairs in $PGL(U)$ are connected. 

\begin{prop} \label{prop:AxA^-simple}
Let $\mathcal{X}$ be a finite abelian group of order $n$, and let $\hat{\mathcal{X}}$ denote its group of characters. Let $X$ be an $n$-dimensional complex vector space. Then 
$$\left ( \mathcal{X} \times \hat{\mathcal{X}} , \; \mathcal{X} \times \hat{\mathcal{X}} \right )$$
can be realized as a dual pair in $PGL(X)$. 
\end{prop}

\begin{proof}
Let us view $X$ as the space of functions $L^2(\mathcal{X},\mathbb{C})$. Then $f \in X$ can be viewed as a column vector $[f(x_1) \; \cdots \; f(x_n)]^T$, where $x_1, \ldots, x_n$ are the elements of $\mathcal{X}$. Each element $x \in \mathcal{X}$ acts by translation ($\tau_x$) on $f \in X$, and each element $\xi \in \hat{\mathcal{X}}$ acts by multiplication ($\sigma_{\xi}$) on $f \in X$. Since multiplying by $x^{-1}$ permutes the elements of $\mathcal{X}$, each
\begin{align*}
    \tau_x \colon X & \rightarrow X \\
    f(x_0) & \mapsto f(x_0 x^{-1})
\end{align*}
can be viewed as a permutation matrix in $GL(X)$, and $x \in \mathcal{X}$ can be viewed as $p(\tau_x) \in PGL(X)$. Additionally, we can view each
\begin{align*}
    \sigma_{\xi} \colon X & \rightarrow X \\
    f(x_0) & \mapsto \xi (x_0) f(x_0)
\end{align*}
as a diagonal matrix $\sigma_{\xi} = \diag (\xi(x_1), \ldots, \xi (x_n)) \in GL(X)$, and $\xi \in \hat{\mathcal{X}}$ can be viewed as $p ( \sigma_{\xi} )\in PGL(X)$.

Now, observe that we have the following:
\begin{align*}
(\sigma_{\xi} \tau_x \sigma_{\xi}^{-1} f)(x_0) &= \xi (x_0) (\tau_x \sigma_{\xi}^{-1} f)(x_0) = \xi (x_0) (\sigma_{\xi}^{-1} f)(x_0 x^{-1}) \\
&= \xi(x_0) \xi (x_0 x^{-1})^{-1} f(x_0x^{-1}) = \xi(x) (\tau_x f)(x_0)
\end{align*}
and
\begin{align*}
(\tau_x \sigma_{\xi} \tau_x^{-1} f)(x_0) &= (\sigma_{\xi} \tau_x^{-1} f)(x_0 x^{-1}) = \xi (x_0 x^{-1}) (\tau_x^{-1} f)(x_0 x^{-1}) \\
&= \xi (x_0 x^{-1}) f(x_0) = \xi (x^{-1}) (\sigma_{\xi} f)(x_0).
\end{align*}
In summary, we have the following relations:
\begin{equation} \label{eq:sigma-tau-relation}
\sigma_{\xi} \tau_x \sigma_{\xi}^{-1} = \xi (x) \tau_x \hspace{.25cm} \text{ and } \hspace{.25cm} \tau_x \sigma_{\xi} \tau_x^{-1} = \xi (x^{-1}) \sigma_{\xi}.
\end{equation}

From these relations, we see that $\langle \{ p(\tau_x) \}_{x \in \mathcal{X} }, \, \{ p(\sigma_{\xi}) \}_{\xi \in \hat{\mathcal{X}}} \rangle = \mathcal{X} \times \hat{\mathcal{X}}$ and that $\mathcal{X} \times \hat{\mathcal{X}}$ is contained in its own centralizer in $PGL(X)$. We have left to show that 
$$Z_{PGL(X)}(\mathcal{X} \times \hat{\mathcal{X}}  ) \subseteq \mathcal{X} \times \hat{\mathcal{X}} .$$ 

To this end, let $t \in p^{-1} (Z_{PGL(X)}( \mathcal{X} \times \hat{\mathcal{X}} ) )$. Then for each $\xi \in \hat{\mathcal{X}}$, 
\begin{equation} \label{eq:self-dp-commute}
t \sigma_{\xi} t^{-1} = k_{t,\xi} \sigma_{\xi} \text{ for some }  k_{t, \xi} \in \mathbb{C}^{\times}.
\end{equation}
Recall that $\sigma_{\xi} = \diag (\xi(x_1), \ldots, \xi (x_n))$, and assume (without loss of generality) that $x_1 = 1$. Since the irreducible characters of a finite group are linearly independent, the linear span of the $\sigma_{\xi}$ is all diagonal matrices. Consequently, (\ref{eq:self-dp-commute}) implies that $t$ normalizes the set of diagonal matrices, meaning
$$t \in N_{GL(X)}(\{\text{diagonal matrices}\}) = \{\text{permutation matrices}\} \cdot \{\text{diagonal matrices}\}.$$ 
Therefore, $t$ can be written as $t = sd$ for some permutation matrix $s$ and some diagonal matrix $d$. Then
$$s \sigma_{\xi} s^{-1} = t \sigma_{\xi} t^{-1} = k_{t,\xi} \sigma_{\xi} \text{ for all } \xi \in \hat{\mathcal{X}},$$
which gives that $\xi (s^{-1} \cdot x_i) = k_{t,\xi} \xi (x_i)$ for all $1 \leq i \leq n$ and $\xi \in \hat{\mathcal{X}}$. In particular, considering $x_1 = 1$, we see that 
$$\xi (s^{-1} \cdot 1) = k_{t,\xi} \xi(1) = k_{t,\xi} \text{ for all } \xi \in \hat{\mathcal{X}}.$$
It follows that 
$$s \sigma_{\xi} s^{-1} = \xi (s^{-1} \cdot 1) \sigma_{\xi} = \tau_{s \cdot 1} \sigma_{\xi} \tau_{s \cdot 1}^{-1} \text{ for all } \xi \in \hat{\mathcal{X}},$$
where we have used (\ref{eq:sigma-tau-relation}). It follows that $s = \tau_{s \cdot 1}$ and that $t = \tau_{s \cdot 1} d$.

Using again our assumption that $t \in p^{-1}(Z_{PGL(X)}(  \mathcal{X} \times \hat{\mathcal{X}} ))$, we have for each $x \in \mathcal{X}$ that
$$\tau_x t \tau_x^{-1} = \ell_{t,x} t \text{ for some } \ell_{t,x} \in \mathbb{C}^{\times}.$$
Since $\tau_x t \tau_x^{-1} = \tau_x \tau_{s \cdot 1} d \tau_x^{-1} = \tau_{s \cdot 1} (\tau_x d \tau_x^{-1})$, we see that $\tau_x d \tau_x^{-1} = \ell_{t,x} d$. Letting $\{ \xi_1, \ldots , \xi_n \}$ denote the elements of $\hat{\mathcal{X}}$, we can write $d = d_1 \sigma_{\xi_1} + \cdots + d_n \sigma_{\xi_n} $ for some $d_1 , \ldots , d_n \in \mathbb{C}$. Then for each $x \in \mathcal{X}$, we get that
\begin{align*}
\tau_x d \tau_x^{-1} &= d_1 \tau_x \sigma_{\xi_1} \tau_x^{-1} + \cdots + d_n \tau_x \sigma_{\xi_n} \tau_x^{-1} \\
&= d_1 \xi_1 (x^{-1}) \sigma_{\xi_1} + \cdots + d_n \xi_n(x^{-1}) \sigma_{\xi_n} \\
&= \ell_{t,x} (d_1 \sigma_{\xi_1} + \cdots + d_n \sigma_{\xi_n}).
\end{align*}
The linear independence of the $\xi_1, \ldots , \xi_n$ therefore implies that for any $1 \leq i \leq n$ with $d_i \neq 0$, we have that $\xi_i(x^{-1}) = \ell_{t,x}$. Using linear independence again, we see that $d_i \neq 0$ for exactly one value of $i$. Consequently, $d = d_i \sigma_{\xi_i}$ for some $1 \leq i \leq n$. 

Putting this all together, we see that $t = \tau_{s \cdot 1} d_i \sigma_{\xi_i}$, and hence that $p(t) \in \mathcal{X} \times \hat{\mathcal{X}} $. Therefore, $Z_{PGL(X)} ( \mathcal{X} \times \hat{\mathcal{X}} ) = \mathcal{X} \times \hat{\mathcal{X}} $, completing the proof. 
\end{proof}

For the results in the remainder of this section, we require some notation. To start, let $U_1$ and $U_2$ be two finite-dimensional complex vector spaces of dimensions $n_1$ and $n_2$, respectively. Then $GL(U_1)$ and $GL(U_2)$ naturally embed in $GL(U_1 \otimes U_2)$ in such a way that the intersection of their images is $\C^{\times}$. Explicitly, if we choose a basis for $U_2$, then $\End (U_1 \otimes U_2)$ gets identified with $n_2 \times n_2$ matrices with entries in $\End (U_1)$. We define the embedding $\iota \colon GL(U_1) \rightarrow GL(U_1 \otimes U_2)$ by 
$$x \xmapsto{\iota} \begin{pmatrix} x & & \\ & \ddots & \\ & & x \end{pmatrix} \in GL(U_1 \otimes U_2),$$
and the embedding $\kappa \colon GL(U_2) \rightarrow GL(U_1 \otimes U_2)$ by
$$ y \xmapsto{\kappa} \begin{pmatrix} y_{11}I_{U_1} & \cdots & y_{1,n_2} I_{U_1} \\ \vdots & \ddots & \vdots \\ y_{n_2,1}I_{U_1} & \cdots & y_{n_2,n_2}I_{U_1} \end{pmatrix} \in GL(U_1 \otimes U_2),$$

Let $H_i \supseteq \C^{\times}$ be a subgroup of $GL(U_i)$ for $i=1,2$. Then we see that 
$$\langle \iota (H_1), \, \kappa (H_2) \rangle = [ \iota (H_1) \times \kappa (H_2) ] / (\C^{\times})_{\antidiag} ,$$
where $(\mathbb{C}^{\times})_{\antidiag}$ denotes the antidiagonal embedding $z \mapsto (z,z^{-1})$ of $\mathbb{C}^{\times}$ in $\iota (H_1) \times \kappa ( H_2 )$. We will typically drop the $\iota$ and $\kappa$ and will denote this group by 
\begin{equation} \label{eq:weird-direct-product-def}
    H_1 \times_{\mathbb{C}^{\times}} H_2 := \left [ \iota( H_1 ) \times \kappa ( H_2 ) \right ]/(\mathbb{C}^{\times})_{\antidiag}.
\end{equation}

Next, let $\mathcal{X}$ be a finite abelian group, and let $X := L^2(\mathcal{X},\C)$. Then, as explained in the proof of Proposition \ref{prop:AxA^-simple}, $\mathcal{X}$ can be viewed as the image in $PGL(X)$ of a group of permutation matrices in $GL(X)$. Suppose that $H \supseteq \mathbb{C}^{\times}$ is a subgroup of $GL(W)^{\dim X}$ that is normalized by $\kappa( p_X^{-1}(\mathcal{X}))$ in $GL(W \otimes X)$. Then $H \cap \kappa (p_X^{-1}(\mathcal{X})) = \mathbb{C}^{\times}$ in $GL(W \otimes X)$, and 
\begin{equation} \label{eq:weird-semidirect-product-def}
\langle H, \, \kappa ( p_X^{-1}(\mathcal{X}) ) \rangle = \left [ H \rtimes \kappa ( p_X^{-1}(\mathcal{X}) ) \right ]/ (\mathbb{C}^{\times})_{\antidiag} =: H \rtimes_{\mathbb{C}^{\times}} p_X^{-1}(\mathcal{X}).
\end{equation}
Note that both $\times_{\mathbb{C}^{\times}}$ and $\rtimes_{\mathbb{C}^{\times}}$ defined in this way are associative: for $H_i \in GL(U_i)$ ($i=1,2,3$), we have
$$(H_1 \times_{\C^{\times}} H_2) \times_{\C^{\times}} H_3 \simeq H_1 \times_{\C^{\times}} (H_2 \times_{\C^{\times}} H_3) \in GL(U_1 \otimes U_2 \otimes U_3),$$
and for $\C^{\times} \subseteq H \subseteq GL(W)^{\dim X}$ (resp.~$\C^{\times} \subseteq H' \subseteq GL(W')^{\dim X}$) normalized by $\kappa ( p_X^{-1}(\mathcal{X}))$ in $GL(W \otimes X)$ (resp.~$GL(W' \otimes X)$), we have
$$ (H \times_{\C^{\times}} H') \rtimes_{\C^{\times}} p_X^{-1}(\mathcal{X}) \simeq H \times_{\C^{\times}} (H' \rtimes_{\C^{\times}} p_X^{-1}(\mathcal{X})) \in GL(W \otimes W' \otimes X).$$

We are now ready to state a well-known lemma and to establish Theorem \ref{thm:AxA^-general}, which generalizes Proposition \ref{prop:AxA^-simple} and describes a major class of disconnected dual pairs in $PGL(W \otimes X)$.

\begin{lemma} \label{lem:natural-iso}
Let $\mathcal{X}$ be a finite abelian group of order $n$, let $X = L^2(\mathcal{X},\mathbb{C})$, and let $W$ be an $m$-dimensional complex vector space. Then there is a natural isomorphism $W \otimes X \xrightarrow{\sim} L^2(\mathcal{X},W)$.
\end{lemma}

\begin{thm} \label{thm:AxA^-general}
Let $\mathcal{X}$ be a finite abelian group of order $n$, and let $\hat{\mathcal{X}}$ denote its group of characters. Let $X = L^2(\mathcal{X},\mathbb{C})$, and let $W$ be an $m$-dimensional complex vector space. Let $(\overline{H}_1, \overline{H}_2)$ be a dual pair in $PGL(W)$ with $H_i := p_W^{-1}(\overline{H}_i)$. Then the groups
$$ H_1 \times_{\mathbb{C}^{\times}} p_X^{-1}(\mathcal{X} \times \hat{\mathcal{X}}) \hspace{.5cm} \text{ and } \hspace{.5cm} H_2 \times_{\mathbb{C}^{\times}} p_X^{-1}(\mathcal{X} \times \hat{\mathcal{X}})$$
descend to a dual pair
$$\left (\overline{H}_1 \times \mathcal{X} \times \hat{\mathcal{X}} , \; \overline{H}_2 \times \mathcal{X} \times \hat{\mathcal{X}}  \right )$$
in $PGL(W \otimes X)$ with component groups $H_1/H_1^{\circ} \times \mathcal{X} \times \hat{\mathcal{X}}$ and $H_2/H_2^{\circ} \times \mathcal{X} \times \hat{\mathcal{X}}$.
\end{thm}

\begin{proof}
Lemma \ref{lem:natural-iso} shows that we can view $W \otimes X$ as the space of functions $L^2(\mathcal{X},W)$. A function $f \in L^2(\mathcal{X},W)$ can be viewed as a column vector $[f(x_1) \; \cdots \; f(x_n)]^T$, where $x_1 , \ldots , x_n $ are the elements of $\mathcal{X}$. For each $\xi \in \hat{\mathcal{X}}$, we can view
\begin{align*}
    \sigma_{\xi} \colon W \otimes X & \rightarrow W \otimes X \\
    f(x_0) & \mapsto \xi (x_0 ) f(x_0)
\end{align*}
as a block diagonal matrix $\sigma_{\xi} = \diag ( \xi (x_1) \cdot \text{id}_W , \ldots, \xi (x_n) \cdot \text{id}_W )$. Additionally, since multiplying by $x^{-1}$ permutes the elements of $\mathcal{X}$, each 
\begin{align*}
    \tau_x \colon W \otimes X & \rightarrow W \otimes X \\
    f (x_0 ) & \mapsto f(x_0 x^{-1}) 
\end{align*}
can be represented as a block permutation matrix in $GL(W \otimes X)$. By the calculations from the proof of Proposition \ref{prop:AxA^-simple}, we get the inclusions 
$$ \overline{H}_1 \times \mathcal{X} \times \hat{\mathcal{X}} \subseteq Z_{PGL(W \otimes X)} (\overline{H}_2 \times \mathcal{X} \times \hat{\mathcal{X}} )$$ 
and 
$$\overline{H}_2 \times \mathcal{X} \times \hat{\mathcal{X}} \subseteq Z_{PGL(W \otimes X)} (\overline{H}_1 \times \mathcal{X} \times \hat{\mathcal{X}} ).$$

Now, let $t \in p^{-1} (Z_{PGL(W \otimes X)} (\overline{H}_1 \times \mathcal{X} \times \hat{\mathcal{X}} ) )$. Then, in particular, $t$ commutes with each $\sigma_{\xi}$ up to a scalar. By the same argument as in the proof of Proposition \ref{prop:AxA^-simple}, this shows that $t$ is a product of a block permutation matrix $s^{-1}$ and a block diagonal matrix $d$, and that $s^{-1} = \tau_{s \cdot 1}$. Using the proof of Proposition \ref{prop:AxA^-simple} yet again, we see that $d = \diag (d_0, \ldots , d_0) \cdot \sigma_{\xi_i}$ for some $d_0 \in GL(W)$ and some $1 \leq i \leq n$. Then $d_0$ commutes (up to a scalar) with every element of $H_1$, meaning $d_0 \in H_2$. In this way, we see that $p(t) \in \overline{H}_2 \times \mathcal{X} \times \hat{\mathcal{X}}$, and hence that 
$$Z_{PGL(W \otimes X)} (\overline{H}_1 \times \mathcal{X} \times \hat{\mathcal{X}} ) = \overline{H}_2 \times \mathcal{X} \times \hat{\mathcal{X}} .$$ 
Reversing the roles of $H_1$ and $H_2$ shows that 
$$Z_{PGL(W \otimes X)} (\overline{H}_2 \times \mathcal{X} \times \hat{\mathcal{X}} ) = \overline{H}_1 \times \mathcal{X} \times \hat{\mathcal{X}} ,$$ 
completing the proof.   
\end{proof}

The following theorem describes classes of disconnected dual pairs in $PGL(W \otimes X)$ constructed from certain $PGL(W)$ or $GL(W)$ dual pairs.

\begin{thm} \label{thm:type-2}
Let $\mathcal{X}$ be a finite abelian group of order $n$, and let $\hat{\mathcal{X}}$ denote its group of characters. Let $X = L^2 (\mathcal{X},\mathbb{C})$, and let $W$ be an $m$-dimensional complex vector space. 
\begin{enumerate}[label = (\roman*)]
\item Let $(\overline{H}_1, \overline{H}_2)$ be a dual pair in $PGL(W)$ (with $H_i := p_W^{-1}(\overline{H}_i)$) such that $H_1^{\circ} = \mathbb{C}^{\times}$. Then the groups
$$ [ H_1  \times_{\mathbb{C}^{\times}} (\mathbb{C}^{\times})^n ] \rtimes_{\mathbb{C}^{\times}} p_X^{-1}(\mathcal{X}) \hspace{.5cm} \text{ and } \hspace{.5cm}  H_2 \times_{\mathbb{C}^{\times}} p_X^{-1} (\hat{\mathcal{X}})$$ 
descend to a dual pair
$$ ( [\overline{H}_1 \times p_X((\mathbb{C}^{\times})^n) ] \rtimes \mathcal{X} , \; \overline{H}_2 \times \hat{\mathcal{X}} ) $$
in $PGL(W \otimes X)$ with component groups $H_1/H_1^{\circ} \times \mathcal{X}$ and $H_2/H_2^{\circ} \times \hat{\mathcal{X}}$.
\item Let $(H_1, H_2)$ be a dual pair in $GL(W)$. Then the groups
$$H_1^n \rtimes_{\mathbb{C}^{\times}} p_X^{-1}(\mathcal{X}) \hspace{.5cm} \text{ and } \hspace{.5cm} H_2 \times_{\mathbb{C}^{\times}} p_X^{-1}(\hat{\mathcal{X}})$$
descend to a dual pair
$$( p(H_1^n) \rtimes \mathcal{X} , \; p(H_2) \times \hat{\mathcal{X}} )$$
in $PGL(W \otimes X)$ with component groups $\mathcal{X}$ and $\hat{\mathcal{X}}$.
\end{enumerate}
\end{thm}

\begin{proof}
As in the proofs of Theorem \ref{thm:AxA^-general}, view $\mathcal{X}$ and $\hat{\mathcal{X}}$ as subgroups of $PGL(X) \subset PGL(W \otimes X)$. For notational convenience, set $G:= PGL(W \otimes X)$ and $p:=p_{W \otimes X}$.

We start by showing that 
\begin{equation} \label{eq:type-2-goal-1}
Z_{G} ( p( [ H_1 \times_{\mathbb{C}^{\times}} (\mathbb{C}^{\times})^{n}  ] \rtimes_{\mathbb{C}^{\times}} p_X^{-1}(\mathcal{X}) )) = p ( H_2 \times_{\mathbb{C}^{\times}} p_X^{-1} (\hat{\mathcal{X}}) ) = \overline{H}_2 \times \hat{\mathcal{X}} .
\end{equation}
Since $p_X^{-1}(\hat{\mathcal{X}})$ sits inside of $(\mathbb{C}^{\times})^{n}$, we see that 
\begin{align*}
Z_{G} ( p ( [H_1 \times_{\mathbb{C}^{\times}} (\mathbb{C}^{\times})^n ] \rtimes_{\mathbb{C}^{\times}} p_X^{-1}(\mathcal{X})  ) ) & \subseteq Z_{G} ( p ( [ H_1 \times_{\mathbb{C}^{\times}} p_X^{-1}(\hat{\mathcal{X}})  ] \rtimes_{\mathbb{C}^{\times}} p_X^{-1}(\mathcal{X}) )  )\\
&= Z_G(\overline{H}_1 \times \mathcal{X} \times \hat{\mathcal{X}} )\\
&= \overline{H}_2 \times \mathcal{X} \times \hat{\mathcal{X}},
\end{align*}
where we have used Theorem \ref{thm:AxA^-general}. Therefore, for $t \in p^{-1}( Z_{G} ( p( [H_1 \times_{\mathbb{C}^{\times}} (\mathbb{C}^{\times})^{n}] \rtimes_{\mathbb{C}^{\times}} p_X^{-1}(\mathcal{X}) ))$, we can write $t = h \cdot \sigma_{\xi} \cdot \tau_x$ for some $h \in H_2$, $\xi \in \hat{\mathcal{X}}$, and $x \in \mathcal{X}$. But since $t$ commutes (up to a scalar) with any element of $(\mathbb{C}^{\times})^n$, we see that $\tau_x = \text{id}_{W \otimes X}$. It follows that $t \in H_2 \times_{\mathbb{C}^{\times}} p_X^{-1}(\hat{\mathcal{X}})$, and hence that 
$$Z_{G} ( p( [H_1 \times_{\mathbb{C}^{\times}} (\mathbb{C}^{\times})^{n}] \rtimes_{\mathbb{C}^{\times}} p_X^{-1}(\mathcal{X}) ) \subseteq p ( H_2 \times_{\mathbb{C}^{\times}} p_X^{-1} (\hat{\mathcal{X}}) ) = \overline{H}_2 \times \hat{\mathcal{X}}. $$
Inclusion the other way is clear, so we have established (\ref{eq:type-2-goal-1}).

It remains to show that 
$$Z_G( \overline{H}_2 \times \hat{\mathcal{X}} ) = p( [H_1  \times_{\mathbb{C}^{\times}} (\mathbb{C}^{\times})^n ] \rtimes_{\mathbb{C}^{\times}} p_X^{-1}(\mathcal{X}) ) = [\overline{H}_1 \times p((\mathbb{C}^{\times})^n)] \rtimes \mathcal{X}.$$ 
To this end, let $t' \in p^{-1}( Z_G( \overline{H}_2 \times \hat{\mathcal{X}} ) )$. Then, in particular, $t'$ commutes with each $\sigma_{\xi}$ up to a scalar, and hence must be of the form $\tau_x \cdot \diag (d_1,\ldots, d_n)$ for some $x \in \mathcal{X}$ and some $d_1, \ldots, d_n \in GL(W)$ (by the same argument as in the proof of Theorem \ref{thm:AxA^-general}). Moreover, $t'$ commutes (up to a scalar) with $\diag (h,\ldots, h)$ for each $h \in H_2$. Therefore, $d_i \in H_1$ for $1 \leq i \leq n$. Moreover, for each $h \in H_2$, there is a scalar $Z_h \in \mathbb{C}^{\times}$ such that $d_i h d_i^{-1} = Z_h h$ for all $1 \leq i \leq n$. It follows that $d_1 , \ldots , d_n$ are all in the same connected component of $H_1$. Since $H_1^{\circ} = \mathbb{C}^{\times}$, it follows that $\diag (d_1, \ldots , d_n)$ can be written as 
$$ \diag (d,\ldots ,d) \cdot \diag (Z_1, \ldots , Z_n)$$ 
for some $d \in H_1$ and some $Z_1,\ldots , Z_n \in \mathbb{C}^{\times}$. Putting this all together, we see that $t' \in [H_1 \times_{\mathbb{C}^{\times}} (\mathbb{C}^{\times})^n] \rtimes_{\mathbb{C}^{\times}} p_X^{-1}(\mathcal{X})$, and hence that 
$$Z_G(\overline{H}_2 \times \hat{\mathcal{X}}) \subseteq [\overline{H}_1 \times p ( (\mathbb{C}^{\times})^n ) ] \rtimes \mathcal{X}.$$ 
Since inclusion in the other direction is clear, this completes the proof of $(i)$.

The proof of $(ii)$ is nearly identical, with just a couple of small differences. Instead of using that $p_X^{-1}(\hat{\mathcal{X}}) \subseteq (\mathbb{C}^{\times})^n$, we use that $H_1 \times_{\mathbb{C}^{\times}} p_X^{-1}(\hat{\mathcal{X}}) \subseteq H_1^n$ when proving that $Z_G( p(H_1^n) \rtimes \mathcal{X}) = p(H_2) \times \hat{\mathcal{X}}$. Finally, when proving that $Z_G(p(H_2) \times \hat{\mathcal{X}}) = p(H_1^n) \rtimes \mathcal{X}$, we get that $t' \in p^{-1}( Z_G ( p(H_2) \times \hat{\mathcal{X}} ) )$ can be written as $\diag (d_1, \ldots , d_n) \cdot \tau_x$, and that $d_1, \ldots , d_n$ are all in the same connected component of $H_1$. Since $H_1$ is connected (by Theorem \ref{thm:dual-pairs-in-GL}), this completes the proof.  
\end{proof}

We now construct a class of dual pairs that (as we will later see) encompases all of the ``single-orbit" dual pairs (as will be defined in Subsection \ref{subsec:G-irred-reps}).

\begin{thm} \label{thm:single-orbit-general}
Let $\mathcal{L}$, $\mathcal{J}$, and $\mathcal{K}$ be finite abelian groups. Set $L:= L^2 (\mathcal{L},\mathbb{C})$, $J:= L^2(\mathcal{J},\mathbb{C})$, and $K:= L^2 (\mathcal{K},\mathbb{C})$. Let $B$ and $E$ be finite-dimensional complex vector spaces. Then the groups
$$G_1' := \left [ (\mathbb{C}^{\times})^{\vert \mathcal{K} \vert} \times_{\mathbb{C}^{\times}} p_L^{-1}(\mathcal{L} \times \hat{\mathcal{L}}) \times_{\mathbb{C}^{\times}} p_J^{-1}(\hat{\mathcal{J}}) \right ] \rtimes_{\mathbb{C}^{\times}} p_K^{-1}(\mathcal{K})$$
and 
$$G_2' := \left [ (\mathbb{C}^{\times})^{\vert \mathcal{J} \vert} \times_{\mathbb{C}^{\times}} p_L^{-1}(\mathcal{L} \times \hat{\mathcal{L}}) \times_{\mathbb{C}^{\times}} p_K^{-1}(\hat{\mathcal{K}}) \right ] \rtimes_{\mathbb{C}^{\times}} p_J^{-1}(\mathcal{J}) $$
descend to a dual pair in $PGL(L \otimes J \otimes K)$. Additionally, the groups
$$G_1:= [ GL(B)^{\vert \mathcal{K} \vert} \rtimes G_1' ]/( (\mathbb{C}^{\times})^{\vert \mathcal{K} \vert} )_{\antidiag} \hspace{.5cm} \text{ and } \hspace{.5cm} G_2 := [GL(E)^{\vert \mathcal{J} \vert} \rtimes G_2']/( (\mathbb{C}^{\times})^{\vert \mathcal{J} \vert} )_{\antidiag}$$
descend to a dual pair in $PGL(B \otimes E \otimes L \otimes J \otimes K)$. Both dual pairs have component groups 
$$\mathcal{L} \times \hat{\mathcal{L}} \times \hat{\mathcal{J}} \times \mathcal{K} \hspace{.5cm} \text{ and } \hspace{.5cm} \mathcal{L} \times \hat{\mathcal{L}} \times \hat{\mathcal{K}} \times \mathcal{J}.$$ 
\end{thm}

\begin{proof}
By the discussion in the proof of Proposition \ref{prop:AxA^-simple}, $\mathcal{L}$ and $\hat{\mathcal{L}}$ can be viewed as subgroups of $PGL(L)$, where the preimage of $\mathcal{L}$ in $GL(L)$ sits in the set of generalized permutation matrices, and where the preimage of $\hat{\mathcal{L}}$ in $GL(L)$ sits in the set of diagonal matrices. With this realization, Proposition \ref{prop:AxA^-simple} also gives that $(\mathcal{L} \times \hat{\mathcal{L}}, \, \mathcal{L} \times \hat{\mathcal{L}})$ is a dual pair in $PGL(L)$.

Now, view $\mathcal{J}$ and $\hat{\mathcal{J}}$ as subgroups of $PGL(J)$, and view $\mathcal{K}$ and $\hat{\mathcal{K}}$ as subgroups of $PGL(K)$ (in the same way that we did for $\mathcal{L}$ and $\hat{\mathcal{L}}$). Then by Theorem \ref{thm:type-2}, 
$$[ p_L^{-1}(\mathcal{L} \times \hat{\mathcal{L}})  \times_{\mathbb{C}^{\times}} (\mathbb{C}^{\times})^{\vert \mathcal{K} \vert}] \rtimes_{\mathbb{C}^{\times}} p_K^{-1}(\mathcal{K}) \hspace{.5cm} \text{ and } \hspace{.5cm} p_L^{-1}(\mathcal{L} \times \hat{\mathcal{L}}) \times_{\mathbb{C}^{\times}} p_K(\hat{\mathcal{K}})$$
descend to a dual pair in $PGL(L \otimes K)$. Applying Theorem \ref{thm:type-2} again, we get that
$$G_1' := \left [ (   p_L^{-1}(\mathcal{L} \times \hat{\mathcal{L}}) \times_{\mathbb{C}^{\times}} (\mathbb{C}^{\times})^{\vert \mathcal{K} \vert} ) \rtimes_{\mathbb{C}^{\times}} p_K^{-1}(\mathcal{K}) \right ] \times_{\mathbb{C}^{\times}} p_J^{-1}(\hat{\mathcal{J}})$$
and 
$$G_2' := \left [  (p_L^{-1}(\mathcal{L} \times \hat{\mathcal{L}}) \times_{\mathbb{C}^{\times}} p_K^{-1}(\hat{\mathcal{K}})) \times_{\mathbb{C}^{\times}}   (\mathbb{C}^{\times})^{\vert \mathcal{J} \vert} \right ] \rtimes_{\mathbb{C}^{\times}} p_J^{-1}(\mathcal{J}) $$
descend to a dual pair in $PGL(L \otimes J \otimes K )$. Now, define
$$G_1 := [ GL(B)^{\vert \mathcal{K} \vert} \rtimes G_1' ]/ ( (\mathbb{C}^{\times})^{\vert \mathcal{K} \vert} )_{\antidiag} \hspace{.25cm} \text{ and } \hspace{.25cm} G_2:= [GL(E)^{\vert \mathcal{J} \vert} \rtimes G_2' ]/( (\mathbb{C}^{\times})^{\vert \mathcal{J} \vert} )_{\antidiag}$$
so that
$$G_1 = \left [ GL(B)^{\vert \mathcal{K} \vert} \times_{\mathbb{C}^{\times}} p_L^{-1} ( \mathcal{L} \times \hat{\mathcal{L}} ) \times_{\mathbb{C}^{\times}} p_J^{-1}(\hat{\mathcal{J}})  \right ] \rtimes_{\mathbb{C}^{\times}} p_K^{-1}(\mathcal{K}) $$
and
$$G_2 = \left [ GL(E)^{\vert \mathcal{J} \vert} \times_{\mathbb{C}^{\times}} p_L^{-1} ( \mathcal{L} \times \hat{\mathcal{L}} ) \times_{\mathbb{C}^{\times}} p_K^{-1}(\hat{\mathcal{K}})  \right ] \rtimes_{\mathbb{C}^{\times}} p_J^{-1}(\mathcal{J}).$$
Let $U := B \otimes E \otimes L \otimes J \otimes K$. We would like to show that $p^{-1}(Z_{PGL(U)}(p(G_1))) = p^{-1}(G_2)$. To this end, note that 
\begin{enumerate}[label=(\alph*)]
\item $Z_{GL(U)} ( GL(B)^{\vert \mathcal{K} \vert} ) = GL(E \otimes L \otimes J)^{\vert \mathcal{K} \vert}$ (by Theorem \ref{thm:dual-pairs-in-GL}). 
\item $p^{-1} ( Z_{PGL(U)} ( \hat{\mathcal{J}} ) ) = GL(B \otimes E \otimes L \otimes K)^{\vert \mathcal{J} \vert} \rtimes_{\mathbb{C}^{\times}} p_J^{-1} (\mathcal{J}) $ (by Theorem \ref{thm:type-2} $(ii)$).
\item $p^{-1}( Z_{PGL(U)}( \mathcal{L} \times \hat{\mathcal{L}} ) ) = GL(B \otimes E \otimes J \otimes K) \times_{\mathbb{C}^{\times}} p_L^{-1}(\mathcal{L} \times \hat{\mathcal{L}})$ (by Theorem \ref{thm:AxA^-general}).
\item $p^{-1} ( Z_{PGL(U)} (\mathcal{K}) ) = GL(B \otimes E \otimes L \otimes J) \times_{\mathbb{C}^{\times}} [ p_K^{-1}(\hat{\mathcal{K}}) \rtimes_{\mathbb{C}^{\times}} p_K^{-1} ( \mathcal{K} ) ] $ (by irreducibility of the embedding of $p_K^{-1}(\mathcal{K})$ in $GL(K)$ and the proof of Proposition \ref{prop:AxA^-simple}). 
\end{enumerate}
Now, since $p^{-1}(Z_{PGL(U)}(p(G_1)))$ is contained in groups (a), (b), (c), and (d), it is also contained in their intersection:
$$p^{-1}(Z_{PGL(U)} (p(G_1))) \subseteq [ GL(E)^{\vert \mathcal{J} \vert} \times_{\mathbb{C}^{\times}} p_L^{-1} ( \mathcal{L} \times \hat{\mathcal{L}} ) \times_{\mathbb{C}^{\times}} p_K^{-1}(\hat{\mathcal{K}}) ] \rtimes_{\mathbb{C}^{\times}} p_J^{-1}(\mathcal{J}) = p^{-1}(G_2).$$ 
Since containment in the other direction is clear, this shows that $p^{-1}(Z_{PGL(U)}(p(G_1))) = p^{-1}(G_2)$. The proof is completed by swapping the roles of $G_1$ and $G_2$ and repeating this argument. 
\end{proof}

Notice that the component groups of the resulting dual pairs are 
$$G_1/G_1^{\circ} = G_1'/(G_1')^{\circ} = \mathcal{L} \times \hat{\mathcal{L}} \times \hat{\mathcal{J}} \times \mathcal{K} \hspace{.25cm} \text{ and }\hspace{.25cm} G_2/G_2^{\circ} = G_2'/(G_2')^{\circ} = \mathcal{L} \times \hat{\mathcal{L}} \times \mathcal{J} \times \hat{\mathcal{K}} .$$
We will soon see that the component groups of the two members of a dual pair in $PGL(U)$ are always dual finite abelian groups (Theorem \ref{thm:dual-finite-ab-gps}).

\section{Preliminaries for showing we've found all ``single-orbit" dual pairs}  \label{sec:prelims-for-showing-all}

As previously mentioned, the connected dual pairs in $PGL(U)$ are in bijection with the dual pairs in $GL(U)$. Additionally, Theorem \ref{thm:single-orbit-general} describes a general class of ``single-orbit" dual pairs in $PGL(U)$ (where a connected dual pair is obtained by taking $\mathcal{J}$, $\mathcal{K}$, and $\mathcal{L}$ to be trivial). Sections \ref{sec:prelims-for-showing-all}--\ref{sec:explicit} are dedicated to proving that Theorem \ref{thm:single-orbit-general} accounts for all of the ``single-orbit" dual pairs in $PGL(U)$ (where ``single-orbit" will be defined in Subsection \ref{subsec:G-irred-reps}). Note that a lot of notation is introduced in these sections; an Index of Notation can be found at the end of the paper.

For the remainder of the paper, we will be considering an arbitrary dual pair $(\overline{G}, \overline{H})$ in $PGL(U)$ and will be trying to understand the possible preimages $G := p^{-1}(\overline{G})$ and $H := p^{-1}(\overline{H})$ in $GL(U)$. These preimages appear in the short exact sequences
\begin{equation} \label{eq:short-exact-sequence}
    1 \rightarrow G^{\circ} \rightarrow G \xrightarrow{\pi_G} \Gamma \rightarrow 1 \hspace{.5cm} \text{ and } \hspace{.5cm}
    1 \rightarrow H^{\circ} \rightarrow H \xrightarrow{\pi_H} \Delta \rightarrow 1,
\end{equation}
where $\Gamma := G/G^{\circ}$ and $\Delta := H / H^{\circ}$. Since $GL(U)$ is a $\mathbb{C}^{\times}$-bundle over $PGL(U)$, we have that $G$ (resp.~$H$) is a $\mathbb{C}^{\times}$-bundle over $\overline{G}$ (resp.~$\overline{H}$). It follows that $\Gamma \simeq \overline{G}/\overline{G}^{\circ}$ and that $\Delta \simeq \overline{H}/\overline{H}^{\circ}
$.

By the mutual centralizer relation and the following lemma, we have a map
\begin{align}
    \mu \colon G \times H & \rightarrow \{ (\dim U)\text{-th roots of unity} \} \label{al:mu-G-H} \\ \nonumber 
    (g,h) & \mapsto ghg^{-1}h^{-1}.
\end{align}

\begin{lemma} \label{lem:roots-of-unity}
Suppose $xyx^{-1} = cy$, where $x,y \in GL(U)$ and $c \in \mathbb{C}^{\times}$. Then $c$ is a $(\dim U)$-th root of unity. 
\end{lemma}

\begin{proof}
Using that the determinant is multiplicative, we see that $\det (xyx^{-1}) = \det (y)$. Therefore, the relation $xyx^{-1} = cy$ gives that 
$$\det (y) = \det (cy) = c^{\dim U} \det (y).$$
This shows that $c^{\dim U} = 1$, or that $c$ is a $(\dim U)$-th root of unity (not necessarily primitive).
\end{proof}

Moreover, it is not hard to check from the defining equation that $\mu$ is actually a group homomorphism for both $G$ and $H$ (i.e.,~that $\mu(g g', h) = \mu(g,h)\mu(g',h)$ and that $\mu(g,h h') = \mu(g,h)\mu(g,h')$ for all $g,g' \in G$ and all $h,h' \in H$).

Note that for each connected component of $G \times H$, $\mu$ restricts to a continuous map from a connected set to a discrete set, and hence must be constant. Therefore, $\mu$ descends to a map 
\begin{equation} \label{eq:mu-gamma-delta}
\mu \colon \Gamma \times \Delta \rightarrow \{ (\dim U)\text{-th roots of unity} \},
\end{equation}
where $\mu (\gamma , \delta)$ can be calculated as $\mu (g_{\gamma}, h_{\delta})$ for any $g_{\gamma} \in \pi_G^{-1}(\gamma)$ and $h_{\delta} \in \pi_H^{-1}(\delta)$. Using that $\mu$ is a group homomorphism for both $G$ and $H$, it is straightforward to check that $\mu$ is a group homomorphism for both $\Gamma$ and $\Delta$ (i.e.,~that $\mu (\gamma \gamma', \delta ) = \mu (\gamma , \delta) \mu (\gamma' , \delta)$ and that $\mu (\gamma ,\delta \delta ' ) = \mu (\gamma , \delta) \mu (\gamma , \delta ')$ for all $\gamma , \gamma' \in \Gamma$ and all $\delta , \delta' \in \Delta$).

\begin{thm} \label{thm:dual-finite-ab-gps}
$\Gamma$ and $\Delta$ are finite abelian groups with natural isomorphisms 
$$\Delta \simeq \hat{\Gamma} \hspace{.5cm} \text{ and } \hspace{.5cm} \Gamma \simeq \hat{\Delta}.$$
\end{thm}

\begin{proof}
Recall that for any algebraic group $G$, the centralizer in $G$ of any subset of $G$ is algebraic. As a consequence, we have that any member of a dual pair is algebraic, and hence that $\Gamma$ and $\Delta$ are finite. 

To see that $\Gamma$ and $\Delta$ are dual and abelian, consider the group homomorphism
\begin{align*}
    M \colon \Gamma & \rightarrow \hat{\Delta} \\
    \gamma & \mapsto \mu(\gamma, \cdot ).
\end{align*}
We will show that $M$ is an isomorphism. Towards proving injectivity, define $\Gamma' := \ker M$. Then $\Gamma '$ corresponds to the following subgroup of $G$:
$$G' := \{ g \in G \, : \, ghg^{-1} = h \text{ for all } h \in H \}.$$
Recall that 
$$G = \{ g \in GL(U) \, : \, ghg^{-1}h^{-1} \in \mathbb{C}^{\times} \text{ for all } h \in H \}.$$
Therefore, we see that
$$G ' = \{ g \in GL(U) \, : \, ghg^{-1} = h \text{ for all } h \in H \} = Z_{GL(U)}(H).$$
Since $G'$ is a centralizer in $GL(U)$, it is part of a dual pair $(G', Z_{GL(U)}(G'))$ (by Fact \ref{fact:triple centralizer}), and is therefore connected (by Theorem \ref{thm:dual-pairs-in-GL}). It follows that $\Gamma '$ is trivial, and hence that $M$ is injective. Since $\hat{\Delta}$ is abelian, this gives that $\Gamma$ is abelian as well. Switching the roles of $\Gamma$ and $\Delta$ in this argument gives that 
\begin{align*}
    M' \colon \Delta & \rightarrow \hat{\Gamma} \\
    \delta & \mapsto \mu(\cdot, \delta)
\end{align*}
is injective, and that $\Delta$ is abelian. Finally, since the dual group functor is a contravariant exact functor for locally compact abelian groups, the injectivity of $M'$ implies the surjectivity of $M$, completing the proof. 
\end{proof}

Moving forward, we will identify $\Delta$ with $\hat{\Gamma}$ and will no longer use ``$\Delta$." As a particular consequence of Theorem \ref{thm:dual-finite-ab-gps}, we obtain information regarding the identity components of preimages of members of $PGL(U)$ dual pairs. In particular, these identity components necessarily equal products of general linear groups.

\begin{cor} \label{cor:identity-components-centralizers}
We have that
$$G^{\circ} = Z_{GL(U)} (H) = \prod_{i} GL(B_i) \hspace{.25cm} \text{ and } \hspace{.25cm} H^{\circ} = Z_{GL(U)} (G) = \prod_{j} GL(E_j),$$
where the $B_i$'s and $E_j$'s are finite-dimensional complex vector spaces satisfying 
$$U \simeq \bigoplus_{i} B_i \otimes A_i \simeq \bigoplus_j E_j \otimes F_j$$ 
for some sets of finite-dimensional complex vector spaces $\{ A_i \}$ and $\{ F_j \}$.
\end{cor}

\begin{proof}
The equalities $G^{\circ} = Z_{GL(U)} (H)$ and $H^{\circ} = Z_{GL(U)} (G)$ follow from the proof of Theorem \ref{thm:dual-finite-ab-gps}. Then by Fact \ref{fact:triple centralizer} and Theorem \ref{thm:dual-pairs-in-GL}, we see that $Z_{GL(U)} (H) = \prod_{i} GL(B_i)$ and that $Z_{GL(U)} (H) = \prod_{j} GL(E_j)$ for some $\{ B_i \}$ and $\{ E_j \}$ as described in the lemma statement.
\end{proof}

We now require a result about the algebraic extensions of reductive groups. Let $G_0$ be a complex reductive algebraic group, and let $\Gamma$ be a finite group. Let $T_0 \subseteq B_0 \subseteq G_0$ be a maximal torus and Borel subgroup of $G_0$, and let $\{ X_{\alpha} \, : \, \alpha \in \Pi (B_0, T_0) \}$ be a choice of root vectors for the corresponding simple root spaces. For an extension $E$ of $\Gamma$ by $G_0$, define 
$$E(\{ T_0, X_{\alpha} \}) := \{ e \in E \; \vert \; \text{Ad}(e)(\{ T_0, X_{\alpha} \}) = \{ T_0, X_{\alpha} \} \},$$ 
where by $\text{Ad}(e) (\{ T_0, X_{\alpha} \}) = \{ T_0, X_{\alpha} \}$, we mean that $\text{Ad}(e)(T_0) = T_0$ and that $\text{Ad}(e)$ preserves the set $\{ X_{\alpha} \}$ (not necessarily pointwise). This is known as the subgroup of $E$ defining \textit{distinguished automorphisms} of $G_0$. 

\begin{prop}[G.--Vogan {\cite[Corollary 3.2 and Theorem 3.3]{Disconnected}}] \label{prop:bijection-extensions}
\hspace{1cm}
\begin{enumerate}[label=(\arabic*)] 
\item The group $E(\{ T_0, X_{\alpha} \})$ is an extension of $\Gamma$ by the abelian group $Z(G_0)$:
$$1 \rightarrow Z(G_0) \rightarrow E(\{ T_0, X_{\alpha} \}) \rightarrow \Gamma \rightarrow 1.$$
\item There is a natural surjective homomorphism
$$G_0 \rtimes E(\{ T_0, X_{\alpha} \}) \rightarrow E$$
with kernel the antidiagonal copy $Z(G_0)_{\antidiag}=\{ (z,z^{-1}) \; \vert \; z \in Z(G_0) \}$.
\item There is a natural bijection between algebraic extensions $E$ of $\Gamma$ by $G_0$ and algebraic extensions $E(\{ T_0, X_{\alpha} \})$ of $\Gamma$ by $Z(G_0)$. These latter are parameterized by the group cohomology $H^2 (\Gamma, Z(G_0))$, and the bijection is given by
$$E = [G_0 \rtimes E(\{ T_0, X_{\alpha} \})]/Z(G_0)_{\antidiag}.$$ 
\end{enumerate}
\end{prop}

Now, by Corollary \ref{cor:identity-components-centralizers} we have that $G^{\circ}$ and $H^{\circ}$ are both products of $GL$'s. We can therefore let $T_G$ (resp.~$T_H$) be the standard maximal torus of $G^{\circ}$ (resp.~$H^{\circ}$) composed of diagonal matrices. Similarly, we can let $\{ X_{\alpha} \}$ (resp.~$\{ Y_{\beta} \}$) be the standard choice of basis vectors for the appropriate simple root spaces (see \cite[Plate I]{bourbaki}). Set $G_d:= G(\{ T_G, X_{\alpha} \})$ and $H_d := H(\{ T_H, Y_{\beta} \})$. Then by Proposition \ref{prop:bijection-extensions}, we get that there exist lifts $\wtilde{(\cdot)} \colon \Gamma \rightarrow G_d$ and $\wtilde{(\cdot)} \colon \hat{\Gamma} \rightarrow H_d$ along with corresponding 2-cocycles 
\begin{equation} \label{eq:s_G-and-s_H}
s_G \colon \Gamma \times \Gamma \rightarrow Z(G^{\circ}) \hspace{.5cm} \text{ and } \hspace{.5cm} s_H \colon \hat{\Gamma} \times \hat{\Gamma} \rightarrow Z(H^{\circ})
\end{equation}
defining our extensions
\begin{equation} \label{eq:pairs-of-extensions}
1 \rightarrow G^{\circ} \rightarrow G \xrightarrow{\pi_G} \Gamma \rightarrow 1 \hspace{.5cm} \text{ and } \hspace{.5cm} 1 \rightarrow H^{\circ} \rightarrow H \xrightarrow{\pi_H} \hat{\Gamma} \rightarrow 1.
\end{equation}
Since $G^{\circ}$ and $H^{\circ}$ are both products of $GL$'s, we see that $Z(G^{\circ})$ and $Z(H^{\circ})$ are both products of copies of $\mathbb{C}^{\times}$.

In this way, we see that the problem of classifying dual pairs in $PGL(U)$ can be reframed as the problem of classifying the pairs of extensions \eqref{eq:pairs-of-extensions}, where we already have a good understand of the identity components, the relationship between the component groups, and the defining lifts and cocycles. In the next section, we further our understanding of these extensions by defining actions of $\hat{\Gamma}$ (resp.~$\Gamma$) on the $G$- and $H^{\circ}$-irreducibles (resp.~$G^{\circ}$- and $H$-irreducibles) appearing in $U$.

\section{Orbits of $G$-, $G^{\circ}$-, $H$-, and $H^{\circ}$-irreducible representations} \label{sec:orbits}

The goal of this section is to define $\hat{\Gamma}$-actions on the $G$- and $H^{\circ}$-irreducibles in $U$, as well as $\Gamma$-actions on the $G^{\circ}$- and $H$-irreducibles in $U$, so that the embeddings of $G$, $H^{\circ}$, $G^{\circ}$, and $H$ can be expressed nicely in terms of irreducibles ``twisted" by these actions.

\subsection{Choosing $\Gamma$-representatives (resp.~$\hat{\Gamma}$-representatives) in $G$ (resp.~$H$)} \label{subsec:choose-Gamma-reps}

Recall that by Proposition \ref{prop:bijection-extensions}, there exist lifts $\wtilde{(\cdot)} \colon \Gamma \rightarrow G_d$ and $\wtilde{(\cdot)} \colon \hat{\Gamma} \rightarrow H_d$ along with 2-cocycles $s_G \colon \Gamma \times \Gamma \rightarrow Z(G^{\circ})$ and $s_H \colon \hat{\Gamma} \times \hat{\Gamma} \rightarrow Z(H^{\circ})$ defining our extensions
$$1 \rightarrow G^{\circ} \rightarrow G \xrightarrow{\pi_G} \Gamma \rightarrow 1 \hspace{.5cm} \text{ and } \hspace{.5cm} 1 \rightarrow H^{\circ} \rightarrow H \xrightarrow{\pi_H} \hat{\Gamma} \rightarrow 1.$$
Then we have
\begin{equation} \label{eq:def-tilde-lift}
\wtilde{\gamma} \, \wtilde{\gamma'} = s_G( \gamma , \gamma' ) \wtilde{\gamma \gamma'} \hspace{.5cm} \text{ and } \hspace{.5cm} \wtilde{\delta} \, \wtilde{\delta'} = s_H(\delta , \delta ' ) \wtilde{\delta \delta '}
\end{equation}
for all $\gamma, \gamma' \in \Gamma$ and $\delta , \delta ' \in \hat{\Gamma}$. Moreover, modifying our lifts $s_G$ and $s_H$ if necessary, we can assume that $\wtilde{1}_{\Gamma} = \text{id}_G$ and that $\wtilde{1}_{\hat{\Gamma}} = \text{id}_H$. (Note that we are using the same notation $\wtilde{(\cdot)}$ for two different lifts, but it should always be clear from context which lift we are referring to.)

\subsection{Defining a $\hat{\Gamma}$-action on the $G$-irreducibles} \label{subsec:hat-Gamma-action-on-G-irreps}

For $h \in H$ and $(\varphi_j, F_j)$ a $G$-irreducible in $U$, define 
\begin{equation} \label{eq:phi_i-twisted-by-h}
\prescript{h}{}{\varphi_j} (g) := \varphi_j ( h g h^{-1} ) = \varphi_j (\mu (g,h)^{-1} g) = (\varphi_j (\cdot) \otimes \mu (\cdot , h)^{-1}) (g).
\end{equation}
Recall that for $\delta \in \hat{\Gamma}$, we have $\mu (\cdot , h_{\delta}) = \mu (\cdot , h_{\delta}')$ for any $h_{\delta}, h_{\delta}' \in \pi_H^{-1}(\delta)$. Therefore, this $H$-action descends to a $\hat{\Gamma}$-action on the $G$-irreducibles: for $\delta \in \hat{\Gamma}$, define 
\begin{equation} \label{eq:phi-twisted-by-delta}
\prescript{\delta}{}{\varphi_j} (\cdot ) := \varphi_j (\cdot ) \otimes \mu (\cdot, \delta)^{-1}.
\end{equation}
Note that since $\prescript{h_{\delta}}{}{\varphi_j}$ and $\prescript{h_{\delta}'}{}{\varphi_j}$ are not only equivalent but \textit{equal} representations of $G$ for $h_{\delta}, h_{\delta}' \in \pi_H^{-1}(\hat{\Gamma})$, $\prescript{\delta}{}{\varphi_j}$ is a well-defined representation of $G$. Since $\varphi_j$ is irreducible and $\mu (\cdot , \delta)^{-1}$ is a one-dimensional (irreducible) representation of $G$, we indeed have that $\prescript{\delta}{}{\varphi_j}$ is irreducible. Write $(\prescript{\delta}{}{\varphi_j}, \prescript{\delta}{}{F_j})$ for this twisted representation; here, the space $\prescript{\delta}{}{F_j}$ is the same vector space as $F_j$ but with a different $G$-action. We can check that this indeed defines a $\hat{\Gamma}$-action:
\begin{align*}
    \prescript{\delta}{}{( \prescript{\delta'}{}{\varphi_j} (g) )} &= \mu(g,\delta)^{-1} \prescript{\delta'}{}{\varphi_j}(g) = \mu (g, \delta)^{-1} \mu (g, \delta')^{-1} \varphi_j (g) \\
    &= \mu (g, \delta \delta')^{-1} \varphi_j (g) = \prescript{(\delta \delta')}{}{ \varphi_j} (g).
\end{align*}

For a $G$-irreducible $\varphi_j$ in $U$, set $I_H(\varphi_j) := \{ h \in H \, : \, \prescript{h}{}{\varphi_j} \simeq \varphi_j \}$. Note that
$$H/I_H(\varphi_j) \simeq (H/H^{\circ})/(I_H(\varphi_j)/H^{\circ}) \simeq \hat{\Gamma}/(I_H(\varphi_j)/H^{\circ}),$$
where $I_H(\varphi_j)/H^{\circ} \simeq \{ \delta \in \hat{\Gamma} \, : \, \prescript{\delta}{}{\varphi_j} \simeq \varphi_j \}$. (So $H/I_H(\varphi_j)$ is the quotient of $\hat{\Gamma}$ by the subgroup of $\hat{\Gamma}$ that sends $\varphi_j$ to an equivalent representation.)

\subsection{Choosing representatives for the $G$-irreducibles} \label{subsec:G-irred-reps}

Now, fix a $G$-irreducible $(\varphi, F)$ in $U$ and set $\mathcal{J}:= H/I_H(\varphi)$, written as an additive group. For each $j \in \mathcal{J}$, choose a corresponding coset representative $\delta_j \in \hat{\Gamma}$. For $j = 0$, choose $\delta_{0} = 1$. Then, by construction, $\delta_j \delta_{j'}$ and $\delta_{j + j'}$ both have image $j+j'$ in $\mathcal{J}$. \textbf{Assume (for now) that the set $\{ (\prescript{\delta_j}{}{\varphi}, \prescript{\delta_j}{}{F}) \}_{ j \in \mathcal{J} }$ is a complete set of representatives of the equivalence classes of $G$-irreducibles appearing in $U$.} We call a dual pair that satisfies this assumption a \textit{single-orbit dual pair}. We make this assumption now largely to avoid excessively complicated notation in this section. In Section \ref{sec:single-orbit}, we will explain how to reduce the general case to the single orbit case, and we will see that many of the results in this section easily extend to the general case. 

With the single orbit assumption, we have a decomposition 
$$U = \bigoplus_{ j \in \mathcal{J} } \prescript{\delta_j}{}{F} \otimes \Hom_{G} ( \prescript{\delta_j}{}{F}, U ),$$
where the embedding $G \hookrightarrow GL(U)$ realizes $G$ as a subgroup of $\prod_{j \in \mathcal{J}} GL(\prescript{\delta_j}{}{F})$.

\subsection{Defining a $\Gamma$-action on the $G^{\circ}$-irreducibles} \label{subsec:Gamma-action-on-G-circ}

For $g \in G$ and $(\beta_i , B_i)$ a $G^{\circ}$-irreducible in $U$,
define 
\begin{equation} \label{eq:beta-twisted-by-g}
\prescript{g}{}{\beta_i}(g_0) := \beta_i ( gg_0g^{-1} ).
\end{equation}
We see that each $\prescript{g}{}{\beta_i}$ is irreducible. Note, however, that this does not descend to a well-defined $\Gamma$-action on $G^{\circ}$-irreducibles: For $g_{\gamma}, g_{\gamma}' \in \pi_G^{-1}(\gamma)$, we have that $\prescript{g_{\gamma}'}{}{\beta_i} \simeq \prescript{g_{\gamma}}{}{\beta_i}$, but not necessarily that $\prescript{g_{\gamma}'}{}{\beta_i} = \prescript{g_{\gamma}}{}{\beta_i}$ (and this will be an important distinction throughout the paper). To get a well-defined $\Gamma$-action, we make use of our lift $\wtilde{(\cdot)} \colon \Gamma \rightarrow G_d$, letting $\gamma \in \Gamma$ act as
\begin{equation*} 
\prescript{\wtilde{\gamma}}{}{\beta_i}(g_0) = \beta_i (\wtilde{\gamma} g_0 (\wtilde{\gamma})^{-1}).
\end{equation*}

Let $\prescript{\wtilde{\gamma}}{}{B_i}$ denote the corresponding representation space; note that $\prescript{\wtilde{\gamma}}{}{B_i}$ is the same vector space as $B_i$ but with a different $G^{\circ}$-action. We can check that this indeed defines a $\Gamma$-action:
\begin{align*}
    \prescript{\wtilde{\gamma}}{}{( \prescript{\wtilde{\gamma'}}{}{\beta_i} (g_0) )} &= \prescript{\wtilde{\gamma'}}{}{\beta_i}( \wtilde{\gamma} g_0 \wtilde{\gamma}^{-1} ) = \beta_i (\wtilde{\gamma'} \, \wtilde{\gamma} g_0 \wtilde{\gamma}^{-1} \wtilde{\gamma'}^{-1}) \\
    &= \beta_i ( s_G(\gamma', \gamma) \wtilde{\gamma' \gamma} g_0 \wtilde{\gamma' \gamma}^{-1} s_G(\gamma', \gamma)^{-1} ) \\
    &= \beta_i ( \wtilde{\gamma \gamma'} g_0 \wtilde{\gamma \gamma'}^{-1} ) = \prescript{\wtilde{\gamma \gamma'}}{}{\beta_i(g_0)}.
\end{align*}
For a $G^{\circ}$-irreducible $\beta_i$ in $U$, set $I_G(\beta_i) := \{ g \in G \, : \, \prescript{g}{}{\beta_i} \simeq \beta_i \}$. Note that 
$$G/I_G(\beta_i) \simeq (G/G^{\circ}) / (I_G(\beta_i)/G^{\circ}) \simeq \Gamma / ( I_G(\beta_i)/G^{\circ} ),$$
where $I_G(\beta_i)/G^{\circ} \simeq \{ \gamma \in \Gamma \, : \, \prescript{\wtilde{\gamma}}{}{\beta_i} \simeq \beta_i \}$. (So $G/I_G(\beta_i)$ is the quotient of $\Gamma$ by the subgroup of $\Gamma$ that sends $\beta_i$ to an equivalent representation.) 

Now, let $\{ (\beta_i , B_i) \}_{1 \leq i \leq r}$ be the set of $G^{\circ}$-irreducibles in $U$ such that
$$\Res^G_{G^{\circ}} (F) = \bigoplus_{1 \leq i \leq r} B_i \otimes \Hom_{G^{\circ}} (B_i , F).$$
Set $(\beta , B) := (\beta_1, B_1)$ and $\mathcal{K}:= G/I_G(\beta)$ (where $\mathcal{K}$ will be written as an additive group). For each $k \in \mathcal{K}$, choose a corresponding coset representative $\gamma_k \in \Gamma$. For $k=0$, choose $\gamma_0 = 1$. Then, by construction, $\gamma_k \gamma_{k'}$ and $\gamma_{k+k'}$ both have image $k+k'$ in $\mathcal{K}$.

\subsection{Identifying the embedding of $G^{\circ}$ with $\bigoplus_{k \in \mathcal{K}} \prescript{\wtilde{\gamma_k}}{}{\beta} \otimes \Hom_{G^{\circ}} ( \prescript{\wtilde{\gamma_k}}{}{B}, U )$} \label{subsec:identify-G-circ-irreps}

By our choice of $(\beta , B)$, \cite[Theorem 2.1]{Clifford} gives that we have an equivalence of representations 
$$d \colon \bigoplus_{1 \leq i \leq r} B_i \otimes \Hom_{G^{\circ}} (B_i , F) \xrightarrow{\sim} \bigoplus_{k \in \mathcal{K}} \prescript{\wtilde{\gamma_k}}{}{B} \otimes \Hom_{G^{\circ}} (\prescript{\wtilde{\gamma_k}}{}{B}, F).$$
Recalling that $\prescript{\delta_j}{}{\varphi}(\cdot) = \varphi (\cdot) \otimes \mu (\cdot , \delta_j)^{-1}$ (where $\mu (\cdot , \delta_j)^{-1}$ is trivial on $G^{\circ}$), we see that $\Res^G_{G^{\circ}} (\prescript{\delta_j}{}{F}) = \Res^G_{G^{\circ}} (F)$, that $\Hom_{G^{\circ}} (B_i, \prescript{\delta_j}{}{F}) = \Hom_{G^{\circ}} (B_i, F)$, and that $\Hom_{G^{\circ}} ( \prescript{\wtilde{\gamma_k}}{}{B}, \prescript{\delta_j}{}{F} ) = \Hom_{G^{\circ}} (\prescript{\wtilde{\gamma_k}}{}{B}, F)$ for all $j \in \mathcal{J}$. Noting that 
$$U = \bigoplus_{j \in \mathcal{J}} \left ( \bigoplus_{1 \leq i \leq r} B_i \otimes \Hom_{G^{\circ}} (B_i,F) \right ) \otimes \Hom_G(\prescript{\delta_j}{}{F},U)$$
and defining 
$$U' := \bigoplus_{j \in \mathcal{J}} \left ( \bigoplus_{k \in \mathcal{K}} \prescript{\wtilde{\gamma_k}}{}{B} \otimes \Hom_{G^{\circ}} (\prescript{\wtilde{\gamma_k}}{}{B}, F) \right ) \otimes \Hom_G(\prescript{\delta_j}{}{F},U),$$
we see that $D:= \diag (d,\ldots , d)$ (with $d$ appearing $\vert \mathcal{J} \vert$ times) gives an isomorphism $U \xrightarrow{\sim} U'$. Under this isomorphism, each $x \in GL(U)$ gets sent to $\sigma (x) := DxD^{-1} \in GL(U')$. In particular, $G^{\circ} = \prod_{1 \leq i \leq r} GL(B_i) \subseteq GL(U)$ gets sent to $\sigma (G^{\circ}) = \prod_{k \in \mathcal{K}} GL(\prescript{\wtilde{\gamma_k}}{}{B}) \subseteq GL(U')$. 

Set $F':= \bigoplus_{k \in \mathcal{K}} \prescript{\wtilde{\gamma_k}}{}{B} \otimes \Hom_{G^{\circ}} (\prescript{\wtilde{\gamma_k}}{}{B}, F)$, and define $\varphi' \colon \sigma(G) \rightarrow GL(F')$ by $\varphi'( \sigma (g) ) = d \varphi (g) d^{-1}$. Then for $\delta \in \hat{\Gamma}$,  
$$\prescript{\sigma (\wtilde{\delta})}{}{\varphi '} ( \sigma (g) ) = \varphi' ( \sigma (\wtilde{\delta}) \sigma (g) \sigma (\wtilde{\delta})^{-1} ) = \varphi ' ( \sigma ( \wtilde{\delta} g \wtilde{\delta}^{-1} ) ) = d \prescript{\delta}{}{\varphi} (g) d^{-1}.$$ 
In this way, we see that $\{ (\prescript{\sigma (\wtilde{\delta_j})}{}{\varphi'}, \prescript{\sigma (\wtilde{\delta_j})}{}{F'} ) \}_{j \in \mathcal{J}}$ is a full set of representatives for the equivalence classes of $\sigma (G)$-irreducibles appearing in $U'$. By construction, we have that 
$$\Hom_{\sigma (G)} ( \prescript{\sigma (\delta_j)}{}{F'}, U' ) = \Hom_{G} (\prescript{\delta_j}{}{F}, U).$$ 
Similarly, we see that $\{ ( \prescript{\sigma (\wtilde{\gamma_k}) }{}{\beta} , \prescript{\sigma ( \wtilde{\gamma_k} )}{}{B} ) \}_{k \in \mathcal{K}}$ is a full set of representatives for the equivalence classes of $\sigma (G^{\circ})$-irreducibles appearing in $U'$. Additionally, we can define $\wtilde{(\cdot)}^{\,'} \colon \Gamma \rightarrow \sigma (G_d)$ via $\wtilde{\gamma}^{\,'} := \sigma (\wtilde{\gamma})$ and can define $\wtilde{(\cdot)}^{\,'} \colon \hat{\Gamma} \rightarrow \sigma (H_d)$ via $\wtilde{\delta}^{\,'} := \sigma (\wtilde{\delta})$.

Finally, we can replace $U$ with $U'$, $H$ with $\sigma (H)$, $G$ with $\sigma (G)$, $(\varphi,F)$ with $(\varphi',F')$, $\wtilde{(\cdot)} \colon \Gamma \rightarrow G_d$ with $\wtilde{(\cdot)}^{\,'} \colon \Gamma \rightarrow \sigma(G_d)$, and $\wtilde{(\cdot)} \colon \hat{\Gamma} \rightarrow H_d$ with $\wtilde{(\cdot)}^{\,'} \colon \hat{\Gamma} \rightarrow \sigma (H_d)$ moving forward so that
\begin{itemize}
\item $\bigoplus_{k \in \mathcal{K}} \prescript{\wtilde{\gamma_k}}{}{\beta} \otimes \Hom_{G^{\circ}} (\prescript{\wtilde{\gamma_k}}{}{B},U)$ is the embedding of $G^{\circ}$ in $U$, and
\item $\bigoplus_{j \in \mathcal{J}} \prescript{\delta_j}{}{\varphi} \otimes \Hom_G (\prescript{\delta_j}{}{F}, U)$ is the embedding of $G$ in $U$.
\end{itemize}
Note that we can make these replacements because we are interested in classifying the \textit{conjugacy classes} of dual pairs in $PGL(U)$.

\subsection{Defining the corresponding $H^{\circ}$-irreducibles} \label{subsec:H0-irreps}

For $\delta \in \hat{\Gamma}$, define 
\begin{equation} \label{eq:E_delta}
E_{\delta} := \Hom_{G} (\prescript{\delta}{}{F}, U) = \{ e \colon F \rightarrow U \; \vert \; e( \prescript{\delta}{}{\varphi} (g) \cdot f) = g \cdot e(f) \; \forall f \in F, \, g \in G \}.
\end{equation}
Then
$$U = \bigoplus_{j \in \mathcal{J}} \prescript{\delta_j}{}{F} \otimes E_{\delta_j} \hspace{.25cm} \text{ and } \hspace{.25cm} H^{\circ} = Z_{GL(U)} ( G ) = \prod_{ j \in \mathcal{J} } GL(E_{\delta_j}).$$
For $j \in \mathcal{J}$, define $\varepsilon_{\delta_j} \colon H^{\circ} \rightarrow GL(E_{\delta_j})$ as $h_0 \in H^{\circ} \mapsto (h_0)\vert_{E_{\delta_j}}$, so that the set $\{ (\varepsilon_{\delta_j}, E_{\delta_j}) \}_{j \in \mathcal{J} }$ is a complete set of representatives of the equivalence classes of $H^{\circ}$-irreducibles appearing in $U$. For notational convenience, set $E:=E_1=E_{\delta_0}$. Similarly, set $\varepsilon := \varepsilon_1 = \varepsilon_{\delta_0}$.

\subsection{Defining the corresponding $H$-irreducibles} \label{subsec:define-H-irreps}

For $\gamma \in \Gamma$, define
\begin{equation} \label{eq:A_gamma}
A_{\wtilde{\gamma}} := \Hom_{G^{\circ}} ( \prescript{\wtilde{\gamma}}{}{B}, U ) = \{ a \colon B \rightarrow U \; \vert \; a(\prescript{\wtilde{\gamma}}{}{\beta}(g_0) \cdot b ) = g_0 \cdot a(b) \; \forall b \in B, \, g_0 \in G^{\circ} \}.
\end{equation}
Then $U = \bigoplus_{k \in \mathcal{K}} \prescript{\wtilde{\gamma_k}}{}{B} \otimes A_{\wtilde{\gamma_k}}$ and the embedding $H \hookrightarrow GL(U)$ realizes $H$ as a subgroup of $\prod_{k \in \mathcal{K}} GL(A_{\wtilde{\gamma_k}})$ (since $H \subseteq Z_{GL(U)} (G^{\circ}) = \prod_{k \in \mathcal{K}} GL(A_{\wtilde{\gamma_k}})$). For $k \in \mathcal{K}$, define $\alpha_{\wtilde{\gamma_k}} \colon H \rightarrow GL(A_{\wtilde{\gamma_k}})$ as $h \in H \mapsto (h)\vert_{A_{\wtilde{\gamma_k}}}$. We claim that the $A_{\wtilde{\gamma_k}}$'s are irreducible and distinct. To see this, let $\{ X_{\ell} \}$ be a complete set of representatives of the equivalence classes of $H$-irreducibles in $U$. Then
\begin{align*}
    Z_{GL(U)} (H) & \simeq \prod_{\ell} GL( \Hom_H ( X_{\ell}, U ) )  \simeq \prod_{\ell} GL \left ( \bigoplus_{k \in \mathcal{K}} \prescript{\wtilde{\gamma_k}}{}{B} \otimes \Hom_H (X_{\ell}, A_{ \wtilde{\gamma_k} }) \right ),
\end{align*}
and therefore
\begin{equation} \label{eq:H-centralizer-dimension-1}
    \dim Z_{GL(U)} (H) = \sum_{\ell} \left [  \sum_{k \in \mathcal{K}} (\dim B)^2 \cdot (\dim \Hom_H (X_{\ell} , A_{\wtilde{\gamma_k}} ) )^2 \right ] .
\end{equation}
However, we know that $Z_{GL(U)}(H) = G^{\circ} = GL(B)^{\vert \mathcal{K}} \vert$, which gives that
\begin{equation} \label{eq:H-centralizer-dimension-2}
    \dim Z_{GL(U)} (H) = \vert \mathcal{K} \vert \cdot (\dim B)^2.
\end{equation}
It is clear that the expressions in \eqref{eq:H-centralizer-dimension-1} and \eqref{eq:H-centralizer-dimension-2} are equal exactly when each $A_{\wtilde{\gamma_k}}$ is equivalent to a distinct $X_{\ell}$. In this way, we see that the set $\{ ( \alpha_{\wtilde{\gamma_k}}, A_{\wtilde{\gamma_k}} ) \}_{k \in \mathcal{K}}$ is a complete set of representatives of the equivalence classes of $H$-irreducibles appearing in $U$. For notational convenience, set $A:= A_{\wtilde{1}} = A_{\wtilde{\gamma_0}}$. Similarly, set $\alpha := \alpha_{\wtilde{1}} = \alpha_{\wtilde{\gamma_0}}$.

\subsection{Identifying each $A_{\wtilde{\gamma}}$ ($\gamma \in \Gamma$) with $A$}

In contrast with $\{ \prescript{\delta}{}{F} \}_{\delta \in \hat{\Gamma}}$ and $\{ \prescript{\wtilde{\gamma}}{}{B} \}_{\gamma \in \Gamma}$, the spaces $\{ A_{\wtilde{\gamma}} \}_{\gamma \in \Gamma}$ are \textit{not} naturally the same as vector spaces. However, the following lemma establishes for each $\gamma \in \Gamma$ an identification between $A_{\wtilde{\gamma}}$ and $A$ by showing that $\wtilde{\gamma} \cdot A_{\wtilde{\gamma}} = A$. (Here, the action of $\wtilde{\gamma}$ on $A_{\wtilde{\gamma}}$ is given by the action of $\wtilde{\gamma}$ on the range of $a_{\wtilde{\gamma}} \in A_{\wtilde{\gamma}} = \Hom_{G^{\circ}} ( \prescript{\wtilde{\gamma}}{}{B}, U )$.)

\begin{lemma} \label{lem:identify-A-gammas}
For any $\gamma \in \Gamma$, we have that $\wtilde{\gamma} \cdot A_{\wtilde{\gamma}} = A$ (as vector spaces). 
\end{lemma}

\begin{proof}
To start, note that 
$$A_{\wtilde{\gamma}} = \Hom_{G^{\circ}} (\prescript{\wtilde{\gamma}}{}{B}, U) = \bigoplus_{j \in \mathcal{J}} E_{\delta_j} \otimes \Hom_{G^{\circ}} ( \prescript{\wtilde{\gamma}}{}{B}, \prescript{\delta_j}{}{F} ).$$
Therefore, we see that
$$\wtilde{\gamma} \cdot A_{\wtilde{\gamma}} = \bigoplus_{j \in \mathcal{J}} E_{\delta_j} \otimes \prescript{\delta_j}{}{\varphi} ( \wtilde{\gamma} ) \cdot \Hom_{G^{\circ}}(\prescript{\wtilde{\gamma}}{}{B}, \prescript{\delta_j}{}{F}).$$
As a result, to prove the lemma, it suffices to show that 
$$\prescript{\delta_j}{}{\varphi} (\wtilde{\gamma}) \cdot \Hom_{G^{\circ}} (\prescript{\wtilde{\gamma}}{}{B}, \prescript{\delta_j}{}{F}) = \Hom_{G^{\circ}} (B , \prescript{\delta_j}{}{F})$$
for each $j \in \mathcal{J}$. To see this, suppose first that $\ell \in \Hom_{G^{\circ}} ( \prescript{\wtilde{\gamma}}{}{B}, \prescript{\delta_j}{}{F} )$. Then 
$$\ell ( \prescript{\wtilde{\gamma}}{}{\beta} (g_0) \cdot b ) = \prescript{\delta_j}{}{\varphi} (g_0) \cdot \ell (b) \hspace{.25cm} \text{ for all } b \in B, \, g_0 \in G^{\circ},$$
which gives that 
$$\ell (\beta (g_0) \cdot b) = \prescript{\delta_j}{}{\varphi} ( \wtilde{\gamma}^{-1} g_0 \wtilde{\gamma} ) \cdot \ell (b) \hspace{.25cm} \text{ for all } b \in B, \, g_0 \in G^{\circ}.$$
Consequently,
\begin{align*}
[ \prescript{\delta_j}{}{\varphi} (\wtilde{\gamma}) \cdot \ell ] ( \beta (g_0) \cdot b ) &= \prescript{\delta_j}{}{\varphi} (\wtilde{\gamma}) \prescript{\delta_j}{}{\varphi} ( \wtilde{\gamma}^{-1} g_0 \wtilde{\gamma} ) \cdot \ell (b) = \prescript{\delta_j}{}{\varphi} (g_0) [ \prescript{\delta_j}{}{\varphi} (\wtilde{\gamma}) \cdot \ell ](b).
\end{align*}
It follows that $\prescript{\delta_j}{}{\varphi} (\wtilde{\gamma}) \cdot \ell \in \Hom_{G^{\circ}} (B, \prescript{\delta_j}{}{F})$, and hence that $\prescript{\delta_j}{}{\varphi} (\wtilde{\gamma}) \cdot \Hom_{G^{\circ}} (\prescript{\wtilde{\gamma}}{}{B}, \prescript{\delta_j}{}{F}) \subseteq \Hom_{G^{\circ}} (B, \prescript{\delta_j}{}{F})$. Towards proving containment in the other direction, suppose that $\ell \in \Hom_{G^{\circ}} (B, \prescript{\delta_j}{}{F})$. Then
$$\ell (\beta (g_0) \cdot b) = \prescript{\delta_j}{}{\varphi} (g_0) \cdot \ell (b) \hspace{.25cm} \text{ for all } b \in B, \, g_0 \in G^{\circ},$$
so we see that 
\begin{align*}
[\prescript{\delta_j}{}{\varphi}(\wtilde{\gamma})^{-1} \cdot \ell] ( \prescript{\wtilde{\gamma}}{}{\beta} (g_0) \cdot b ) &= \prescript{\delta_j}{}{\varphi} (\wtilde{\gamma})^{-1} \prescript{\delta_j}{}{\varphi} (\wtilde{\gamma} g_0 \wtilde{\gamma}^{-1}) \cdot \ell (b) = \prescript{\delta_j}{}{\varphi} (g_0) [ \prescript{\delta_j}{}{\varphi} (\wtilde{\gamma})^{-1} \cdot \ell ](b).
\end{align*}
It follows that $\Hom_{G^{\circ}} (B,\prescript{\delta_j}{}{F}) \subseteq \prescript{\delta_j}{}{\varphi}(\wtilde{\gamma}) \cdot \Hom_{G^{\circ}} ( \prescript{\wtilde{\gamma}}{}{B}, \prescript{\delta_j}{}{F} )$, completing the proof.
\end{proof}

As a result of this lemma, we see that each $\wtilde{\gamma_k} \vert_{A_{\wtilde{\gamma_k}}}$ defines an isomorphism $\wtilde{\gamma_k}\vert_{A_{\wtilde{\gamma_k}}} \colon A_{\wtilde{\gamma_k}} \xrightarrow{\sim} A$ (with inverse $(\wtilde{\gamma_k}\vert_{A_{\wtilde{\gamma_k}}})^{-1} = ( \wtilde{\gamma_k}^{-1} )\vert_A \colon A \xrightarrow{\sim} A_{\wtilde{\gamma_k}}$.) In the next subsection, we will see that each $\wtilde{\gamma_k} \vert_{A_{\wtilde{\gamma_k}}}$ defines a certain equivalence of representations.

\subsection{Defining a $\Gamma$-action on the $H$-irreducibles} \label{subsec:Gamma-action-on-H-irreps}

For $g \in G$ and $(\alpha_i, A_i)$ an $H$-irreducible in $U$, define 
\begin{equation} \label{eq:alpha-twisted-by-g}
\prescript{g}{}{\alpha_i} (h) := \alpha_i ( ghg^{-1} ) = \alpha_i (\mu (g,h) g) = ( \alpha_i (\cdot) \otimes \mu (g,\cdot) )(h).
\end{equation}
Recall that $\mu(g_{\gamma}, \cdot) = \mu (g_{\gamma}', \cdot)$ for any $\gamma \in \Gamma$ and any $g_{\gamma}, g_{\gamma}' \in \pi_G^{-1}(\gamma)$. Therefore, this $G$-action descends to a $\Gamma$-action on the $H$-irreducibles: for $\gamma \in \Gamma$, define
\begin{equation} \label{eq:alpha-twisted-by-gamma}
\prescript{\gamma}{}{\alpha_i}(\cdot) := \alpha_i (\cdot ) \otimes \mu (\gamma , \cdot).
\end{equation}
Note that since $\prescript{g_{\gamma}}{}{\alpha_i}$ and $\prescript{g_{\gamma}'}{}{\alpha_i}$ are not only equivalent but \textit{equal} representations of $H$ for $g_{\gamma}, g_{\gamma}' \in \pi_G^{-1}(\Gamma)$, $\prescript{\gamma}{}{\alpha_i}$ is a well-defined representation of $H$. Since $\alpha_i$ is irreducible and $\mu (\gamma, \cdot)$ is a one-dimensional (irreducible) representation of $H$, we indeed have that $\prescript{\gamma}{}{\alpha_i}$ is irreducible. Write $(\prescript{\gamma}{}{\alpha_i}, \prescript{\gamma}{}{A_i})$ for this twisted representation; here, the space $\prescript{\gamma}{}{A_i}$ is the same vector space as $A_i$ but with a different $H$-action. As in Subsection \ref{subsec:hat-Gamma-action-on-G-irreps}, we can easily check that this indeed defines a $\Gamma$-action. For an $H$-irreducible $\alpha_i$ in $U$, set 
$$I_G(\alpha_i) := \{ g \in G \, : \, \prescript{g}{}{\alpha_i} \simeq \alpha_i \}.$$

Recall that each $\wtilde{\gamma_k} \vert_{A_{\wtilde{\gamma_k}}}$ defines an isomorphism $ \wtilde{\gamma_k} \vert_{A_{\wtilde{\gamma_k}}} \colon A_{\wtilde{\gamma_k}} \xrightarrow{\sim} A$. We now see that each $\wtilde{\gamma_k} \vert_{A_{\wtilde{\gamma_k}}}$ in fact defines an equivalence of representations between $\alpha_{\wtilde{\gamma_k}}$ and $\prescript{\gamma_k}{}{\alpha}$: Indeed, using that $\wtilde{\gamma_k} \cdot A_{\wtilde{\gamma_k}} = A$, it is not hard to see that 
$$\prescript{\gamma_k}{}{\alpha} (h) = \alpha ( \wtilde{\gamma_k} h \wtilde{\gamma_k}^{-1} ) = (\wtilde{\gamma_k}) \vert_{A_{\wtilde{\gamma_k}}} \alpha_{\wtilde{\gamma_k}} (h) ( \wtilde{\gamma_k})\vert_{A_{\wtilde{\gamma_k}}}^{-1}  \hspace{.25cm} \text{ for all } h \in H.$$
Therefore, $\{ ( \prescript{\gamma_k}{}{\alpha}, \prescript{\gamma_k}{}{A} ) \}_{k \in \mathcal{K}}$ is a complete set of representatives of the equivalence classes of $H$-irreducibles in $U$. Consequently, we see that 
\begin{equation} \label{eq:I_G}
I_G(\alpha) = I_G(\beta) =: I_G.
\end{equation}

With the $H$-irreducibles $\{ (\prescript{\gamma_k}{}{\alpha}, \prescript{\gamma_k}{}{A})  \}_{k \in \mathcal{K}}$ defined, we make the following observation about how the $H^{\circ}$-irreducibles $\{ ( \varepsilon_{\delta_j}, E_{\delta_j} ) \}_{j \in \mathcal{J}}$ sit inside of the $H$-irreducibles $\{ (\prescript{\gamma_k}{}{\alpha}, \prescript{\gamma_k}{}{A})  \}_{k \in \mathcal{K}}$:

\begin{lemma} \label{lem:identify-hom-spaces}
For all $k \in \mathcal{K}$, $\Res^{H}_{H^{\circ}} ( \prescript{\gamma_k}{}{A} ) = \bigoplus_{j \in J} E_{\delta_j} \otimes \Hom_{G^{\circ}} ( B, F ) $. In particular, 
$$\Hom_{H^{\circ}} ( E_{\delta_j}, A ) = \Hom_{H^{\circ}} (E_{\delta_j} , \prescript{\gamma_k}{}{A}) = \Hom_{G^{\circ}} (B, \prescript{\delta_j}{}{F}) = \Hom_{G^{\circ}} (B, F)$$
for all $j \in \mathcal{J}$ and $k \in \mathcal{K}$.
\end{lemma}

\begin{proof}
By construction, we have that $\Res^{H}_{H^{\circ}} (A) = \bigoplus_{j \in \mathcal{J}} E_{\delta_j} \otimes \Hom_{H^{\circ}} (E_{\delta_j}, A)$. Recalling that $\prescript{\gamma_k}{}{\alpha} (\cdot ) = \alpha (\cdot ) \otimes \mu (\gamma_k, \cdot)$ (where $\mu (\gamma_k, \cdot)$ is trivial on $H^{\circ}$), we see that $\Res^{H}_{H^{\circ}} ( \prescript{\gamma_k}{}{A} ) = \Res^H_{H^{\circ}} (A)$ and that $\Hom_{H^{\circ}} ( E_{\delta_j}, \prescript{\gamma_k}{}{A} ) = \Hom_{H^{\circ}} ( E_{\delta_j} , A )$ for all $k \in \mathcal{K}$. By the same reasoning (and as noted in Subsection \ref{subsec:identify-G-circ-irreps}), we have that $\Hom_{G^{\circ}} (B, F) = \Hom_{G^{\circ}} (B, \prescript{\delta_j}{}{F})$. Finally, by the proof of Lemma \ref{lem:identify-A-gammas}, we have that 
$$\Hom_{H^{\circ}} (E_{\delta_j}, \prescript{\gamma_k}{}{A}) = \prescript{\delta_j}{}{\varphi} (\wtilde{\gamma_k}) \cdot \Hom_{G^{\circ}} ( \prescript{\wtilde{\gamma_k}}{}{B}, \prescript{\delta_j}{}{F}) = \Hom_{G^{\circ}} (B, \prescript{\delta_j}{}{F}),$$
and hence that $\Hom_{H^{\circ}} (E_{\delta_j}, A) = \Hom_{G^{\circ}} (B, F)$, completing the proof. 
\end{proof}

\subsection{Identifying the embedding of $H$ with $\bigoplus_{k \in \mathcal{K}} \prescript{\gamma_k}{}{\alpha} \otimes B$} \label{subsec:identify-H-irreps}

As in Subsection \ref{subsec:identify-G-circ-irreps}, we will define an isomorphism $U \xrightarrow{\sim} U'$ so that $\bigoplus_{k \in \mathcal{K}} \prescript{\gamma_k}{}{\alpha} \otimes B$ is the embedding of $H$ in $U'$. Noting that 
$$U = \bigoplus_{k \in \mathcal{K}} \prescript{\wtilde{\gamma_k}}{}{B} \otimes A_{\wtilde{\gamma_k}} = \bigoplus_{\substack{j \in \mathcal{J} \\ k \in \mathcal{K}}} \prescript{\wtilde{\gamma_k}}{}{B} \otimes E_{\delta_j} \otimes \Hom_{G^{\circ}}(\prescript{\wtilde{\gamma_k}}{}{B}, F)$$
and defining 
$$U' := \bigoplus_{ k \in \mathcal{K}} \prescript{\wtilde{\gamma_k}}{}{B} \otimes \prescript{\gamma_k}{}{A} = \bigoplus_{\substack{j \in \mathcal{J} \\ k \in \mathcal{K}}} \prescript{\wtilde{\gamma_k}}{}{B} \otimes E_{\delta_j} \otimes \Hom_{G^{\circ}} (B, F), $$
we see that $D:= \text{diag} \{\text{id}_B \otimes \wtilde{\gamma_k} \vert_{A_{\wtilde{\gamma_k}} } \}$ gives an isomorphism $U \xrightarrow{\sim} U'$.

Under this isomorphism, each $x \in GL(U)$ gets sent to $\sigma (x) := DxD^{-1} \in GL(U')$. In particular, $H \subseteq \prod_{k \in \mathcal{K}} GL(A_{\wtilde{\gamma_k}}) \subseteq GL(U)$ gets sent to $\sigma ( H ) \subseteq \prod_{k \in \mathcal{K}} GL(\prescript{\gamma_k}{}{A}) \subseteq GL(U')$. We see that $\{ ( \prescript{\sigma ( \wtilde{\gamma_k} )}{}{\alpha} , \, \prescript{\sigma ( \wtilde{\gamma_k} )}{}{A}  ) \}_{k \in \mathcal{K}}$ is a complete set of representatives for the equivalence classes of $\sigma (H)$-irreducibles appearing in $U'$. Note that for $h_0 \in H^{\circ}$, $\sigma (h_0) = D h_0 D^{-1}$ is in some sense equal to $h_0$ (since conjugation by $D$ does not affect any of the $E_{\delta_j}$ spaces on which $H^{\circ}$ lives), but $\sigma (h_0)$ sits in $GL(U')$ whereas $h_0$ sits in $GL(U)$. We can therefore identify $\sigma (H^{\circ})$ with $H^{\circ}$.

Similarly, we can identify $\sigma (G^{\circ})$ with $G^{\circ}$. Set $F' := \bigoplus_{k \in \mathcal{K}} \prescript{\wtilde{\gamma_k}}{}{B} \otimes \Hom_{G^{\circ}} (B, F)$
and define $\varphi' \colon \sigma (G) \rightarrow GL(F')$ by $\varphi' (\sigma (g)) = D \vert_{F} \varphi (g) D \vert_F^{-1}$. Furthermore, for $\delta \in \hat{\Gamma}$, define $\prescript{\sigma (\delta)}{}{\varphi'}(\sigma (g)) := \varphi' ( \sigma (\wtilde{\delta}) \sigma (g) \sigma (\wtilde{\delta})^{-1} )$. Then 
$$\prescript{\sigma (\delta)}{}{\varphi'}(\sigma (g)) = \varphi' ( \sigma (\wtilde{\delta}g \wtilde{\delta}^{-1}) ) = D \vert_F \prescript{\delta}{}{\varphi}(g) D \vert_F^{-1}.$$
In this way, we see that $\{ ( \prescript{\sigma (\delta_j)}{}{\varphi'}, \prescript{\sigma (\delta_j)}{}{F'} ) \}_{j \in \mathcal{J}}$ is a full set of representatives for the equivalence classes of $\sigma (G)$-irreducibles appearing in $U'$. By construction, we have that $\Hom_{G^{\circ}} ( \prescript{\wtilde{\gamma_k}}{}{B}, F' ) = \Hom_{G^{\circ}} (B,F)$, and we see that $\{ ( \prescript{\sigma ( \wtilde{\gamma_k} )}{}{\beta} , \, \prescript{\sigma ( \wtilde{\gamma_k} )}{}{B}  ) \}_{k \in \mathcal{K}}$ is a complete set of representatives for the equivalence classes of $\sigma (G^{\circ})$-irreducibles appearing in $U'$. We can therefore replace $U$ with $U'$, $H$ with $\sigma (H)$, $G$ with $\sigma (G)$, $(\varphi, F)$ with $(\varphi' , F')$, $\wtilde{(\cdot)} \colon \Gamma \rightarrow G_d$ with $\wtilde{(\cdot)}^{\,'} \colon \Gamma \rightarrow \sigma(G_d)$, and $\wtilde{(\cdot)} \colon \hat{\Gamma} \rightarrow H_d$ with $\wtilde{(\cdot)}^{\,'} \colon \hat{\Gamma} \rightarrow \sigma (H_d)$ moving forward so that
\begin{itemize}
\item $\bigoplus_{k \in \mathcal{K}} \prescript{\gamma_k}{}{\alpha} \otimes B$ is the embedding of $H$ in $U$,
\item $\bigoplus_{k \in \mathcal{K}} \prescript{\wtilde{\gamma_k}}{}{\beta} \otimes A$ is the embedding of $G^{\circ}$ in $U$, and
\item $\bigoplus_{j \in \mathcal{J}} \prescript{\delta_j}{}{\varphi} \otimes E_{\delta_j}$ is the embedding of $G$ in $U$.
\end{itemize}

\subsection{Identifying each $E_{\delta}$ ($\delta \in \hat{\Gamma}$) with $E$}

In contrast with $\{ \prescript{\delta}{}{F} \}_{\delta \in \hat{\Gamma}}$, $\{ \prescript{\wtilde{\gamma}}{}{B} \}_{\gamma \in \Gamma}$, and $\{ \prescript{\gamma}{}{A} \}_{\gamma \in \Gamma}$, the spaces $\{ E_{\delta} \}_{\delta \in \hat{\Gamma}}$ are \textit{not} naturally the same as vector spaces. In the following lemma, we establish for each $\delta \in \hat{\Gamma}$ an identification between $E_{\delta}$ and $E$ by showing that $h_{\delta} \cdot E_{\delta} = E$ for any $h_{\delta} \in \pi_H^{-1}(\delta)$. (Here, the action of $h_{\delta}$ on $E_{\delta}$ is given by the action of $h_{\delta}$ on the range of $e_{\delta} \in E_{\delta} = \Hom_{G} (\prescript{\delta}{}{F},U)$.)

\begin{lemma} \label{lem:E-orbit}
For any $\delta \in \hat{\Gamma}$ and $h_{\delta} \in \pi_H^{-1}(\delta)$, we have that $h_{\delta} \cdot E_{\delta} = E$ (as vector spaces).
\end{lemma}

\begin{proof}
We have that 
$$E := \Hom_{G}(F,U) = \{ e \colon F \rightarrow U \; \vert \; e(\varphi (g) \cdot f) = g \cdot e(f) \text{ for all } f\in F, \, g \in G \}.$$
Looking at $E_{\delta} := \Hom_{G} ( \prescript{\delta}{}{F}, U)$, we see that
\begin{align*}
    E_{\delta} &= \{ e \colon F \rightarrow U \; \vert \; e( \prescript{\delta}{}{\varphi} (g) \cdot f) = g \cdot e(f) \text{ for all } f \in F, \, g \in G \} \\ 
    &= \{ e \colon F \rightarrow U \; \vert \; e( \varphi (g) \cdot \mu (g,\delta)^{-1} \cdot f ) = g \cdot e(f) \text{ for all } f \in F, \, g \in G \} \\
    &= \{ e \colon F \rightarrow U \; \vert \; e(\varphi (g)\cdot f) = (\mu (g,\delta) g) \cdot e(f) \text{ for all } f \in F, g \in G \}.
\end{align*}
Suppose first that we have $h_{\delta}^{-1} \cdot e \in h_{\delta}^{-1} \cdot E$. Then for all $g \in G$ and $f \in F$,
$$h_{\delta}^{-1} \cdot e (\varphi (g) \cdot f) = h_{\delta}^{-1} (g \cdot e (f)) =( \mu(g, \delta) g) \cdot (h_{\delta}^{-1} \cdot e)(f),$$
so $h_{\delta}^{-1} \cdot e \in E_{\delta}$ and $e \in h_{\delta} \cdot E_{\delta}$. On the other hand, suppose we have $e_{\delta} \in E_{\delta}$, so
$$e_{\delta} ( \varphi (g) \cdot f) = (\mu (g,\delta)g) \cdot e_{\delta}(f) \text{ for all } f \in F, g \in G.$$
Then
\begin{align*}
    (h_{\delta} \cdot e_{\delta}) (\varphi (g) \cdot f) &= h_{\delta} \cdot (\mu (g,\delta) g) \cdot e_{\delta}(f) \\
    &= \mu (g,\delta) \cdot (h_{\delta} \cdot g) \cdot e_{\delta}(f) \\
    &= g (h_{\delta} \cdot e_{\delta})(f),
\end{align*}
so $h_{\delta} \cdot e_{\delta} \in E$. It follows that $h_{\delta} \cdot E_{\delta} = E$, as desired.
\end{proof}

In particular, we have that $\wtilde{\delta} \cdot E_{\delta} = E$ for all $\delta \in \hat{\Gamma}$. Therefore, each $\wtilde{\delta_j}\vert_{E_{\delta_j}}$ defines an isomorphism $\wtilde{\delta_j}\vert_{E_{\delta_j}} \colon E_{\delta_j} \xrightarrow{\sim} E$ (with inverse $(\wtilde{\delta_j}\vert_{E_{\delta_j}})^{-1}= (\wtilde{\delta_j}^{-1})\vert_E \colon E \xrightarrow{\sim} E_{\delta_j}$). In the next subsection, we will see that each $\wtilde{\delta_j}\vert_{E_{\delta_j}}$ defines a certain equivalence of representations. 

\subsection{Defining a $\hat{\Gamma}$-action on the $H^{\circ}$-irreducibles} \label{subsec:defining-hat-Gamma-action}

As in Subsection \ref{subsec:Gamma-action-on-G-circ}, we can define an $H$-action and a $\hat{\Gamma}$-action on the $H^{\circ}$-irreducibles appearing in $U$. For $h \in H$ and $(\varepsilon_j,E_j)$ an $H^{\circ}$-irreducible in $U$, define 
\begin{equation} \label{eq:epsilon-twisted-by-h}
\prescript{h}{}{\varepsilon_j(h_0)} := \varepsilon_j (h h_0 h^{-1}).
\end{equation}
We see that each $\prescript{h}{}{\varepsilon_j}$ is irreducible. However, as in Subsection \ref{subsec:Gamma-action-on-G-circ}, this action does not descend to a well-defined $\hat{\Gamma}$-action on $H^{\circ}$-irreducibles. To get a well-defined $\hat{\Gamma}$-action, we make use of our lift $\wtilde{(\cdot)} \colon \hat{\Gamma} \rightarrow H$, letting $\delta \in \hat{\Gamma}$ act as
\begin{equation*} 
\prescript{\wtilde{\delta}}{}{\varepsilon_j (h_0)} = \varepsilon_j ( \wtilde{\delta} h_0  (\wtilde{\delta})^{-1} ).
\end{equation*}
Let $\prescript{\wtilde{\delta}}{}{E_j}$ denote the corresponding representation space; note that $\prescript{\wtilde{\delta}}{}{E_j}$ is the same vector space as $E_j$ but with a different $H^{\circ}$-action. As in Subsection \ref{subsec:Gamma-action-on-G-circ}, we can check that this indeed defines a $\hat{\Gamma}$-action. For an $H^{\circ}$-irreducible $\varepsilon_j$ in $U$, set 
$$I_H(\varepsilon_j) := \{ h \in H \, : \, \prescript{h}{}{\varepsilon_j} \simeq \varepsilon_j \}.$$

Recall that each $\wtilde{\delta_j}\vert_{E_{\delta_j}}$ defines an isomorphism $\wtilde{\delta_j}\vert_{E_{\delta_j}} \colon E_{\delta_j} \xrightarrow{\sim} E$. We now see that each $\wtilde{\delta_j}\vert_{E_{\delta_j}}$ in fact defines an equivalence of representations between $\varepsilon_{\delta_j}$ and $\prescript{\wtilde{\delta_j}}{}{\varepsilon}$: indeed, using that $\wtilde{\delta_j} \cdot E_{\delta_j} = E$, it is not hard to see that 
$$\prescript{\wtilde{\delta_j}}{}{\varepsilon} (h_0) = \varepsilon (\wtilde{\delta_j} h_0 \wtilde{\delta_j}^{-1}) = (\wtilde{\delta_j})\vert_{E_{\delta_j}} \varepsilon_{\delta_j}(h_0) (\wtilde{\delta_j})\vert_{E_{\delta_j}}^{-1}.$$ 
Therefore, $\{ (\prescript{\wtilde{\delta_j}}{}{\varepsilon}, \prescript{\wtilde{\delta_j}}{}{E}) \}_{j \in \mathcal{J}}$ is a complete set of representatives of the equivalence classes of $H^{\circ}$-irreducibles in $U$. Consequently, we see that 
\begin{equation} \label{eq:I_H}
I_H(\varepsilon) = I_H(\varphi)=:I_H. 
\end{equation}

\subsection{Identifying the embedding of $H^{\circ}$ with $\bigoplus_{j\in \mathcal{J}} \prescript{\wtilde{\delta_j}}{}{\varepsilon} \otimes F$} \label{subsec:identify-H0}

As in Subsections \ref{subsec:identify-G-circ-irreps} and \ref{subsec:identify-H-irreps}, we would like to define an isomorphism $U \xrightarrow{\sim} U'$ (for an appropriate choice of $U'$) so that $\bigoplus_{j \in \mathcal{J}} \prescript{\wtilde{\delta_j}}{}{\varepsilon} \otimes F$ is the embedding of $H^{\circ}$ in $U'$. Recall that $U = \bigoplus_{j \in \mathcal{J}} E_{\delta_j} \otimes \prescript{\delta_j}{}{F}$ and define $U' := \bigoplus_{j \in \mathcal{J}} \prescript{\wtilde{\delta_j}}{}{E} \otimes \prescript{\delta_j}{}{F}$. Then $D:= \diag \{ (\wtilde{\delta_j})\vert_{E_{\delta_j}} \otimes \text{id}_{F} \}_{j \in \mathcal{J}}$ defines an isomorphism $U \xrightarrow{\sim} U'$.

Under this isomorphism, each $x \in GL(U)$ gets sent to $\sigma (x) := DxD^{-1} \in GL(U')$. In particular, $H^{\circ} = \prod_{j \in \mathcal{J}} GL(E_{\delta_j}) \subseteq GL(U)$ gets sent to $\sigma (H^{\circ}) = \prod_{j \in \mathcal{J}} GL(\prescript{\wtilde{\delta_j}}{}{E}) \subseteq GL(U')$. Note that for $g \in G$, we have $\sigma (g) = DgD^{-1}$ is in some sense equal to $g$ (since conjugation by $D$ does not affect any of the $\prescript{\delta_j}{}{F}$ spaces on which $G$ lives), but $\sigma (g)$ sits in $GL(U') = GL(\bigoplus_{j \in \mathcal{J}} \prescript{\delta_j}{}{F} \otimes \prescript{\wtilde{\delta_j}}{}{E})$ whereas $g$ sits in $GL(U) = GL(\bigoplus_{j \in \mathcal{J}} \prescript{\delta_j}{}{F} \otimes E_{\delta_j})$. We can therefore identify $\sigma (G)$ with $G$. 

Recall that $A = \bigoplus_{j \in \mathcal{J}} E_{\delta_j} \otimes \Hom_{G^{\circ}} (B,F)$. Set $A' := \bigoplus_{j \in \mathcal{J}} \prescript{\wtilde{\delta_j}}{}{E} \otimes \Hom_{G^{\circ}} (B,F)$. Then $D$ induces a map $D \vert_A \colon A \rightarrow A'$. Define $\alpha' \colon \sigma (H) \rightarrow GL(A')$ by $\alpha' ( \sigma (h) ) = D\vert_A \alpha (h) D\vert_A^{-1}$. Furthermore, for $\gamma \in \Gamma$, define $\prescript{\sigma (\gamma)}{}{\alpha'} ( \sigma (h) ) := \alpha' ( \sigma (\wtilde{\gamma}) \sigma (h) \sigma (\wtilde{\gamma})^{-1} ) $. Then 
$$\prescript{\sigma (\gamma)}{}{\alpha '} ( \sigma (h) ) = \alpha' ( \sigma ( \wtilde{\gamma} h \wtilde{\gamma}^{-1} ) ) = D\vert_A \prescript{\gamma}{}{\alpha}(h)D\vert_A^{-1}.$$
In this way, we see that $\{ ( \prescript{\sigma (\gamma_k)}{}{\alpha'}, \prescript{\sigma (\gamma_k)}{}{A'} ) \}_{k \in \mathcal{K}}$ is a full set of representatives for the equivalence classes of $\sigma (H)$-irreducibles appearing in $U'$. Note also that $\Hom_{\sigma (H)} (\prescript{\sigma (\gamma_k)}{}{A'}, U') = \Hom_H (\prescript{\gamma_k}{}{A}, U)$ and that 
$$\Hom_{\sigma(H^{\circ})} ( \prescript{\wtilde{\delta_j}}{}{E}, A') = \Hom_{H^{\circ}} (E_{\delta_j}, A) = \Hom_{G^{\circ}} (B,F).$$ 
We can therefore replace $U$ with $U'$, $H$ with $\sigma(H)$, $(\alpha, A)$ with $(\alpha', A')$, $G$ with $\sigma (G)$, $\wtilde{(\cdot)} \colon \Gamma \rightarrow G_d$ with $\wtilde{(\cdot)}^{\,'} \colon \Gamma \rightarrow \sigma(G_d)$, and $\wtilde{(\cdot)} \colon \hat{\Gamma} \rightarrow H_d$ with $\wtilde{(\cdot)}^{\,'} \colon \hat{\Gamma} \rightarrow \sigma (H_d)$ moving forward so that

\begin{itemize}
\item $\bigoplus_{j \in \mathcal{J}} \prescript{\wtilde{\delta_j}}{}{\varepsilon} \otimes F$ is the embedding of $H^{\circ}$ in $U$,
\item $\bigoplus_{k \in \mathcal{K}} \prescript{\gamma_k}{}{\alpha} \otimes B$ is the embedding of $H$ in $U$,
\item $\bigoplus_{k \in \mathcal{K}} \prescript{\wtilde{\gamma_k}}{}{\beta} \otimes A$ is the embedding of $G^{\circ}$ in $U$, and
\item $\bigoplus_{j \in \mathcal{J}} \prescript{\delta_j}{}{\varphi} \otimes E$ is the embedding of $G$ in $U$.
\end{itemize}

\section{Irreducible representations of \texorpdfstring{$I_G$ and $I_H$}{inertia subgroups}} \label{sec:reps-of-I_G-and-I_H}

By Section \ref{sec:orbits}, we have that $\Hom_{G^{\circ}} ( \prescript{\wtilde{\gamma_k}}{}{B}, \prescript{\delta_j}{}{F} ) = \Hom_{H^{\circ}} ( \prescript{\wtilde{\delta_{j'}}}{}{E}, \prescript{\gamma_{k'}}{}{A} )$ for all $j,j' \in \mathcal{J}$ and $k,k' \in \mathcal{K}$. Moving forward, we will call this vector space $L$. (When we wish to refer specifically to the copy of $L$ corresponding to a specific $j \in \mathcal{J}$ and $k\in \mathcal{K}$, we will write $L_{j,k}$.) It is worth noting that
$$L = \Hom_{G^{\circ}} ( B,F ) = \Hom_{G^{\circ}} (B, \Hom_{H^{\circ}} (E, U)) = \Hom_{G^{\circ} \times H^{\circ}} (B \otimes E, U),$$
where we are using the tensor-Hom adjunction. Thus, there are three ways to view $L$: as the multiplicity space of $B$ in $F$, as the multiplicity space of $E$ in $A$, or as the multiplicity space of $B \otimes E$ in $U$. Therefore, Section \ref{sec:orbits} gives the following:

\begin{cor} \label{cor:G,H-restricted-reps}
The $G$- and $G^{\circ}$-irreducibles (resp.~$H$- and $H^{\circ}$-irreducibles) are related by
$$\Res^{G}_{G^{\circ}} (\prescript{\delta_j}{}{\varphi}) = \bigoplus_{k \in \mathcal{K}} \prescript{\wtilde{\gamma_k}}{}{\beta} \otimes L \hspace{.5cm} \text{ (resp.} \hspace{.25cm} \Res^{H}_{H^{\circ}} (\prescript{\gamma_k}{}{\alpha} ) = \bigoplus_{j \in \mathcal{J}} \prescript{\wtilde{\delta_j}}{}{\varepsilon} \otimes L \text{)}$$
for all $j \in \mathcal{J}$ and $k \in \mathcal{K}$.
\end{cor}

\subsection{Defining $I_G$- and $I_H$-irreducibles and relating them to the $G$-, $G^{\circ}$, $H$-, and $H^{\circ}$-irreducibles} \label{subsec:define-I_G-and-I_H-irreps}

Let $\Irr(\cdot)$ denote the set of all irreducible representations of the group $\cdot$, and let $\preceq$ denote the relation of being a subrepresentation. Then by Clifford correspondence \cite[Theorem 2.2]{Clifford}, we have for each $k \in \mathcal{K}$ that the map
\begin{align*}
\{ \rho \in \Irr (I_G) \, : \, \rho \preceq \Ind_{G^{\circ}}^{I_G} ( \prescript{\wtilde{\gamma_k}}{}{\beta} ) \} & \rightarrow \{ \theta \in \Irr (G) \, : \, \prescript{\wtilde{\gamma_k}}{}{\beta} \preceq \Res^{G}_{G^{\circ}} (\theta) \} \\
\rho & \mapsto \Ind_{I_G}^{G} (\rho)
\end{align*}
is a bijection. Therefore, for each $j \in \mathcal{J}$ and $k \in \mathcal{K}$, there exists $\rho_{j,k} \in \Irr (I_G)$ such that 
\begin{equation} \label{eq:rho-jk-def}
\rho_{j,k} \preceq \Ind_{G^{\circ}}^{I_G}(\prescript{\wtilde{\gamma_k}}{}{\beta})\hspace{.25cm} \text{ and } \hspace{.25cm} \Ind_{I_G}^{G} (\rho_{j,k}) = \prescript{\delta_j}{}{\varphi}.
\end{equation}
Let $R_{j,k}$ denote the representation space corresponding to $\rho_{j,k}$. 

In the same way, we see that for each $j \in \mathcal{J}$ and $ k \in \mathcal{K}$, there exists $\psi_{j,k} \in \Irr(I_H)$ such that 
\begin{equation} \label{eq:psi-jk-def}
\psi_{j,k} \preceq \Ind_{H^{\circ}}^{I_H}(\prescript{\wtilde{\delta_j}}{}{\varepsilon}) \hspace{.25cm} \text{ and } \hspace{.25cm} \Ind_{I_H}^{H} (\psi_{j,k}) = \prescript{\gamma_k}{}{\alpha}.
\end{equation}
Let $P_{j,k}$ denote the representation space corresponding to $\psi_{j,k}$.

Now, by Frobenius reciprocity, there are natural isomorphisms
$$\Hom_{G} ( \Ind_{I_G}^{G} (R_{j',k}), \prescript{\delta_j}{}{F} ) \simeq \Hom_{I_G} (R_{j',k}, \Res^{G}_{I_G} (\prescript{\delta_j}{}{F}) )$$
and
$$ \Hom_{H} ( \Ind_{I_H}^{H} (P_{j,k'}), \prescript{\gamma_k}{}{A} ) \simeq \Hom_{I_H} (P_{j,k'}, \Res^{H}_{I_H} (\prescript{\gamma_k}{}{A}) ).$$
In this way, we see that 
$$\Res^{G}_{I_G} ( \prescript{\delta_j}{}{\varphi} ) = \bigoplus_{k \in \mathcal{K}} \rho_{j,k}
 \hspace{.25cm} \text{ and } \hspace{.25cm} \Res^{H}_{I_H} ( \prescript{\gamma_k}{}{\alpha} ) = \bigoplus_{j \in \mathcal{J}} \psi_{j,k} .$$  
These relationships will allow us to show that $I_G$- and $I_H$-irreducibles we've defined are in fact complete sets of representatives of the irreducibles of these groups appearing in $U$.

\begin{lemma} \label{lem:full-set-of-interia-reps}
The set $\{ ( \rho_{j,k}, R_{j,k} ) \}_{ j \in \mathcal{J} ,\, k \in \mathcal{K}  }$ (resp.~$\{ ( \psi_{j,k}, P_{j,k} ) \}_{ j \in \mathcal{J} ,\, k \in \mathcal{K}  }$) is a complete set of pairwise non-equivalent representatives of the equivalence classes of $I_G$-irreducibles (resp.~$I_H$-irreducibles) appearing in $U$.
\end{lemma}

\begin{proof}
Since $\Res^{G}_{I_G} ( \prescript{\delta_j}{}{\varphi} ) = \bigoplus_{k \in \mathcal{K}} \rho_{j,k}$ (and since $\{ (\prescript{\delta_j}{}{\varphi} , \prescript{\delta_j}{}{F} ) \}_{j \in \mathcal{J}}$ is a complete set of representatives of the $G$-irreducibles in $U$), it is clear that $\{ ( \rho_{j,k}, R_{j,k} ) \}_{ j \in \mathcal{J},\,  k \in \mathcal{K}  } $ is a complete set of representatives of the $I_G$-irreducibles in $U$. 

It remains to show that the elements of $\{ ( \rho_{j,k}, R_{j,k} ) \}_{ j \in \mathcal{J}, \,  k \in \mathcal{K}  }$ are pairwise non-equivalent. To this end, consider $\rho_{j,k}$ and $\rho_{j',k'}$, and suppose first that $j \neq j'$. Then 
$$\Ind_{I_G}^{G} (\rho_{j,k}) = \prescript{\delta_j}{}{\varphi} \not \simeq \prescript{\delta_{j'}}{}{\varphi} = \Ind_{I_G}^{G} (\rho_{j',k'}),$$
so $\rho_{j,k} \not \simeq \rho_{j',k'}$. Next, suppose that $k \neq k'$. The condition $\rho_{j,k} \preceq \Ind_{G^{\circ}}^{I_G}(\prescript{\wtilde{\gamma_k}}{}{\beta})$ gives that 
$$\Res^{I_G}_{G^{\circ}} (\rho_{j,k}) \preceq \Res^{I_G}_{G^{\circ}} (\Ind_{G^{\circ}}^{I_G}(\prescript{\wtilde{\gamma_k}}{}{\beta})) = \prescript{\wtilde{\gamma_k}}{}{\beta} \otimes \Hom_{G^{\circ}} ( \prescript{\wtilde{\gamma_k}}{}{B} , \Res^{I_G}_{G^{\circ}} ( \Ind_{G^{\circ}}^{I_G} ( \prescript{\wtilde{\gamma_k}}{}{B} ) )  ).$$ 
Similarly, we have that
$$\Res^{I_G}_{G^{\circ}} (\rho_{j',k'}) \preceq \Res^{I_G}_{G^{\circ}} (\Ind_{G^{\circ}}^{I_G}(\prescript{\wtilde{\gamma_{k'}}}{}{\beta})) = \prescript{\wtilde{\gamma_{k'}}}{}{\beta} \otimes \Hom_{G^{\circ}} ( \prescript{\wtilde{\gamma_{k'}}}{}{B} , \Res^{I_G}_{G^{\circ}} ( \Ind_{G^{\circ}}^{I_G} ( \prescript{\wtilde{\gamma_{k'}}}{}{B} ) )  ).$$ 
Since $\prescript{\wtilde{\gamma_k}}{}{\beta} \not \simeq \prescript{\wtilde{\gamma_{k'}}}{}{\beta}$, it follows that $\rho_{j,k} \not \simeq \rho_{j',k'}$. This completes the proof for $\{ ( \rho_{j,k}, R_{j,k} ) \}_{ j \in \mathcal{J} ,\, k \in \mathcal{K}  }$, and the case of $\{ ( \psi_{j,k}, P_{j,k} ) \}_{ j \in \mathcal{J} ,\, k \in \mathcal{K}  }$ follows in the same way. 
\end{proof}

To understand how these irreducibles relate to the irreducibles of $G^{\circ}$ and $H^{\circ}$, we note that 
$$\Res^{G}_{G^{\circ}} (\prescript{\delta_j}{}{\varphi}) = \Res_{G^{\circ}}^{I_G} ( \Res^{G}_{I_G} (\prescript{\delta_j}{}{\varphi}) ) = \Res^{I_G}_{G^{\circ}} \Big ( \bigoplus_{k \in \mathcal{K}} \rho_{j,k} \Big ) = \bigoplus_{k \in \mathcal{K}} \Res^{I_G}_{G^{\circ}} ( \rho_{j,k}).$$
On the other hand, we have that $\Res^{G}_{G^{\circ}} (\prescript{\delta_j}{}{\varphi}) = \bigoplus_{k \in \mathcal{K}} \prescript{\wtilde{\gamma_k}}{}{\beta} \otimes L$. Since $\rho_{j,k} \preceq \Ind_{G^{\circ}}^{I_G} ( \prescript{\wtilde{\gamma_k}}{}{\beta} )$, we see that
$$\Res^{I_G}_{G^{\circ}} (\rho_{j,k}) = \prescript{\wtilde{\gamma_k}}{}{\beta} \otimes L.$$
Similarly, we get that 
$$\Res^{I_H}_{H^{\circ}} (\psi_{j,k}) = \prescript{\wtilde{\delta_j}}{}{\varepsilon} \otimes L.$$
Now, note that we have natural isomorphisms
$$L = \Hom_{G^{\circ}} ( \prescript{\wtilde{\gamma_k}}{}{B}, R_{j,k} ) \simeq \Hom_{I_G} ( \Ind_{G^{\circ}}^{I_G} (\prescript{\wtilde{\gamma_k}}{}{B}) , R_{j,k} ) \simeq \Hom_{I_G} ( R_{j,k}, \Ind_{G^{\circ}}^{I_G}(\prescript{\wtilde{\gamma_k}}{}{B}) )^*$$
and 
$$L = \Hom_{H^{\circ}} ( \prescript{\wtilde{\delta_j}}{}{E}, P_{j,k} ) \simeq \Hom_{I_H} ( \Ind_{H^{\circ}}^{I_H} (\prescript{\wtilde{\delta_j}}{}{E}) , P_{j,k} ) \simeq \Hom_{I_H} ( P_{j,k}, \Ind_{H^{\circ}}^{I_H}(\prescript{\wtilde{\delta_j}}{}{E}) )^*.$$
Therefore, we have 
$$\Ind_{G^{\circ}}^{I_G}(\prescript{\wtilde{\gamma_k}}{}{\beta}) = \bigoplus_{j \in \mathcal{J}} \rho_{j,k} \otimes L^* \hspace{.5cm} \text{ and } \hspace{.5cm} \Ind_{H^{\circ}}^{I_H}(\prescript{\wtilde{\delta_j}}{}{\varepsilon}) = \bigoplus_{k \in \mathcal{K}} \psi_{j,k} \otimes L^* .$$

With this, we have established the following:

\begin{cor} \label{cor:inertia-summary}
For all $j\in \mathcal{J}$ and $k \in \mathcal{K}$, we have the following relationships between the $I_G$-, $G$-, and $G^{\circ}$-irreducibles (as well as the $I_H$-, $H$-, and $H^{\circ}$-irreducibles):
\begin{itemize}
\item $\Res^{G}_{I_G} (\prescript{\delta_j}{}{\varphi}) = \bigoplus_{k \in \mathcal{K}} \rho_{j,k}$ and $\Ind_{I_G}^{G} (\rho_{j,k}) = \prescript{\delta_j}{}{\varphi}$.
\item $\Res^{H}_{I_H} (\prescript{\gamma_k}{}{\alpha}) = \bigoplus_{j \in \mathcal{J}} \psi_{j,k}$ and $\Ind_{I_H}^{H} (\psi_{j,k}) = \prescript{\gamma_k}{}{\alpha}$. 
\item $\Res^{I_G}_{G^{\circ}} (\rho_{j,k}) = \prescript{\wtilde{\gamma_k}}{}{\beta} \otimes L$ and $\Ind_{G^{\circ}}^{I_G}(\prescript{\wtilde{\gamma_k}}{}{\beta}) = \bigoplus_{k \in \mathcal{K}} \rho_{j,k} \otimes L^*$.
\item $\Res^{I_H}_{H^{\circ}} (\psi_{j,k}) = \prescript{\wtilde{\delta_j}}{}{\varepsilon} \otimes L$ and $\Ind_{H^{\circ}}^{I_H}(\prescript{\wtilde{\delta_j}}{}{\varepsilon}) = \bigoplus_{k \in \mathcal{K}} \psi_{j,k} \otimes L^*$.
\end{itemize}
\end{cor}

To conclude this section, we show how the representation $\rho_{j,k}$ relates to $\rho_{0,k}$, and how the representation $\psi_{j,k}$ relates to $\psi_{0,k}$.

\newpage

\begin{cor} \label{cor:inertia-reps-jk}
We have the following:
\begin{itemize}
\item $\rho_{j,k}(\cdot) = \mu (\cdot, \delta_j)^{-1} \otimes \rho_{0,k}(\cdot) \text{ for all } j \in \mathcal{J} \text{ and } k\in \mathcal{K}.$
\item $\psi_{j,k}(\cdot) = \mu (\gamma_k, \cdot) \otimes \psi_{j,0} (\cdot) \text{ for all } j \in \mathcal{J} \text{ and } k\in \mathcal{K}.$
\end{itemize}
\end{cor}

\begin{proof}
By Corollary \ref{cor:inertia-summary}, we have that $\Res^{G}_{I_G} ( \prescript{\delta_j}{}{\varphi} ) = \bigoplus_{k \in \mathcal{K}} \rho_{j,k}$. On the other hand,
$$\Res^{G}_{I_G} (\prescript{\delta_j}{}{\varphi}) = \Res^{G}_{I_G} ( \mu (\cdot,\delta_j)^{-1} \otimes \varphi (\cdot) ) = \mu (\cdot, \delta_j)^{-1} \otimes \Res^{G}_{I_G}(\varphi) = \mu (\cdot, \delta_j)^{-1} \otimes \Big ( \bigoplus_{k \in \mathcal{K}} \rho_{0,k} \Big ).$$
Comparing these two expressions, we conclude that $\rho_{j,k}(\cdot) = \mu (\cdot, \delta_j)^{-1} \otimes \rho_{0,k}(\cdot)$ for all $j \in \mathcal{J}$ and $k \in \mathcal{K}$. Similarly, we conclude that $\psi_{j,k}(\cdot ) = \mu (\gamma_k, \cdot) \otimes \psi_{j,0} (\cdot)$ for all $j \in \mathcal{J}$ and $k \in \mathcal{K}$. 
\end{proof}

\section{Understanding \texorpdfstring{$G$}{G} and \texorpdfstring{$H$}{H} explicitly} \label{sec:explicit}

In this section, we work in the context of $G$ and state the anolog of each result for $H$.

\subsection{Showing that $(I_G)_d, \, (I_H)_d \subseteq GL(L)^{\vert \mathcal{J} \vert \cdot \vert \mathcal{K} \vert }$} \label{subsec:(I_G)_d-in-GL(L)'s}

We already know from Lemma \ref{lem:full-set-of-interia-reps} and Corollary \ref{cor:inertia-summary} that
$$I_G, \, I_H \subseteq GL(B \otimes L)^{\vert \mathcal{J} \vert \cdot \vert \mathcal{K} \vert}  .$$
In this subsection, we show that each $g_I \in I_G$ (resp.~$h_I \in I_H$) takes the form of a certain product of block diagonal matrices. In doing so, we show that $(I_G)_d, \, (I_H)_d \subseteq GL(L)^{\vert \mathcal{J} \vert \cdot \vert \mathcal{K} \vert}$.

\begin{prop} \label{prop:I_G-and-I_H}
\hspace{1cm}
\begin{itemize}
\item $I_G = [GL(B)^{\vert \mathcal{K} \vert} \times (I_G)_d ]/((\mathbb{C}^{\times})^{\vert \mathcal{K} \vert})_{\antidiag}$, and $(I_G)_d \subseteq GL(L)^{\vert \mathcal{J} \vert \cdot \vert \mathcal{K} \vert}$.
\item $I_H = [GL(E)^{\vert \mathcal{J} \vert} \times (I_H)_d ]/((\mathbb{C}^{\times})^{\vert \mathcal{J} \vert})_{\antidiag}$, and $(I_H)_d \subseteq GL(L)^{\vert \mathcal{J} \vert \cdot \vert \mathcal{K} \vert}$.
\end{itemize}
\end{prop}

\begin{proof}
We start by showing that $I_G \subseteq \left [ GL(B)^{\vert \mathcal{K} \vert } \times GL(L)^{\vert \mathcal{J} \vert \cdot \vert \mathcal{K} \vert } \right ]/ ((\mathbb{C}^{\times})^{\vert \mathcal{K} \vert})_{\antidiag}$. To this end, fix $g \in I_G$. By definition of $I_G$, we have that for each $k \in \mathcal{K}$,
$$\prescript{g}{}{ ( \prescript{\wtilde{\gamma_k}}{}{\beta} (\cdot) )} = N_k \prescript{\wtilde{\gamma_k}}{}{\beta} (\cdot ) N_k^{-1}$$ 
for some $N_k \in GL(B)$. Set
$$N := \diag ( \underbrace{ N_0 , \ldots , N_0}_{\vert \mathcal{J} \vert \cdot (\dim L)} , \ldots , \underbrace{ N_k , \ldots , N_k }_{\vert \mathcal{J} \vert \cdot (\dim L)} , \ldots ).$$
Recalling that $g_0 \in G^{\circ}$ gets identified with 
$$\diag ( \underbrace{\beta (g_0), \ldots , \beta (g_0)}_{\vert \mathcal{J} \vert \cdot (\dim L)}, \ldots , \underbrace{\prescript{\wtilde{\gamma_k}}{}{\beta}(g_0), \ldots , \prescript{\wtilde{\gamma_k}}{}{\beta} (g_0)}_{\vert \mathcal{J} \vert \cdot (\dim L)}, \ldots ) \in GL(B \otimes L)^{\vert \mathcal{J} \vert \cdot \vert \mathcal{K} \vert } ,$$
we see that 
$$g g_0 g^{-1} = N g_0 N^{-1}.$$
In particular, we see that 
$$N^{-1} g \in Z_{ GL(B \otimes L)^{\vert \mathcal{J} \vert \cdot \vert \mathcal{K} \vert } } \left ( GL(B)^{\vert \mathcal{K} \vert } \right ) = GL(L)^{\vert \mathcal{J} \vert \cdot \vert \mathcal{K} \vert }.$$
It follows that $g$ can be written as $g = N M$ for some element $M \in GL(L)^{\vert \mathcal{J} \vert \cdot \vert \mathcal{K} \vert }$, and hence that $I_G \subseteq \left [ GL(B)^{\vert \mathcal{K} \vert } \times GL(L)^{\vert \mathcal{J} \vert \cdot \vert \mathcal{K} \vert } \right ]/ ((\mathbb{C}^{\times})^{\vert \mathcal{K} \vert})_{\antidiag}$.

Therefore, since $GL(B)^{\vert \mathcal{K} \vert } \subseteq I_G$, we can write
$$I_G = [ GL(B)^{\vert \mathcal{K} \vert} \times I_G' ]/ ((\mathbb{C}^{\times})^{\vert \mathcal{K} \vert})_{\antidiag}$$
for some subgroup $I_G'$ of $GL(L)^{\vert \mathcal{J} \vert \cdot \vert \mathcal{K} \vert}$ containing $(\mathbb{C}^{\times})^{\vert \mathcal{K} \vert}$. By definition of $(I_G)_d$, we have that $I_G' \subseteq (I_G)_d$. Additionally, by \cite[Proposition 3.1]{Disconnected}, we have that 
$$(I_G)_d \cap GL(B)^{\vert \mathcal{K} \vert} = Z( GL(B)^{\vert \mathcal{K} \vert} ) = (\mathbb{C}^{\times})^{\vert \mathcal{K} \vert}.$$
It follows that 
$$(I_G)_d = [ (\mathbb{C}^{\times})^{\vert \mathcal{K} \vert} \times I_G' ]/((\mathbb{C}^{\times})^{\vert \mathcal{K} \vert})_{\antidiag} = I_G' \subseteq GL(L)^{ \vert \mathcal{J} \vert \cdot \vert \mathcal{K} \vert },$$
as desired. 
\end{proof}

\subsection{Writing $G_d$ (resp.~$H_d$) in terms of $(I_G)_d$ and $\mathcal{K}$ (resp.~$(I_H)_d$ and $\mathcal{J}$)}\label{subsec:delta-J}

Set $J:= L^2 (\mathcal{J},\mathbb{C})$ and $K:=L^2(\mathcal{K},\mathbb{C})$. By Corollary \ref{cor:inertia-summary}, we see that 
$$U = \left [ F \otimes E \right ]^{\oplus \vert \mathcal{J} \vert} = \Big [ \bigoplus_{k \in \mathcal{K}} R_{j,k} \Big ]^{\oplus \vert \mathcal{J} \vert} = \left [ B \otimes E \otimes L \right ]^{\oplus \vert \mathcal{J} \vert \cdot \vert \mathcal{K} \vert} ,$$
which we can identify with $E \otimes B \otimes J \otimes K \otimes L$. Under this identification, each element of $\mathcal{J}$ (resp.~$\mathcal{K}$) can be viewed as the image in $PGL(J)$ (resp.~$PGL(K)$) of a permutation matrix in $GL(J)$ (resp.~$GL(K)$), and each element of $\hat{\mathcal{J}}$ (resp.~$\hat{\mathcal{K}}$) can be viewed as the image in $PGL(J)$ (resp.~$PGL(K)$) of a diagonal matrix in $GL(J)$ (resp.~$GL(K)$). (See the proof of Proposition \ref{prop:AxA^-simple} to recall the details of such an identification.)

\begin{prop} \label{prop:G_d-and-H_d}
\hspace{1cm}
\begin{itemize}
\item $G_d = (I_G)_d \rtimes_{\mathbb{C}^{\times}} p_K^{-1}(\mathcal{K})$.
\item $H_d = (I_H)_d \rtimes_{\mathbb{C}^{\times}} p_J^{-1}(\mathcal{J})$.
\end{itemize}
\end{prop}

\begin{proof}
Let $t_G \colon \mathcal{K} \times \mathcal{K} \rightarrow I_G/G^{\circ}$ be the 2-cocycle defining our extension
$$1 \rightarrow I_G/G^{\circ} \rightarrow \Gamma \rightarrow \mathcal{K} \rightarrow 1,$$
so that $\gamma_k \gamma_{k'} = t_G(k,k') \gamma_{k+k'}$. Consider a fixed $k \in \mathcal{K}$. Then for all $k' \in \mathcal{K}$ and $g_0 \in G^{\circ}$, 
\begin{align*}
\prescript{\wtilde{\gamma_{k}\gamma_{k'}}}{}{\beta}(g_0) &= \beta ( \wtilde{\gamma_{k}\gamma_{k'}} g_0 (\wtilde{\gamma_{k}\gamma_{k'}})^{-1} ) \\
&= \beta ( s_G(t_G(k,k'), \gamma_{k+k'})^{-1}  \wtilde{t_G(k,k')} \wtilde{\gamma_{k+k'}} g_0 \wtilde{\gamma_{k+k'}}^{-1} \wtilde{t_G(k,k')}^{-1} s_G(t_G(k,k'), \gamma_{k+k'})) \\
&= \beta (\wtilde{t_G(k,k')} \wtilde{\gamma_{k+k'}} g_0 \wtilde{\gamma_{k+k'}}^{-1} \wtilde{t_G(k,k')}^{-1} ) \\
&= \prescript{\wtilde{\gamma_{k+k'}}}{}{\beta}(g_0),
\end{align*}
where in the last step we have used that $\wtilde{t_G(k,k')} \in (I_G)_d \subseteq GL(L)^{\vert \mathcal{J} \vert \cdot \vert \mathcal{K} \vert}$ (by Proposition \ref{prop:I_G-and-I_H}). From this, we see that $\wtilde{\gamma_k} g_0 \wtilde{\gamma_k}^{-1} = P_k^{-1} g_0 P_k$, where $P_k$ is the block permutation sending $\prescript{\wtilde{\gamma_{k'}}}{}{B}$ to $\prescript{\wtilde{\gamma_{k+k'}}}{}{B}$ for all $k' \in \mathcal{K}$. In particular,
$$P_k \wtilde{\gamma_k} \in Z_{ GL(F)^{\vert \mathcal{J} \vert } } \left ( GL(B)^{\vert \mathcal{K} \vert} ) \right ) = GL(L)^{\vert \mathcal{J} \vert \cdot \vert \mathcal{K} \vert}.$$
It follows that $\wtilde{\gamma_k}$ can be written as $\wtilde{\gamma_k} = M_k P_k^{-1}$ for some element $M_k \in GL(L)^{\vert \mathcal{J} \vert \cdot \vert \mathcal{K} \vert}$. 

Now, notice that
$$(I_G)_d = \{ g \in I_G \, : \, \text{Ad}(g) \text{ is distinguished} \} = I_G \cap G_d = I_{G_d}.$$
Therefore, we have that $\mathcal{K} = G/I_G = ([G^{\circ} \rtimes G_d]/Z(G^{\circ})_{\antidiag})/([G^{\circ} \rtimes (I_G)_d]/Z(G^{\circ})_{\antidiag}) = G_d/(I_G)_d$. In this way, we see that we have the following short exact sequence:
$$1 \rightarrow (I_G)_{d} \rightarrow G_{d} \rightarrow \mathcal{K} \rightarrow 1.$$
In particular, any $g \in G_d$ can be written as $g = g_I \wtilde{\gamma_k}$ for some $g_I \in (I_G)_d$ and some $k \in \mathcal{K}$. Therefore, since $(I_G)_d \subseteq GL(L)^{\vert \mathcal{J} \vert \cdot \vert \mathcal{K} \vert}$ (by Proposition \ref{prop:I_G-and-I_H}) and since $\wtilde{\gamma_k} \in GL(L)^{\vert \mathcal{J} \vert \cdot \vert \mathcal{K} \vert} \rtimes_{\mathbb{C}^{\times}} p_K^{-1}(\mathcal{K})$ (by what we just showed above), we see that $G_d \subseteq GL(L)^{\vert \mathcal{J} \vert \cdot \vert \mathcal{K} \vert} \rtimes_{\mathbb{C}^{\times}} p_K^{-1}(\mathcal{K})$. Therefore, we can write $G_d = G_d' \rtimes_{\mathbb{C}^{\times}} p_K^{-1}(\mathcal{K})$ for some subgroup $G_d'$ of $GL(L)^{\vert \mathcal{J} \vert \cdot \vert \mathcal{K} \vert}$ containing $\mathbb{C}^{\times}$. Then since $(I_G)_d \subseteq G_d \cap GL(L)^{\vert \mathcal{J} \vert \cdot \vert \mathcal{K} \vert}$, we see that $(I_G)_d \subseteq G_d'$. On the other hand, since $G_d' \subseteq G_d \cap GL(L)^{\vert \mathcal{J} \vert \cdot \vert \mathcal{K} \vert}$, we see that $G_d' \subseteq (I_G)_d$ (by definition of $(I_G)_d$). This completes the proof. 
\end{proof}

\begin{cor} \label{cor:Gamma=I_GxK}
\hspace{.5cm}
\begin{itemize}
    \item $\Gamma = I_G/G^{\circ} \times \mathcal{K}$.
    \item $\hat{\Gamma} = I_H/H^{\circ} \times \mathcal{J}$.
\end{itemize}
\end{cor}

\begin{proof}
By Proposition \ref{prop:bijection-extensions}, we have that $G = [G^{\circ} \rtimes G_d]/Z(G^{\circ})_{\antidiag}$, where $G^{\circ} \cap G_d = Z(G^{\circ})$ by \cite[Proposition 3.1]{Disconnected}. Therefore, we see that $\Gamma = G/G^{\circ} = G_d/Z(G^{\circ})$. By Proposition \ref{prop:G_d-and-H_d}, we see that $G_d/Z(G^{\circ}) = (I_G)_d/Z(G^{\circ}) \times \mathcal{K}$. Finally, applying Proposition \ref{prop:I_G-and-I_H}, we see that $(I_G)_d/Z(G^{\circ}) = I_G/G^{\circ}$, completing the proof. 
\end{proof}

\subsection{Defining $\Omega_{j,k}$ and $\Lambda_{j,k}$} \label{subsec:PGL(L)-dual-pairs}

Notice that by Proposition \ref{prop:G_d-and-H_d}, we have
\begin{equation} \label{eq:rho_jk(I_G)}
\rho_{j,k} (I_G) = GL(B) \times_{\mathbb{C}^{\times}} (I_G)_d \vert_{L_{j,k}} \hspace{.5cm} \text{ and } \hspace{.5cm} \psi_{j,k}(I_H) = GL(E) \times_{\mathbb{C}^{\times}} (I_H)_d \vert_{L_{j,k}} 
\end{equation} 
for all $j \in \mathcal{J}$ and $k \in \mathcal{K}$. 

Now, since $I_G^{\circ} = G^{\circ} = GL(B)^{\vert \mathcal{K} \vert}$, we see that $\rho_{j,k}(I_G)^{\circ} = GL(B)$. Therefore, (\ref{eq:rho_jk(I_G)}) gives that
$$\rho_{j,k}(I_G) / GL(B) = (I_G)_d \vert_{L_{j,k}} / \mathbb{C}^{\times}  =: \Omega_{j,k},$$ 
which can be viewed as a finite subgroup of $PGL(L_{j,k})$. Similarly, $(\psi_{j,k}(I_H))^{\circ} = GL(E)$, and 
$$ \psi_{j,k}(I_H)/GL(E) = (I_H)_d \vert_{L_{j,k}} / \mathbb{C}^{\times} =: \Lambda_{j,k} ,$$
where $\Lambda_{j,k}$ can be viewed as a subgroup of $PGL(L_{j,k})$. With this notation, we see that we have the following short exact sequences associated to each pair $j,k$:
$$1 \rightarrow GL(B) \rightarrow \rho_{j,k}(I_G) \rightarrow \Omega_{j,k} \rightarrow 1 \hspace{.25cm} \text{ and } \hspace{.25cm} 1 \rightarrow GL(E) \rightarrow \psi_{j,k} (I_H) \rightarrow \Lambda_{j,k} \rightarrow 1.$$

\subsection{Showing that $\Omega_{j,k}$ and $\Lambda_{j,k}$ are dual finite abelian groups} \label{subsec:Omega-Lambda-duality}

\begin{lemma} \label{lem:dual-inertia-comp}
For all $j \in \mathcal{J}$ and $k \in \mathcal{K}$, $\Omega_{j,k}$ and $\Lambda_{j,k}$ are finite abelian groups with natural isomorphisms 
$$\Omega_{j,k} \simeq \hat{\Lambda}_{j,k} \hspace{.5cm} \text{ and } \hspace{.5cm} \Lambda_{j,k} \simeq \hat{\Omega}_{j,k}.$$
\end{lemma}

\begin{proof}
We follow the argument from the proof of Theorem \ref{thm:dual-finite-ab-gps}. For notational convenience, set $V := B \otimes E \otimes L$. By Lemma \ref{lem:roots-of-unity}, our usual map 
$$\mu \colon G \times H \rightarrow \{ (\dim U)\text{-th roots of unity} \}$$ 
restricts to a map 
$$\mu \colon \rho_{j,k} (I_G) \times \psi_{j,k}(I_H) \rightarrow \{ (\dim V)\text{-th roots of unity} \}.$$
As in the proof of Theorem \ref{thm:dual-finite-ab-gps}, we see that this map further descends to a map
$$\mu \colon \Omega_{j,k} \times \Lambda_{j,k} \rightarrow \{ (\dim V)\text{-th roots of unity} \}$$ 
arising from a group homomorphism 
\begin{align*}
M \colon \Omega_{j,k} & \rightarrow \hat{\Lambda}_{j,k} \\
\omega & \mapsto \mu (\omega , \cdot)
\end{align*}
We will show that $M$ is in fact an isomorphism. Towards proving injectivity, define $\Omega_{j,k}' := \ker M$. Then $\Omega_{j,k} '$ corresponds to the following subgroup of $\rho_{j,k} (I_G)$:
$$\rho_{j,k}(I_G)' := \{ x \in \rho_{j,k} (I_G) \, : \, xyx^{-1} = y \text{ for all } y \in \psi_{j,k}(I_H) \}.$$
But by the irreducibility of $\psi_{j,k}$, we see that $\rho_{j,k}(I_G)' = GL(B) = (\rho_{j,k}(I_G))^{\circ}$. Therefore, $\ker M = \Omega_{j,k} '$ is trivial, meaning $M$ is injective. Since $\hat{\Lambda}_{j,k}$ is abelian, this gives that $\Omega_{j,k}$ is abelian as well. Switching the roles of $\Omega_{j,k}$ and $\Lambda_{j,k}$ in this argument gives that
\begin{align*}
M' \colon \Lambda_{j,k} & \rightarrow \hat{\Omega}_{j,k} \\
\lambda & \mapsto \mu (\cdot , \lambda)
\end{align*} 
is injective and that $\Lambda_{j,k}$ is abelian. 

Finally, since the dual group functor is a contravariant exact functor for locally compact abelian groups, the injectivity of $M'$ implies the surjectivity of $M$, completing the proof.
\end{proof}

Moving forward, we will identify $\Lambda_{j,k}$ with $\hat{\Omega}_{j,k}$ and will no longer use ``$\Lambda_{j,k}$."

\subsection{Showing that $(\rho_{j,k}(I_G)/\mathbb{C}^{\times}, \psi_{j,k}(I_H)/\mathbb{C}^{\times} )$ is a dual pair in $PGL(B \otimes E \otimes L)$  } \label{subsec:PGL(BxExL)-dual-pair}

We will now use Lemma \ref{lem:dual-inertia-comp} and will again apply the argument from the proof of Theorem \ref{thm:dual-finite-ab-gps} to prove that $(\rho_{j,k}(I_G)/\mathbb{C}^{\times}, \psi_{j,k}(I_H)/\mathbb{C}^{\times})$ is a dual pair in $PGL(B \otimes E \otimes L)$:

\begin{prop} \label{prop:PGL(Ujk)-dps}
For all $j \in \mathcal{J}$ and $k \in \mathcal{K}$, $(\rho_{j,k}(I_G)/\mathbb{C}^{\times}, \psi_{j,k} (I_H)/\mathbb{C}^{\times})$ is a dual pair in $PGL(B \otimes E \otimes L)$.
\end{prop}

\begin{proof}
For notational convenience, set $V := B \otimes E \otimes L$. Let $G_{j,k}$ denote the preimage in $GL(V)$ of $Z_{PGL(V)}(\psi_{j,k}(I_H)/\mathbb{C}^{\times})$. The containment $\rho_{j,k}(I_G)/\mathbb{C}^{\times} \subseteq Z_{PGL(V)} ( \psi_{j,k} (I_H)/\mathbb{C}^{\times} ) $ (i.e.,~$\rho_{j,k}(I_G) \subseteq G_{j,k}$) is clear. Towards containment in the other direction, note that we have a map
\begin{center}
\begin{tabular}{c l l}
$\nu \colon G_{j,k} \times \psi_{j,k}(I_H)$ & $\rightarrow$ & $\{ (\dim V)\text{-th roots of unity} \}$ \\
$(x,y)$ & $\mapsto$ & $xyx^{-1}y^{-1}$
\end{tabular}
\end{center}
Now, since $\nu$ is constant on connected components, we see that 
$$(G_{j,k})^{\circ} \subseteq \{ x \in G_{j,k} \, : \, xyx^{-1} = y \text{ for all } y \in \psi_{j,k}(I_H) \} = GL(B).$$
On the other hand, since $\rho_{j,k}(I_G) \subseteq G_{j,k}$, it is clear that $GL(B) \subseteq (G_{j,k})^{\circ}$. In this way, we see that $(G_{j,k})^{\circ} = GL(B)$. Define $\Omega_{j,k} ' := G_{j,k}/(G_{j,k})^{\circ} = G_{j,k}/GL(B)$. Then we see that $\nu$ descends to a map 
$$\nu \colon \Omega_{j,k}' \times \Lambda_{j,k}  \rightarrow \{ (\dim V)\text{-th roots of unity} \}.$$ 
By the same reasoning as in the proofs of Theorem \ref{thm:dual-finite-ab-gps} and Lemma \ref{lem:dual-inertia-comp}, we see that $\nu$ arises from a group homomorphism
\begin{align*}
N \colon \Omega_{j,k}' & \rightarrow \hat{\Lambda}_{j,k} \\
\omega'  & \mapsto \nu (\omega' , \cdot)
\end{align*} 
We will show that $N$ is in fact an isomorphism. To see that $N$ is injective, note that $\ker N$ corresponds to the identity component 
$$(G_{j,k})^{\circ} = \{ x \in G_{j,k} \, : \, xyx^{-1} = y \text{ for all } y \in \psi_{j,k}(I_H) \} = GL(B)$$
of $G_{j,k}$. For surjectivity, it suffices to show that the following homomorphism is injective:
\begin{align*}
N' \colon \Lambda_{j,k} & \rightarrow \hat{\Omega'}_{j,k} \\
\lambda & \mapsto \nu (\cdot , \lambda)
\end{align*}
To this end, note that $\ker N'$ corresponds to the following subgroup of $\psi_{j,k}(I_H)$:
$$\{ y \in \psi_{j,k}(I_H) \, : \, xyx^{-1} = y \text{ for all } x \in G_{j,k} \} = Z_{GL(V)}(G_{j,k}) \cap \psi_{j,k}(I_H).$$
Here, $Z_{GL(V)} (G_{j,k}) \subseteq Z_{GL(V)} (\rho_{j,k}(I_G)) = GL(E)$. Moreover, since $Z_{PGL(V)} ( \psi_{j,k} (I_H)/\mathbb{C}^{\times} ) \subseteq Z_{PGL(V)} ( GL(E)/\mathbb{C}^{\times} )$, we see that $G_{j,k} \subseteq GL(B \otimes L)$, and hence that $Z_{GL(V)} (G_{j,k}) \supseteq Z_{GL(V)} (B \otimes L) = GL(E)$. In this way, we see that $\ker N'$ corresponds to 
$$ Z_{GL(V)} (G_{j,k})\cap \psi_{j,k}(I_H) = GL(E),$$
and hence that $N'$ is injective. By the same argument as in the proof of Theorem \ref{thm:dual-finite-ab-gps}, this shows that $N$ is surjective, and hence an isomorphism. 

Finally, applying Lemma \ref{lem:dual-inertia-comp}, we see that  
$$\rho_{j,k}(I_G)/GL(B) = \Omega_{j,k} \simeq \hat{\Lambda}_{j,k} \simeq \Omega_{j,k}' = G_{j,k}/ GL(B).$$
From this, it is clear that $\rho_{j,k}(I_G) = G_{j,k}$, as desired. Applying the same argument with the roles of $\rho_{j,k}(I_G)$ and $\psi_{j,k}(I_H)$ reversed completes the proof.
\end{proof}

\subsection{Showing that $\Omega_{j,k} = \Omega_{j',k'} =: \Omega$ for all $j,j' \in \mathcal{J}$ and $k,k' \in \mathcal{K}$} \label{subsec:omegas-are-all-the-same}

\begin{lemma} \label{lem:omegas-are-all-the-same}
For all $j \in \mathcal{J}$ and $k \in \mathcal{K}$, we have natural isomorphisms $\Omega_{j,k} \simeq \Omega_{0,0}$ and $\hat{\Omega}_{j,k} \simeq \hat{\Omega}_{0,0}$. 
\end{lemma}

\begin{proof}
By Corollary \ref{cor:inertia-reps-jk}, we have that $\rho_{j,k}(g_I) \cdot GL(B) = \mu (g_I, \delta_j)^{-1} \cdot \rho_{0,k}(g_I) \cdot GL(B) = \rho_{0,k}(g_I) \cdot GL(B)$ for all $g_I \in I_G$, $j \in \mathcal{J}$, and $k \in \mathcal{K}$. Therefore, we obtain a natural isomorphism $\Omega_{j,k} \xrightarrow{\sim} \Omega_{0,k}$ via $\rho_{j,k}(g_I) \cdot GL(B) \mapsto \rho_{0,k}(g_I) \cdot GL(B)$. Applying Lemma \ref{lem:dual-inertia-comp} and this same argument, we see that there is also a natural isomorphism $\hat{\Omega}_{j,k} \simeq \hat{\Omega}_{j,0}$. Using that there is a natural isomorphism between a group and its double dual, we further get that there are natural isomorphisms
$$\Omega_{j,k} \simeq \Omega_{0,k} \simeq \hat{ \hat{\Omega}}_{0,k}  \simeq \hat{ \hat{\Omega}}_{0,0}  = \Omega_{0,0}.$$
Similarly, $\hat{\Omega}_{j,k} \simeq \hat{\Omega}_{0,0}$, completing the proof. 
\end{proof}

Moving forward, we identify $\Omega_{j,k}$ with $\Omega := \Omega_{0,0}$ for all $j \in \mathcal{J}$ and $k \in \mathcal{K}$.

\subsection{Showing that $(\Omega, \hat{\Omega})$ is a dual pair in $PGL(L)$, and that $\Omega = \hat{\Omega} = \mathcal{L} \times \hat{\mathcal{L}}$ for a finite abelian group $\mathcal{L}$} \label{subsec:omega-omega-hat-dual-pair}

\begin{cor} \label{cor:reducing-to-PGL(L)-dp}
$( \Omega, \hat{\Omega} )$ is a dual pair in $PGL(L)$.
\end{cor}

\begin{proof}
This follows immediately from (\ref{eq:rho_jk(I_G)}) and Proposition \ref{prop:PGL(Ujk)-dps}.
\end{proof}

\newpage

\begin{prop} \label{prop:Omega=Omega-hat}
$\Omega = \hat{\Omega}$.
\end{prop}

\begin{proof}
For $\gamma_1, \gamma_2 \in I_G/G^{\circ}$, we have that
$$\wtilde{\gamma_1} \wtilde{\gamma_2} = s_G(\gamma_1, \gamma_2) \wtilde{\gamma_1 \gamma_2} = s_G(\gamma_1, \gamma_2) s_G(\gamma_2, \gamma_1)^{-1} \wtilde{\gamma_2} \wtilde{\gamma_1},$$
where $s_G$ takes values in $Z(G^{\circ})$. Restricting to $L_{j,k}$ for any $j \in \mathcal{J}$ and $k \in \mathcal{K}$ shows that
$$\Omega = \Omega_{j,k} \subseteq Z_{PGL(L_{j,k})} (\Omega_{j,k}) = \hat{\Omega}_{j,k} = \hat{\Omega},$$
where we have used Lemma \ref{lem:omegas-are-all-the-same} and Corollary \ref{cor:reducing-to-PGL(L)-dp}. Repeating this process with $\hat{\Omega}$ in place of $\Omega$, we get containment in the other direction. It follows that $\Omega = \hat{\Omega}$, as desired.
\end{proof}

\begin{prop} \label{prop:omega=LxL}
$\Omega$ can be written as $\Omega = \mathcal{L} \times \hat{\mathcal{L}}$ for some finite abelian group $\mathcal{L}$.
\end{prop}

\begin{proof}
For each $\omega \in \Omega$, pick some $\wideparen{\omega} \in GL(L)$ so that $p_L(\wideparen{\omega}) = \omega$. Then by Lemma \ref{lem:roots-of-unity}, Corollary \ref{cor:reducing-to-PGL(L)-dp}, and Proposition \ref{prop:Omega=Omega-hat}, we see that we have a non-degenerate symplectic form
\begin{align*}
\mu_{\Omega} \colon \Omega \times \Omega & \rightarrow \{ (\dim L)\text{-th roots of unity} \} \\
(\omega_1, \omega_2) & \mapsto \wideparen{\omega}_1  \wideparen{\omega}_2  \wideparen{\omega}_1^{-1} \wideparen{\omega}_2^{-1} 
\end{align*}

Let $\lambda$ be an element of maximal order (say $r$) in $\Omega$. Since $\mu_{\Omega}$ is nondegenerate, we see that there exists $\lambda' \in \Omega$ (with order $r$) such that $\wideparen{\lambda} \,  \wideparen{\lambda'} \, \wideparen{\lambda}^{-1} \, \wideparen{\lambda'}^{-1}$ equals a primitive $r$-th root of unity. 

We claim that $\mu_{\Omega} \vert_{\langle \lambda \rangle \times \langle \lambda' \rangle }$ is nondegenerate. To see this, suppose that $\lambda^{r_1} (\lambda')^{r_1'} \in \langle \lambda \rangle \times \langle \lambda ' \rangle$ is such that $\mu_{\Omega} ( \lambda^{r_1} (\lambda ' )^{r_1'}, \lambda^{r_2} (\lambda')^{r_2'} ) = 1$ for all $\lambda^{r_2} (\lambda ' )^{r_2'} \in \langle \lambda \rangle \times \langle \lambda ' \rangle $. Then we get that
\begin{align*}
\mu_{\Omega} ( \lambda^{r_1} (\lambda ' )^{r_1'}, \lambda^{r_2} (\lambda')^{r_2'} ) &= \wideparen{( \lambda^{r_1} (\lambda')^{r_1'} )} \wideparen{ ( \lambda^{r_2} (\lambda ' )^{r_2'} ) } ( \wideparen{\lambda^{r_1} (\lambda')^{r_1'} }  )^{-1} ( \wideparen{\lambda^{r_2} (\lambda')^{r_2'} }  )^{-1} \\
&= (\wideparen{\lambda})^{r_1} ( \wideparen{\lambda'} )^{r_1'} (\wideparen{\lambda})^{r_2} (\wideparen{\lambda'})^{r_2'} (\wideparen{\lambda'})^{-r_1'} (\wideparen{\lambda})^{-r_1} (\wideparen{\lambda'})^{-r_2'} (\wideparen{\lambda})^{-r_2} \\
&= \zeta^{r_1' r_2 - r_1r_2'} = 1
\end{align*} 
for all $0 \leq r_2, r_2' \leq r$, where $\zeta$ denotes a primitive $r$-th root of unity. This clearly forces $r_1 = r_1' = 0$, which proves that $\mu_{\Omega} \vert_{\langle \lambda \rangle \times \langle \lambda ' \rangle}$ is nondegenerate. Letting $( \langle \lambda \rangle \times \langle \lambda' \rangle )^{\perp}$ denote the orthogonal complement of $\langle \lambda \rangle \times \langle \lambda' \rangle$ with respect to $\mu_{\Omega}$, we therefore see that 
$$ \langle \lambda \rangle \times \langle \lambda' \rangle \cap (\langle \lambda \rangle \times \langle \lambda' \rangle)^{\perp} = \{ 1 \} .$$
Since $\Omega$ is generated by $\langle \lambda \rangle \times \langle \lambda ' \rangle$ and $(\langle \lambda \rangle \times \langle \lambda ' \rangle)^{\perp}$ (which can be seen by noting that $(\langle \lambda \rangle \times \langle \lambda ' \rangle)^{\perp}$ is the kernel of the map $\hat{\Omega} \rightarrow \hat{ ( \langle \lambda \rangle \times \langle \lambda ' \rangle )}$ defined by $\omega \mapsto \omega \vert_{\langle \lambda \rangle \times \langle \lambda ' \rangle }$), we furthermore get that
$$\Omega = \langle \lambda \rangle \times \langle \lambda ' \rangle \times (\langle \lambda \rangle \times \langle \lambda ' \rangle)^{\perp} = \langle \lambda \rangle \times \hat{\langle \lambda \rangle} \times (\langle \lambda \rangle \times \hat{\langle \lambda \rangle})^{\perp}.$$ 
Now, since $\mu_{\Omega}$ is nondegenerate on $\Omega$ and on $\langle \lambda \rangle \times \hat{\langle \lambda \rangle}$, we see that $\mu_{\Omega}$ is nondegenerate on $\Omega ' := ( \langle \lambda \rangle \times \hat{\langle \lambda \rangle} )^{\perp}$. We can therefore iterate this process (replacing $\Omega$ with $\Omega'$), which proves the result.
\end{proof}

\subsection{Understanding $(I_G)_d$ and $(I_H)_d$ explicitly} \label{subsec:(I_G)_d-and-(I_H)_d}

\begin{prop} \label{prop:(I_G)_d-and-(I_H)_d}
\hspace{1cm}
\begin{itemize}
\item $(I_G)_d =  p_L^{-1}(\Omega) \times_{\mathbb{C}^{\times}} (\mathbb{C}^{\times})^{\vert \mathcal{K} \vert} \times_{\mathbb{C}^{\times}} p_J^{-1}(\hat{\mathcal{J}}) $.
\item $(I_H)_d = p_L^{-1}(\Omega) \times_{\mathbb{C}^{\times}} (\mathbb{C}^{\times})^{\vert \mathcal{J} \vert} \times_{\mathbb{C}^{\times}} p_K^{-1}(\hat{\mathcal{K}})  $.
\end{itemize}
\end{prop}

\begin{proof}
We first show that $(I_G)_d \subseteq p_L^{-1}(\Omega) \times_{\mathbb{C}^{\times}} (\mathbb{C}^{\times})^{\vert \mathcal{K} \vert} \times_{\mathbb{C}^{\times}} p_J^{-1}(\hat{\mathcal{J}})$. To this end, let $g \in (I_G)_d \subseteq GL(L)^{\vert \mathcal{J} \vert \cdot \vert \mathcal{K} \vert}$. Then for all $j \in \mathcal{J}$ and $k \in \mathcal{K}$, $\rho_{j,k}(g) \in GL(L)$. We claim that for any $k \in \mathcal{K}$, $\rho_{0,k}(g)$ and $\rho_{0,0}(g)$ have the same image in $\Omega \subset PGL(L)$. To see this, note that for any $h \in (I_H)_d$, we have 
\begin{align*}
\mu ( g, h ) &= \rho_{0,k} (g) \psi_{0,k}(h) \rho_{0,k}(g)^{-1} \psi_{0,k}(h)^{-1} \\
&= \rho_{0,k} (g) \mu (\gamma_k, h) \psi_{0,0}(h) \rho_{0,k} (g)^{-1} \psi_{0,0}(h)^{-1} \mu (\gamma_k, h)^{-1} \\
&= \rho_{0,k}(g) \psi_{0,0} (h) \rho_{0,k}(g)^{-1} \psi_{0,0} (h)^{-1},
\end{align*}
where we have used Corollary \ref{cor:inertia-reps-jk}. On the other hand,  we have 
$$ \mu (g,h) = \rho_{0,0}(g) \psi_{0,0}(h) \rho_{0,0} (g)^{-1} \psi_{0,0}(h)^{-1}.$$
Therefore, since $h \in (I_H)_d$ was chosen arbitrarily, we see that $\rho_{0,k}(g)$ and $\rho_{0,0}(g)$ define the same linear functional on $\hat{\Omega}_{0,0}$. In other words, $\rho_{0,k}(g)$ and $\rho_{0,0}(g)$ have the same image in $\Omega \subset PGL(L)$, and therefore differ only by a scalar in $GL(L)$.

Next, by Corollary \ref{cor:inertia-reps-jk}, we have that $\rho_{j,k}(g) = \mu (g,\delta_j)^{-1} \rho_{0,k} (g)$ for all $j \in \mathcal{J}$ and $k \in \mathcal{K}$. Putting this all together and recalling that $g \in (I_G)_d$ gets identified with 
$$\diag ( \rho_{0,0}(g), \ldots , \rho_{j,k}(g), \ldots ) \in GL(L)^{\vert \mathcal{J} \vert \cdot \vert \mathcal{K} \vert},$$ 
we see that 
$$(I_G)_d \subseteq  p_L^{-1}(\Omega) \times_{\mathbb{C}^{\times}} (\mathbb{C}^{\times})^{\vert \mathcal{K} \vert} \times_{\mathbb{C}^{\times}} p_J^{-1}(\hat{\mathcal{J}}) .$$   
Similarly, we see that 
$$(I_H)_d \subseteq  p_L^{-1}(\Omega) \times_{\mathbb{C}^{\times}} (\mathbb{C}^{\times})^{\vert \mathcal{J} \vert} \times_{\mathbb{C}^{\times}} p_K^{-1}(\hat{\mathcal{K}}) . $$
Therefore,
\begin{align*}
H &= [GL(E)^{\vert \mathcal{J} \vert} \rtimes H_d]/((\mathbb{C}^{\times})^{\vert \mathcal{J} \vert})_{\antidiag} \\
&\subseteq \left [ GL(E)^{\vert \mathcal{J} \vert} \rtimes \left ( [  p_L^{-1}(\Omega) \times_{\mathbb{C}^{\times}}   (\mathbb{C}^{\times})^{\vert \mathcal{J} \vert} \times_{\mathbb{C}^{\times}} p_K^{-1}(\hat{\mathcal{K}})  ] \rtimes_{\mathbb{C}^{\times}} p_J^{-1}(\mathcal{J})  \right ) \right ]/((\mathbb{C}^{\times})^{\vert \mathcal{J} \vert})_{\antidiag},  
\end{align*}
where we have used Propositions \ref{prop:bijection-extensions} and \ref{prop:G_d-and-H_d}. Since $ p_L^{-1}(\Omega) \times_{\mathbb{C}^{\times}} (\mathbb{C}^{\times})^{\vert \mathcal{K} \vert} \times_{\mathbb{C}^{\times}} p_J^{-1}(\hat{\mathcal{J}}) $ commutes (up to a scalar) with this latter expression, we see that 
$$ p_L^{-1}(\Omega) \times_{\mathbb{C}^{\times}} (\mathbb{C}^{\times})^{\vert \mathcal{K} \vert} \times_{\mathbb{C}^{\times}} p_J^{-1}(\hat{\mathcal{J}})  \subseteq G \cap GL(L)^{\vert \mathcal{J} \vert \cdot \vert \mathcal{K} \vert} = (I_G)_d, $$
completing the proof.
\end{proof}

\begin{cor} \label{cor:I_G/G0=JxOmega}
Let $\mathcal{L}$ be as in Proposition \ref{prop:omega=LxL}. Then we have the following:
\begin{itemize}
\item $\Gamma = I_G/G^{\circ} \times \mathcal{K} = \mathcal{L} \times \hat{\mathcal{L}} \times \hat{\mathcal{J}} \times \mathcal{K} $.
\item $\hat{\Gamma} = I_H/H^{\circ} \times \mathcal{J} = \mathcal{L} \times \hat{\mathcal{L}} \times \mathcal{J} \times \hat{\mathcal{K}} $.
\end{itemize}
\end{cor}

\begin{proof}
Combining several of the above results, we get that
\begin{center}
    \begin{tabular}{rclcl}
    $\Gamma$ & $=$ & $I_G/G^{\circ} \times \mathcal{K}$ & \hspace{.25cm} &(Corollary \ref{cor:Gamma=I_GxK}) \\
     & $=$ & $(I_G)_d/Z(G^{\circ}) \times \mathcal{K}$ & \hspace{.25cm} &(Proposition \ref{prop:I_G-and-I_H}) \\
     & $=$ & $\Omega \times \hat{\mathcal{J}} \times \mathcal{K} $ & \hspace{.25cm} &(Proposition \ref{prop:(I_G)_d-and-(I_H)_d}) \\
     & $=$ & $\mathcal{L} \times \hat{\mathcal{L}} \times \hat{\mathcal{J}} \times \mathcal{K} $ & \hspace{.25cm} &(Proposition \ref{prop:omega=LxL}),
\end{tabular}
\end{center}
as desired.
\end{proof}

\begin{cor} \label{cor:|L|=dimL}
Let $\mathcal{L}$ be as in Proposition \ref{prop:omega=LxL}. Then $\vert \mathcal{L} \vert = \dim L$.
\end{cor}

\begin{proof}
    By Corollaries \ref{cor:G,H-restricted-reps} and \ref{cor:inertia-summary}, we have that
    \begin{align*}
        \Res^H_{H^{\circ}} \left ( \Ind_{H^{\circ}}^{H} (\varepsilon) \right ) &= \Res^H_{H^{\circ}} \left ( \Ind_{I_H }^{H} \left ( \Ind_{H^{\circ}}^{I_H} (\varepsilon) \right ) \right ) = \Res^H_{H^{\circ}} \left ( \Ind_{I_H }^{H} \left ( \bigoplus_{k\in \mathcal{K}} \psi_{0,k} \otimes L^* \right ) \right ) \\
        &= \Res^H_{H^{\circ}} \left ( \bigoplus_{k \in \mathcal{K}} \prescript{\gamma_k}{}{\alpha} \otimes L^* \right ) = \Big [ \bigoplus_{j \in \mathcal{J}} \prescript{\wtilde{\delta_j}}{}{\varepsilon} \otimes L \otimes L^* \Big ]^{\oplus \vert \mathcal{K} \vert} \\
        &\simeq  \Big [ \bigoplus_{j \in \mathcal{J}} \prescript{\wtilde{\delta_j}}{}{\varepsilon} \Big ]^{\oplus (\dim L)^2 \cdot \vert \mathcal{K} \vert}.
    \end{align*}
    On the other hand, by \cite[Theorem 2.1(1)]{Clifford}, we have that
    $$\Res^H_{H^{\circ}} \left ( \Ind_{H^{\circ}}^{H} (\varepsilon) \right ) \simeq  \Big [ \bigoplus_{j \in \mathcal{J}} \prescript{\wtilde{\delta_j}}{}{\varepsilon} \Big ]^{\oplus \vert I_H/H^{\circ} \vert}.$$
    Using Corollary \ref{cor:I_G/G0=JxOmega}, it follows that
    $$(\dim L)^2 \cdot \vert \mathcal{K} \vert = \vert I_H/H^{\circ} \vert = \vert \Gamma \vert/\vert \mathcal{J} \vert = \vert \mathcal{K} \vert \cdot \vert \mathcal{L} \vert^2,$$
    and hence that $\vert \mathcal{L} \vert = \dim L$.
\end{proof}

\subsection{Understanding $G$, $G_d$, $H$, and $H_d$ explicitly} \label{subsec:G_d-and-H_d}

\begin{thm} \label{thm:G_d}
Let $\mathcal{L}$ be as in Proposition \ref{prop:omega=LxL}. Then we have the following:
\begin{itemize}
\item $G_d = \left [ (\mathbb{C}^{\times})^{\vert \mathcal{K} \vert} \times_{\mathbb{C}^{\times}} p_L^{-1} (\mathcal{L} \times \hat{\mathcal{L}}) \times_{\mathbb{C}^{\times}} p_J^{-1}(\hat{\mathcal{J}}) \right ] \rtimes_{\mathbb{C}^{\times}} p_K^{-1}(\mathcal{K})$.
\item $H_d = \left [ (\mathbb{C}^{\times})^{\vert \mathcal{J} \vert} \times_{\mathbb{C}^{\times}} p_L^{-1} (\mathcal{L} \times \hat{\mathcal{L}}) \times_{\mathbb{C}^{\times}} p_K^{-1}(\hat{\mathcal{K}}) \right ] \rtimes_{\mathbb{C}^{\times}} p_J^{-1}(\mathcal{J})$.
\end{itemize}
Therefore, $G$ and $H$ are of the form described in Theorem \ref{thm:single-orbit-general}.
\end{thm}

\begin{proof}
The expressions for $G_d$ and $H_d$ follow immediately from Propositions \ref{prop:G_d-and-H_d}, \ref{prop:omega=LxL}, and \ref{prop:(I_G)_d-and-(I_H)_d}. The conclusion that $G$ and $H$ are of the form described in Theorem \ref{thm:single-orbit-general} follows from Proposition \ref{prop:bijection-extensions}.
\end{proof}

\section{Classifying ``multi-orbit" dual pairs in $PGL(U)$} \label{sec:single-orbit}

In this section, we first construct a class of ``multi-orbit" dual pairs by ``gluing" together a set of single-orbit dual pairs whose component groups are all isomorphic (Subsection \ref{subsec:general-construction}). We then show that an arbitrary $PGL(U)$ dual pair is necessarily of this form (Subsection \ref{subsec:general-showing-we've-found-all}). Note that the notation in this section is analogous to the notation introduced in Sections \ref{sec:prelims-for-showing-all}--\ref{sec:explicit} (albeit with more indices), and hence will not be included in the Index of Notation.

\subsection{Constructing ``multi-orbit" dual pairs in $PGL(U)$} \label{subsec:general-construction}

For $1 \leq i \leq r$, let $U_i$ be a finite-dimensional complex vector space, and let $(G_i, H_i)$ be a pair of subgroups of $GL(U_i)$ that descend to a single-orbit dual pair in $PGL(U_i)$ (of the form described in Theorem \ref{thm:single-orbit-general}). Set $\Gamma_i := G_i / G_i^{\circ}$, so that $\hat{\Gamma}_i = H_i / H_i^{\circ}$. For each $1 \leq i \leq r$, define
\begin{align*}
    \mu_i \colon G_i \times H_i & \rightarrow \{ (\dim U_i)\text{-th roots of unity}\} \\
    (g_i,h_i) & \mapsto g_i h_i g_i^{-1} h_i^{-1},
\end{align*}
which descends to a map on $\Gamma_i \times \hat{\Gamma}_i$. Suppose that for each $1 \leq i , i' \leq r$, we have isomorphisms
$$q_{i,i'} \colon \Gamma_i \xrightarrow{\sim} \Gamma_{i'} \hspace{.5cm} \text{ and } \hspace{.5cm} u_{i,i'} \colon \hat{\Gamma}_i \xrightarrow{\sim} \hat{\Gamma}_{i'}$$
such that 
$$\mu_i (\gamma_i , \delta_i) = \mu_{i'} ( q_{i,i'}(\gamma_i) , u_{i,i'}(\delta_i) ) $$ 
for all $\gamma_i \in \Gamma_i$ and $\delta_i \in \hat{\Gamma}_i$, and such that 
$$q_{i',i''} \circ q_{i,i'} = q_{i,i''} \hspace{.5cm} \text{ and } \hspace{.5cm} u_{i',i''} \circ u_{i,i'} = u_{i,i''}$$ 
for all $1 \leq i , i' , i'' \leq r$ (with $q_{i,i}$ and $u_{i,i}$ equalling the identity operators). 

Note that $q_{i,i'}$ determines $u_{i,i'}$. Indeed, with a fixed choice of $q_{i,i'}$, we know that for all $\gamma_i \in \Gamma_i$ and $\delta_i \in \hat{\Gamma}_i$, we have 
$$\mu_{i'} ( q_{i,i'}(\gamma_i), u_{i,i'} (\delta_i) ) = \mu_i (\gamma_i, \delta_i).$$
Letting $\gamma_i$ vary over $\Gamma_i$, we see that each $u_{i,i'}(\delta_i)$ is uniquely determined. The composition requirement therefore gives that shows that the full collection of isomorphisms $\{ q_{i,i'}, u_{i,i'} \}_{1 \leq i \leq i' \leq r}$ is determined by (for example) the set $\{ q_{1,i} \}_{1 \leq i \leq r}$.

Consider the following subgroups of $\Gamma_1 \times \cdots \times \Gamma_r$ and $\hat{\Gamma}_1 \times \cdots \times \hat{\Gamma}_r$:
$$\Gamma := \{ ( \gamma_1, \, q_{1,2}(\gamma_1), \, q_{1,3}(\gamma_1), \, \ldots, \, q_{1,r}(\gamma_1) ) \}_{\gamma_1 \in \Gamma_1} \subset \Gamma_1 \times \cdots \times \Gamma_r$$
and
$$\hat{\Gamma} := \{ ( \delta_1, \, u_{1,2}(\delta_1), \, u_{1,3}(\delta_1), \, \ldots, \, u_{1,r}(\delta_1) ) \}_{\delta_1 \in \hat{\Gamma}_1} \subset \hat{\Gamma}_1 \times \cdots \times \hat{\Gamma}_r.$$
Let $G$ and $H$ be the corresponding subgroups of $G_1 \times \cdots \times G_r$ and $H_1 \times \cdots \times H_r$, respectively. Set $U := \bigoplus_{1 \leq i \leq r} U_i$. Then we have the following:

\begin{thm} \label{thm:general-construction}
Let $(G,H)$ be as constructed above. Then $(p(G), p(H))$ is a dual pair in $PGL(U)$.
\end{thm}

\begin{proof}
Recall from the construction of $(G,H)$ that $(p(G_i),p(H_i))$ is a $PGL(U_i)$ dual pair for each $1 \leq i \leq r$, and that $G$ and $H$ are the subgroups of $G_1 \times \cdots \times G_r$ and $H_1 \times \cdots \times H_r$ corresponding to  
$$\Gamma := \{ ( \gamma_1, \, q_{1,2}(\gamma_1), \, q_{1,3}(\gamma_1), \, \ldots, \, q_{1,r}(\gamma_1) ) \}_{\gamma_1 \in \Gamma_1} \subset \Gamma_1 \times \cdots \times \Gamma_r$$
and
$$\hat{\Gamma} := \{ ( \delta_1, \, u_{1,2}(\delta_1), \, u_{1,3}(\delta_1), \, \ldots, \, u_{1,r}(\delta_1) ) \}_{\delta_1 \in \hat{\Gamma}_1} \subset \hat{\Gamma}_1 \times \cdots \times \hat{\Gamma}_r.$$
Therefore, for any $g=(g_1,\ldots,g_r) \in G$ and $h=(h_1,\ldots,h_r) \in H$, we have that
$$\mu_i (g_i,h_i) = \mu_i ( q_{1,i}( g_1 \cdot G_1^{\circ} ), \, u_{1,i} (h_1 \cdot H_1^{\circ}) ) = \mu_1 ( g_1, h_1) $$
for all $1 \leq i \leq r$. Therefore, it is clear that 
$$p(H) \subseteq Z_{PGL(U)} ( p(G) ) \hspace{.5cm} \text{ and } \hspace{.5cm} p(G) \subseteq Z_{PGL(U)} ( p(H) ) .$$
It remains to show the containments in the other direction. To this end, notice that
\begin{align*}
    p^{-1}(Z_{PGL(U)}(p(H))) \subseteq Z_{GL(U)} (H^{\circ}) &= Z_{GL(U)} ( H_1^{\circ} \times \cdots \times H_r^{\circ} ) \\
    &= Z_{GL(U_1)} (H_1^{\circ}) \times \cdots \times Z_{GL(U_r)}(H_r^{\circ}) \\
    &\subseteq GL(U_1) \times \cdots \times GL(U_r).
\end{align*}
From this, we see that 
\begin{align*}
    p^{-1}(Z_{PGL(U)}(p(H))) & \subseteq p_{U_1}^{-1}( Z_{PGL(U_1)}(p(H_1)) ) \times \cdots \times p_{U_r}^{-1}( Z_{PGL(U_r)}(p(H_r)) ) \\
    &= G_1 \times \cdots \times G_r.
\end{align*}
We can therefore write $p^{-1}(Z_{PGL(U)}(p(H)))$ as 
\[ \left \{
\begin{tabular}{c|c}
     & $\mu_i (g_i,h_i) = \mu_{i'} (g_{i'}, h_{i'})$  \\
    $(g_1,\ldots , g_r) \in G_1 \times \cdots \times G_r$ & $\forall \;(h_1,\ldots,h_r) \in H$ \\
     & $\forall \; 1 \leq i,i' \leq r$
\end{tabular}
\right \},
\]
or, equivalently, as
\[ \left \{
\begin{tabular}{c|c}
     & $\mu_i (g_i \cdot G_i^{\circ}, \, \delta_i) = \mu_{i'} (g_{i'} \cdot G_{i'}^{\circ}, \, u_{i,i'}( \delta_i) )$  \\
    $(g_1,\ldots , g_r) \in G_1 \times \cdots \times G_r$ & $\forall \; \delta_i \in \hat{\Gamma}_i$ \\
     & $\forall \; 1 \leq i,i' \leq r$
\end{tabular}
\right \}.
\]
In this way, we see that for any $(g_1,\ldots,g_r) \in p^{-1}(Z_{PGL(U)}(p(H)))$, we have that $g_{i'} \cdot G_{i'}^{\circ} = q_{i,i'} ( g_i \cdot G_i^{\circ} ) $ for all $1 \leq i,i' \leq r$. In other words, 
$$p^{-1}(Z_{PGL(U)}(p(H))) \subseteq G.$$
Swapping the roles of $G$ and $H$ and repeating this argument completes the proof.
\end{proof}

\subsection{Showing we've found all dual pairs in $PGL(U)$} \label{subsec:general-showing-we've-found-all}

The goal of this subsection is to prove that an arbitrary dual pair $(\overline{G},\overline{H})$ in $PGL(U)$ is necessarily of the form described in the previous subsection. To this end, set $G=p^{-1}(\overline{G})$ and $H = p^{-1}(\overline{H})$. Now, instead of assuming that $\{ (\prescript{\delta_j}{}{\varphi}, \prescript{\delta_j}{}{F}) \}_{j \in J}$ is a complete set of representatives of the equivalence classes of $G$-irreducibles in $U$ (as we did in Section \ref{sec:orbits}), let us choose a set of representatives $\{ \varphi_1, \ldots, \varphi_r \}$ for the $\hat{\Gamma}$-orbits of $G$-irreducibles in $U$. Set $\mathcal{J}_i := H/I_H(\varphi_i)$. For each $j_{i,\ell} \in \mathcal{J}_i$, choose a corresponding coset representative $\delta_{j_{i,\ell}} \in \hat{\Gamma}$; for $j_{i,0} = 0$, choose $\delta_{j_{i,0}} = 1$. Then, by construction, 
$$\{ ( \prescript{\delta_{j_{i,\ell}}}{}{\varphi_i}, \prescript{\delta_{j_{i,\ell}}}{}{F_i}  ) \}_{1 \leq i \leq r, \; j_{i,\ell} \in \mathcal{J}_i}$$ 
is a complete set of representatives for the equivalence classes of $G$-irreducibles in $U$. 

The definitions and results of Section \ref{sec:orbits} can be extended (with only superficial changes to the statements and proofs) beyond the single orbit case. In particular, for $g \in G$, $\gamma \in \Gamma$, and $(\beta, B)$ a $G^{\circ}$-irreducible in $U$, $\prescript{g}{}{\beta}$, $\prescript{\wtilde{\gamma}}{}{\beta}$, and $I_G(\beta)$ can be defined in the same way as in the single orbit case. For each $i$, let $\{ (\beta_{i,m}, B_{i,m}) \}_m$ be the set of $G^{\circ}$-irreducibles in $U$ such that 
$$\Res^G_{G^{\circ}} (F_i) = \bigoplus_{m} B_{i,m} \otimes \Hom_{G^{\circ}} (B_{i,m}, F_i).$$
Set $(\beta_i, B_i) := (\beta_{i,1}, B_{i,1})$ and set $\mathcal{K}_i := G/I_G(\beta_i)$. For each $k_{i,m} \in \mathcal{K}_i$, choose a corresponding coset representative $\gamma_{k_{i,m}} \in \Gamma$; for $k_{i,0}=0$, choose $\gamma_{k_{i,0}} = 1$. Then, as in Subsection \ref{subsec:identify-G-circ-irreps}, we can identify 
$$ \bigoplus_{i,m} \prescript{\wtilde{\gamma_{k_{i,m}}}}{}{\beta_i} \otimes \Hom_{G^{\circ}} (  \prescript{\wtilde{\gamma_{k_{i,m}}}}{}{B_i}, U) $$
with the embedding of $G^{\circ}$ in $U$. Under this identification, $G^{\circ}$ gets identified with $\prod_{i,m} GL(\prescript{\wtilde{\gamma_{k_{i,m}}}}{}{B_i}) = \prod_{i,m} GL(B_i)^{\vert \mathcal{K}_i \vert }$.

Next, similar to the single orbit case, for each $1 \leq i \leq r$ and for $\delta \in \hat{\Gamma}$, we can define $(E_i)_{\delta} := \Hom_G ( \prescript{\delta}{}{F_i}, U )$ so that 
$$H^{\circ} = Z_{GL(U)} (G) = \prod_{i,\ell} GL( (E_i)_{\delta_{j_{i,\ell}}} ) .$$
Then for $i,\ell$, define $(\varepsilon_i)_{\delta_{j_{i,\ell}}} \colon H^{\circ} \rightarrow GL((E_i)_{\delta_{j_{i,\ell}}})$ so that the set $\{ ( (\varepsilon_i)_{\delta_{j_{i,\ell}}}, (E_i)_{\delta_{j_{i,\ell}}} ) \}_{i,\ell}$ is a complete set of representatives of the equivalence classes of $H^{\circ}$-irreducibles in $U$. For each $i$, set $E_i := (E_i )_{\delta_{j_{i,0}}} = (E_i)_1$ and $\varepsilon_i := (\varepsilon_i )_{\delta_{j_{i,0}}} = (\varepsilon_i)_1$.

Similarly, define $(A_i)_{\wtilde{\gamma_{k_{i,m}}}} := \Hom_{G^{\circ}} ( \prescript{\wtilde{\gamma_{k_{i,m}}}}{}{B_i}, U )$, so that $H \subseteq \prod_{i,m} GL((A_i)_{\wtilde{\gamma_{k_{i,m}}}})$. Define $(\alpha_i)_{\wtilde{\gamma_{k_{i,m}}}} \colon H \rightarrow GL((A_i)_{\wtilde{\gamma_{k_{i,m}}}})$ as $h \mapsto (h) \vert_{(A_i)_{\wtilde{\gamma_{k_{i,m}}}}}$ so that $\{ ( (\alpha_i)_{\wtilde{\gamma_{k_{i,m}}}}, (A_i)_{\wtilde{\gamma_{k_{i,m}}}} ) \}_{i,m}$ is a complete set of representatives of the equivalence classes of $H$-irreducibles in $U$. For $1 \leq i \leq r$, set $A_i:= (A_i)_{\wtilde{\gamma_{k_{i,0}}}} = (A_i)_{\wtilde{1}}$ and $\alpha_i:= (\alpha_i)_{\wtilde{\gamma_{k_{i,0}}}} = (\alpha_i)_{\wtilde{1}}$.    

For $g\in G$, $\gamma \in \Gamma$, and $(\alpha, A)$ an $H$-irreducible in $U$, we can define $\prescript{g}{}{\alpha}$, $\prescript{\gamma}{}{\alpha}$, and $I_G(\alpha)$ in the same way as in the single orbit case. Continuing to follow the single orbit sections further shows that 
$$\bigoplus_{i,m} \prescript{ \gamma_{k_{i,m}} }{}{\alpha}_i \otimes B_i  $$
can be identified with the embedding of $H$ in $U$. Under this identification, $H$ sits inside of $\prod_{i,m} GL( \prescript{ \gamma_{k_{i,m}} }{}{A}_i ) = \prod_{i,m} GL(A_i)^{\vert \mathcal{K}_i \vert} $, and we see that 
\begin{equation} \label{eq:I_G's-equal-general}
    I_G(\alpha_i) = I_G(\beta_i) =: (I_G)_i
\end{equation}
for each $1 \leq i \leq r$. 

Similarly, for $h \in H$, $\delta \in \hat{\Gamma}$, and $(\varepsilon, E)$ an $H^{\circ}$-irreducible in $U$, we can define $\prescript{h}{}{\varepsilon}$, $\prescript{\wtilde{\delta}}{}{\varepsilon}$, and $I_H(\varepsilon)$ in the same way as in the single orbit case. Following the single orbit sections, we see that 
$$\bigoplus_{i,\ell} \prescript{\wtilde{\delta_{j_{i,\ell}}}}{}{\varepsilon_i} \otimes F_i$$
can be identified with the embedding of $H^{\circ}$ in $U$. Under this identification, $H^{\circ}$ gets identified with $\prod_{i,\ell} GL( \prescript{ \wtilde{\delta_{j_{i,\ell}}} }{}{ E_i} ) = \prod_{i,\ell} GL(E_i)^{\vert \mathcal{J}_i \vert} $, and we see that 
\begin{equation} \label{eq:I_H's-equal-general}
I_H(\varepsilon_i) = I_H(\varphi_i) =: (I_H)_i    
\end{equation}
for each $1 \leq i \leq r$.

By \eqref{eq:I_G's-equal-general} and \eqref{eq:I_H's-equal-general}, for each $1 \leq i \leq r$ we can define 
$$U_i := (F_i \otimes E_i)^{\oplus \vert \mathcal{J}_i \vert} = (A_i \otimes B_i)^{\oplus \vert \mathcal{K}_i \vert} ,$$
so that $U=\bigoplus_{1\leq i \leq r} U_i$. In other words, we have defined subspaces $U_i$ of $U$ such that each $U_i$ corresponds to a single orbit of $H$-, $G^{\circ}$-, $G$-, and $H^{\circ}$-irreducibles. Now, set $G_i := G \vert_{U_i}$ and $H_i := H \vert_{U_i}$. Much of the remainder of this section is dedicated to proving that each $(p_{U_i}(G_i), p_{U_i}(H_i))$ ($1 \leq i \leq r$) is a single-orbit dual pair in $PGL(U_i)$.

\subsubsection{Understanding the component groups of the $G_i$'s and $H_i$'s}

By definition of $G_i$ and $H_i$, we see that 
$$G \subseteq \prod_{1 \leq i \leq r} G_i \hspace{.5cm} \text{ and } \hspace{.5cm}  H \subseteq \prod_{1 \leq i \leq r} H_i.$$  
Also, we see that
$$\prod_{1 \leq i \leq r} Z_{GL(U_i)} (H_i) = Z_{GL(U)} \Big ( \prod_{1\leq i \leq r} H_i \Big )  \subseteq Z_{GL(U)} (H) = G^{\circ} \subseteq \prod_{1 \leq i \leq r} G_i^{\circ} ,$$
and hence that $Z_{GL(U_i)} (H_i) \subseteq G_i^{\circ}$, where we have used Corollary \ref{cor:identity-components-centralizers}. On the other hand, we have that our usual map $\mu \colon G \times H \rightarrow \{ (\dim U)\text{-th roots of unity}\}$ induces a map
\begin{align*}
    \mu_i \colon G_i \times H_i & \rightarrow \{ (\dim U_i)\text{-th roots of unity}\} \\
    (g_i,h_i) & \mapsto g_i h_i g_i^{-1} h_i^{-1}
\end{align*}
for each $1 \leq i \leq r$. Using that a continuous map is constant on connected components, we see that $G_i^{\circ} \subseteq Z_{GL(U_i)}(H_i)$ for each $1 \leq i \leq r$. It follows that
$$\prod_{1 \leq i \leq r} G_i^{\circ} = \prod_{1 \leq i \leq r} Z_{GL(U_i)} (H_i) = G^{\circ} ,$$
and similarly that
$$\prod_{1 \leq i \leq r} H_i^{\circ} = \prod_{1 \leq i \leq r} Z_{GL(U_i)} (G_i) = H^{\circ} .$$
Therefore, we have that
$$\Gamma = G/G^{\circ} \subseteq (G_1 \times \cdots \times G_r)/(G_1^{\circ} \times \cdots \times G_r^{\circ}) =: \Gamma_1 \times \cdots \times \Gamma_r$$
and that
$$\hat{\Gamma} = H/H^{\circ} \subseteq (H_1 \times \cdots \times H_r)/(H_1^{\circ} \times \cdots \times H_r^{\circ}) =: \Delta_1 \times \cdots \times \Delta_r,$$
where $\Gamma_i:= G_i/G_i^{\circ}$ and $\Delta_i:=H_i/H_i^{\circ}$.

\begin{prop} \label{prop:Gamma-i-isos}
    For each $1 \leq i, i' \leq r$, there are isomorphisms 
    \begin{itemize}
        \item $q_{i,i'} \colon \Gamma_i \xrightarrow{\sim} \Gamma \xrightarrow{\sim} \Gamma_{i'}$, and
        \item $u_{i,i'} \colon \hat{\Gamma}_i \xrightarrow{\sim} \hat{\Gamma} \xrightarrow{\sim} \hat{\Gamma}_{i'}$
    \end{itemize}
    such that $\mu_i ( \gamma_i, \delta_i ) = \mu_{i'} ( q_{i,i'}(\gamma_i), u_{i,i'}(\delta_i) ) $ for all $\gamma_i \in \Gamma_i$ and $\delta_i \in \hat{\Gamma}_i$, and such that $q_{i',i''} \circ q_{i,i'} = q_{i,i''}$ and $u_{i',i''} \circ u_{i,i'} = u_{i,i''}$ for all $1 \leq i,i',i'' \leq r$.
\end{prop}

\begin{proof}
We start by showing that we have an isomorphism $\Gamma \simeq \Gamma_i$ for each $1 \leq i \leq r$. For each $1 \leq i \leq r$, we have a homomorphism $Q_i \colon \Gamma \rightarrow \Gamma_{i}$, where the coset $gG^{\circ} \in \Gamma$ corresponding to $g = (g_1,\ldots , g_r) \in G$ gets sent to $g_i G_i^{\circ} \in \Gamma_i$. To check that this is well-defined, notice that 
$$gG^{\circ} = g_1 G_1^{\circ} \times \cdots \times g_r G_r^{\circ}.$$ 
Therefore, if we have $g G^{\circ} = g' G^{\circ} \in \Gamma$, it is clear that $g_i G_i^{\circ} = g_i' G_i^{\circ} \in \Gamma_i$. We see that $Q_i$ is surjective: Indeed, by definition of $G_i$, we have that for any $g_i \in G_i$, there is some $g \in G$ with $U_i$-component equal to $g_i$. 

Next, we claim that $Q_i$ is injective. We will prove the result for $Q_1$, but the same proof idea works for $2 \leq i \leq r$. Suppose that $\gamma = (\gamma_1, \gamma_2, \ldots , \gamma_r)$ and $\gamma' = (\gamma_1, \gamma_2', \ldots , \gamma_r')$ are elements of $\Gamma$ that both map to $\gamma_1 \in \Gamma_1$ under $Q_1$. Then for all $\delta = (\delta_1, \ldots, \delta_r) \in \hat{\Gamma}$, we have that 
$$ \mu (\gamma, \delta) = \mu_1(\gamma_1, \delta_1) = \mu (\gamma ' , \delta).$$
It follows that $\gamma = \gamma '$, and hence that $Q_1$ is injective, as desired.

Hence we have established that for each $1 \leq i \leq r$, there is an isomorphism $Q_i \colon \Gamma \xrightarrow{\sim} \Gamma_i$. By the same argument, we see that there is an isomorphism $U_i \colon \hat{\Gamma} \xrightarrow{\sim} \hat{\Gamma}_i$. For each $1 \leq i,i' \leq r$, define
$$q_{i,i'} := Q_{i'} \circ Q_i^{-1} \colon \Gamma_i \xrightarrow{\sim} \Gamma \xrightarrow{\sim} \Gamma_{i'} $$
and
$$u_{i,i'} := U_{i'} \circ U_i^{-1} \colon \hat{\Gamma}_i \xrightarrow{\sim} \hat{\Gamma} \xrightarrow{\sim} \hat{\Gamma}_{i'} .$$
By construction, it is clear that $\mu_i ( \gamma_i, \delta_i ) = \mu_{i'} ( q_{i,i'}(\gamma_i), u_{i,i'}(\delta_i) ) $ for all $\gamma_i \in \Gamma_i$ and $\delta_i \in \hat{\Gamma}_i$, and that $q_{i',i''} \circ q_{i,i'} = q_{i,i''}$ and $u_{i',i''} \circ u_{i,i'} = u_{i,i''}$ for all $1 \leq i,i',i'' \leq r$, completing the proof.    
\end{proof}

\begin{cor} \label{cor:q_i,i'-Gamma-picture}
The groups $\Gamma$ and $\hat{\Gamma}$ can be identified with
$$\Gamma = \{ ( \gamma_1, \, q_{1,2}(\gamma_1), \, q_{1,3}(\gamma_1), \, \ldots, \, q_{1,r}(\gamma_1) ) \}_{\gamma_1 \in \Gamma_1} \subset \Gamma_1 \times \cdots \times \Gamma_r$$
and
$$\hat{\Gamma} = \{ ( \delta_1, \, u_{1,2}(\delta_1), \, u_{1,3}(\delta_1), \, \ldots, \, u_{1,r}(\delta_1) ) \}_{\delta_1 \in \hat{\Gamma}_1} \subset \hat{\Gamma}_1 \times \cdots \times \hat{\Gamma}_r.$$
\end{cor}

Now note that $(I_G)_i$ and $(I_H)_i$ (as subgroups of $G$ and $H$, respectively) sit inside of $\prod_i GL(U_i)$. With this in mind, it is not hard to see that
$$ (I_G)_i \vert_{U_i} = I_{G_i}(\beta_i \vert_{G_i}) =: I_{G_i} \hspace{.5cm} \text{ and } \hspace{.5cm} (I_H)_i \vert_{U_i} = I_{H_i}(\varepsilon_i \vert_{H_i}) =: I_{H_i}.$$
The proof of Proposition \ref{prop:Gamma-i-isos} gives that 
$$I_G(\beta_i)/G^{\circ} \simeq I_{G_i}/G_i^{\circ} \hspace{.5cm} \text{ and } \hspace{.5cm} I_H(\varepsilon_i)/H^{\circ} \simeq I_{H_i}/H_i^{\circ} .$$
Therefore, for each $i$ we have 
$$1 \rightarrow G_i^{\circ} \rightarrow G_i \rightarrow \Gamma_i \simeq \Gamma \rightarrow 1 \hspace{.5cm} \text{ and } \hspace{.5cm} 1 \rightarrow I_{G_i}/G_i^{\circ} \rightarrow \Gamma_i \simeq \Gamma \rightarrow \mathcal{K}_i \rightarrow 1.$$
In the same way,
$$1 \rightarrow H_i^{\circ} \rightarrow H_i \rightarrow \hat{\Gamma}_i \simeq \hat{\Gamma} \rightarrow 1 \hspace{.5cm} \text{ and } \hspace{.5cm} 1 \rightarrow I_{H_i}/H_i^{\circ} \rightarrow \hat{\Gamma}_i \simeq \hat{\Gamma} \rightarrow \mathcal{J}_i \rightarrow 1.$$
In this way, we see that on each $U_i$-factor, the pair $(G_i,H_i)$ looks a lot like the single-orbit dual pair picture outlined earlier in the paper. 

\subsubsection{Showing that each $(p_{U_i}(G_i), p_{U_i}(H_i))$ is a single-orbit dual pair in $PGL(U_i)$}

Towards showing that $(p_{U_i}(G_i), p_{U_i}(H_i))$ indeed gives a single-orbit dual pair in $PGL(U_i)$, notice that we have the following relationships between the groups of distinguished automorphisms in $G_i$, $G$, $H_i$, and $H$:
$$(G_i)_d = (G_d)\vert_{U_i} \hspace{.5cm} \text{and} \hspace{.5cm} (I_{G_i})_d = (I_G)_d \vert_{U_i} = I_{(G_i)_d}.$$
Similarly,
$$(H_i)_d = (H_d)\vert_{U_i} \hspace{.5cm} \text{and} \hspace{.5cm} (I_{H_i})_d = (I_H)_d \vert_{U_i} = I_{(H_i)_d}.$$
Now, as in Section \ref{sec:reps-of-I_G-and-I_H}, for each fixed $i$, we can identify the spaces
$$ \Hom_{G^{\circ}} ( \prescript{\wtilde{\gamma_{k_{i,m}}}}{}{B_i}, \prescript{\delta_{j_{i,\ell }}}{}{F_i}  ) = \Hom_{H^{\circ}} ( \prescript{\wtilde{\delta_{j_{i,\ell'}}}}{}{E_i}, \prescript{\gamma_{k_{i,m '}}}{}{A_i}  ) =: L_i$$
for all $\ell$, $\ell'$, $m$, and $m'$. When we wish to refer to the particular copy of $L_i$ corresponding to $j \in \mathcal{J}_i$ and $k \in \mathcal{K}_i$, we will write $L_{i,j,k}$. With this established, we can now state general analogs to many of the results in Section \ref{sec:explicit}. The proofs from Section \ref{sec:explicit} can be carried over with only superficial modification.

\begin{prop}[Analog of Proposition \ref{prop:I_G-and-I_H}] \label{prop:I_G_i-general}
\hspace{.5cm}
    \begin{itemize}
        \item $I_{G_i} = [ GL(B_i)^{\vert \mathcal{K}_i \vert} \times (I_{G_i})_d ] / ((\mathbb{C}^{\times})^{\vert \mathcal{K}_i \vert})_{\antidiag}$, and $(I_{G_i})_d \subseteq GL(L_i)^{\vert \mathcal{J}_i \vert \cdot \vert \mathcal{K}_i \vert}$.
        \item $I_{H_i} = [ GL(E_i)^{\vert \mathcal{J}_i \vert} \times (I_{H_i})_d ] / ((\mathbb{C}^{\times})^{\vert \mathcal{J}_i \vert})_{\antidiag}$, and $(I_{H_i})_d \subseteq GL(L_i)^{\vert \mathcal{J}_i \vert \cdot \vert \mathcal{K}_i \vert}$.
    \end{itemize}
\end{prop}

Now, Proposition \ref{prop:I_G_i-general} gives that $(I_{G_i})\vert_{L_{i,j,k}} = GL(B_i) \times_{\mathbb{C}^{\times}} (I_{G_i})_d \vert_{L_{i,j,k}}$ for $j \in \mathcal{J}_i$ and $k \in \mathcal{K}_i$. Define 
$$\Omega_{i,j,k} := (I_{G_i})\vert_{L_{i,j,k}} / GL(B_i) = (I_{G_i})_d \vert_{L_{i,j,k}} / \mathbb{C}^{\times} .$$
Similarly, define 
$$\Lambda_{i,j,k} := (I_{H_i})\vert_{L_{i,j,k}} / GL(E_i) = (I_{H_i})_d \vert_{L_{i,j,k}} / \mathbb{C}^{\times} .$$
These can be viewed as finite subgroups of $PGL(L_i)$.

\begin{lemma}[Analog of Lemma \ref{lem:dual-inertia-comp}]
    For all $1 \leq i \leq r$, $j \in \mathcal{J}_i$, and $k \in \mathcal{K}_i$, $\Omega_{i,j,k}$ and $\Lambda_{i,j,k}$ are finite abelian groups with natural isomorphisms
    $$ \Omega_{i,j,k} \simeq \hat{\Lambda}_{i,j,k} \hspace{.5cm} \text{ and } \hspace{.5cm} \Lambda_{i,j,k} \simeq \hat{\Omega}_{i,j,k} .$$ 
\end{lemma}

\begin{prop}[Analog of Proposition \ref{prop:PGL(Ujk)-dps}]
    For all $1 \leq i \leq r$, $j \in \mathcal{J}_i$, and $k \in \mathcal{K}_i$, $( (I_{G_i})\vert_{L_{i,j,k}} /\mathbb{C}^{\times}, (I_{H_i})\vert_{L_{i,j,k}}/\mathbb{C}^{\times}  )$ is a dual pair in $PGL(B_i \otimes E_i \otimes L_i)$.
\end{prop}

\begin{prop}[Analog of Lemma \ref{lem:omegas-are-all-the-same}, Corollaries \ref{cor:reducing-to-PGL(L)-dp} and \ref{cor:|L|=dimL}, and Propositions \ref{prop:Omega=Omega-hat} and \ref{prop:omega=LxL}]
    For all $1 \leq i \leq r$, $j \in \mathcal{J}_i$, and $k \in \mathcal{K}_i$, we have the following:
    \begin{itemize}
        \item There are natural isomorphisms $\Omega_{i,j,k} \simeq \Omega_{i,0,0} =: \Omega_i$ and $\hat{\Omega}_{i,j,k} \simeq \hat{\Omega}_{i,0,0}$.
        \item $(\Omega_i , \hat{\Omega}_i)$ is a dual pair in $PGL(L_i)$.
        \item $\Omega_i$ equals its dual: $\Omega_i = \hat{\Omega}_i$.
        \item $\Omega_i$ can be written as $\Omega_i = \mathcal{L}_i \times \hat{\mathcal{L}}_i$ for some finite abelian group $\mathcal{L}_i$ with $\vert \mathcal{L}_i \vert = \dim L_i$.
    \end{itemize}
\end{prop}

Set $J_i := L^2( \mathcal{J}_i , \mathbb{C} )$ and $K_i := L^2 ( \mathcal{K}_i , \mathbb{C} )$.

\begin{thm}[Analog of Proposition \ref{prop:G_d-and-H_d}, Proposition \ref{prop:(I_G)_d-and-(I_H)_d}, and Theorem \ref{thm:G_d}] \label{thm:(G_i)_d}
For $1\leq i \leq r$, we have the following:
\begin{itemize}
    \item $(G_i)_d = (I_{G_i})_d \rtimes_{\mathbb{C}^{\times}} p_{K_i}^{-1}(\mathcal{K}_i) = \left [ (\mathbb{C}^{\times})^{\vert \mathcal{K}_i \vert} \times_{\mathbb{C}^{\times}} p_{L_i}^{-1} (\mathcal{L}_i \times \hat{\mathcal{L}}_i ) \times_{\mathbb{C}^{\times}} p_{J_i}^{-1} (\hat{\mathcal{J}}_i)  \right ] \rtimes_{\mathbb{C}^{\times}} p_{K_i}^{-1}(\mathcal{K}_i)$.
    \item $(H_i)_d = (I_{H_i})_d \rtimes_{\mathbb{C}^{\times}} p_{J_i}^{-1}(\mathcal{J}_i) = \left [ (\mathbb{C}^{\times})^{\vert \mathcal{J}_i \vert} \times_{\mathbb{C}^{\times}} p_{L_i}^{-1} (\mathcal{L}_i \times \hat{\mathcal{L}}_i ) \times_{\mathbb{C}^{\times}} p_{K_i}^{-1} (\hat{\mathcal{K}}_i)  \right ] \rtimes_{\mathbb{C}^{\times}} p_{J_i}^{-1}(\mathcal{J}_i)$.
    \item $G_i = [GL(B_i)^{\vert \mathcal{K}_i \vert} \rtimes (G_i)_d] / ( (\mathbb{C}^{\times})^{\vert \mathcal{K}_i \vert} )_{\antidiag}$.
    \item $H_i = [GL(E_i)^{\vert \mathcal{J}_i \vert} \rtimes (H_i)_d] / ( (\mathbb{C}^{\times})^{\vert \mathcal{J}_i \vert} )_{\antidiag}$.
\end{itemize}
\end{thm}

\begin{cor}[Analog of Corollaries \ref{cor:Gamma=I_GxK} and \ref{cor:I_G/G0=JxOmega}]
For $1 \leq i \leq r$, we have the following:
\begin{itemize}
    \item $\Gamma_i = I_{G_i}/G_i^{\circ} \times \mathcal{K}_i = \Omega_i \times \hat{\mathcal{J}}_i \times  \mathcal{K}_i$.
    \item $\hat{\Gamma}_i = I_{H_i}/H_i^{\circ} \times \mathcal{J}_i = \Omega_i \times \hat{\mathcal{K}}_i \times  \mathcal{J}_i$.
\end{itemize}    
\end{cor}

\begin{thm} \label{thm:single-orbit-reduction}
For $1 \leq i \leq r$, we have the following:
\begin{itemize}
    \item The pair $( (G_i)_d , (H_i)_d )$ descends to a single-orbit dual pair in $PGL(L_i \otimes J_i \otimes K_i )$ with component groups $(\Gamma, \hat{\Gamma})$.
    \item The pair $( G_i, H_i )$ descends to a single-orbit dual pair in $PGL(U_i)$ with component groups $(\Gamma, \hat{\Gamma})$.
\end{itemize}
\end{thm}

\begin{proof}
This follows immediately from Theorem \ref{thm:(G_i)_d}, Theorem \ref{thm:single-orbit-general}, and Proposition \ref{prop:Gamma-i-isos}.
\end{proof}

\begin{thm} \label{thm:general-is-of-the-right-form}
    The pair $(G,H)$ is of the form described in Subsection \ref{subsec:general-construction} and Theorem \ref{thm:general-construction}.
\end{thm}

\begin{proof}
    This follows immediately from Theorem \ref{thm:single-orbit-reduction} and Corollary \ref{cor:q_i,i'-Gamma-picture}.
\end{proof}

\section{Conclusion: The ``ingredients" that uniquely determine a dual pair in $PGL(U)$}

\subsection{Single-orbit dual pairs}

By Theorems \ref{thm:single-orbit-general} and \ref{thm:G_d}, a single-orbit dual pair is uniquely specified by the following:
\begin{itemize}
    \item finite-dimensional complex vector spaces $B$ and $E$, and
    \item finite abelian groups $\mathcal{L}$, $\mathcal{J}$, and $\mathcal{K}$.
\end{itemize}
Set $L = L^2(\mathcal{L},\mathbb{C})$, $J = L^2(\mathcal{J},\mathbb{C})$, and $K = L^2(\mathcal{K},\mathbb{C})$. Additionally, set
$$G = \left [ GL(B)^{\vert \mathcal{K} \vert} \times_{\mathbb{C}^{\times}} p_L^{-1}(\mathcal{L} \times \hat{\mathcal{L}}) \times_{\mathbb{C}^{\times}} p_J^{-1}( \hat{\mathcal{J}})  \right ] \rtimes_{\mathbb{C}^{\times}} p_K^{-1}(\mathcal{K})$$
and
$$H = \left [ GL(E)^{\vert \mathcal{J} \vert} \times_{\mathbb{C}^{\times}} p_L^{-1}(\mathcal{L} \times \hat{\mathcal{L}}) \times_{\mathbb{C}^{\times}} p_K^{-1}(\hat{\mathcal{K}})  \right ] \rtimes_{\mathbb{C}^{\times}} p_J^{-1}(\mathcal{J}).$$
Then $(G,H)$ descends to a dual pair in $PGL(B \otimes E \otimes L \otimes J \otimes K)$ (Theorem \ref{thm:single-orbit-general}), and every single-orbit dual pair is of this form (Theorem \ref{thm:G_d}).

\subsection{Multi-orbit dual pairs}

By Theorems \ref{thm:general-construction} and \ref{thm:general-is-of-the-right-form}, an $r$-orbit dual pair is uniquely specified by the following:
\begin{itemize}
    \item a finite abelian group $\Gamma$;
    \item for each $1 \leq i \leq r$, a finite-dimensional complex vector space $U_i$, a pair of subgroups $(G_i,H_i)$ in $GL(U_i)$ that descend to a single-orbit dual pair in $PGL(U_i)$, and an isomorphism $q_{i} \colon \Gamma \xrightarrow{\sim} \Gamma_i$ (where $\Gamma_i := G_i/G_i^{\circ}$).
\end{itemize}
Set $U = \bigoplus_{i = 1}^{r} U_i$. For each $1 \leq i \leq r$, let $u_{i} \colon \hat{\Gamma} \rightarrow \hat{\Gamma_i}$ be such that $u_{i}(\delta)( q_{i}(\gamma) ) = \delta (\gamma)$ for all $\gamma \in \Gamma$ and $\delta \in \hat{\Gamma}$. Let $G$ and $H$ be the subgroups of $G_1 \times \cdots \times G_r$ and $H_1 \times \cdots \times H_r$ corresponding to  
$$\{ ( q_{1} (\gamma), \, q_{2}(\gamma), \, q_{3}(\gamma), \, \ldots, \, q_{r}(\gamma) ) \}_{\gamma \in \Gamma} \subset \Gamma_1 \times \cdots \times \Gamma_r$$
and
$$\{ ( u_{1}(\delta), \, u_{2}(\delta), \, u_{3}(\delta), \, \ldots, \, u_{r}(\delta) ) \}_{\delta \in \hat{\Gamma}} \subset \hat{\Gamma}_1 \times \cdots \times \hat{\Gamma}_r.$$
Then $(G,H)$ descends to a dual pair in $PGL(U)$ (Theorem \ref{thm:general-construction}), and every dual pair in $PGL(U)$ is of this form (Theorem \ref{thm:general-is-of-the-right-form}).

\section*{Acknowledgements}

The author would like to thank David Vogan for suggesting this topic of study and for his guidance throughout the project.

\printbibliography

\newpage

\section*{Index of Notation}

\renewcommand{\arraystretch}{1.45}

\begin{tabularx}{\linewidth}{c c X}
\nind{$U$}{a finite-dimensional complex vector space}{Section \ref{sec:introduction}} 
\\
\nind{$Z_G(\cdot)$}
{the centralizer of $\cdot$ in the group $G$}
{Section \ref{sec:introduction}}
\\
\nind{$\mf{z}_{\mf{g}}(\cdot)$}
{the centralizer of $\cdot$ in the Lie algebra $\mf{g}$}
{Section \ref{sec:introduction}}
\\
\nind{$p_V$}
{the natural projection $GL(V) \rightarrow PGL(V)$ (written as $p$ if $V$ is clear)}
{Section \ref{sec:prelims-for-showing-all}}
\\
\nind{$\times_{\mathbb{C}^{\times}}$}{for $\mathbb{C}^{\times} \subseteq H_i \subseteq GL(U_i)$ ($i=1,2$), $H_1 \times_{\mathbb{C}^{\times}} H_2 := [H_1 \times H_2]/(\mathbb{C}^{\times})_{\antidiag}$ in $GL(U_1 \otimes U_2)$}{Section \ref{sec:constructing-single-orbit}, Eq.~\ref{eq:weird-direct-product-def}} 
\\
\nind{$\rtimes_{\mathbb{C}^{\times}}$}{for $\mathcal{X} \subseteq PGL(X)$ a permutation group and $\mathbb{C}^{\times} \subseteq H \subseteq GL(W)^{\dim X}$ normalized by $p_X^{-1}(\mathcal{X})$, $H \rtimes_{\mathbb{C}^{\times}} p_X^{-1}(\mathcal{X}) := [H \rtimes p_X^{-1}(\mathcal{X})]/(\mathbb{C}^{\times})_{\antidiag}$ in $GL(W \otimes X)$}{Section \ref{sec:constructing-single-orbit}, Eq.~\ref{eq:weird-semidirect-product-def}} 
\\
\nind{$(\overline{G}, \overline{H})$}
{a dual pair in $PGL(U)$ (assumed to be ``single-orbit" in Sections \ref{sec:prelims-for-showing-all}--\ref{sec:explicit})}
{Section \ref{sec:prelims-for-showing-all}} 
\\
\nind{$(G, H)$}
{the preimages $G:=p^{-1}(\overline{G})$, $H:=p^{-1}(\overline{H})$}
{Section \ref{sec:prelims-for-showing-all}}
\\
\nind{$\Gamma$, $\Delta$}
{the quotient groups $\Gamma:=G/G^{\circ}$, $\Delta := H/H^{\circ}$}
{Section \ref{sec:prelims-for-showing-all}, Eq.~(\ref{eq:short-exact-sequence})}
\\
\nind{$\pi_G$, $\pi_H$}
{the quotient maps $\pi_G \colon G \rightarrow \Gamma$, $\pi_H \colon H \rightarrow \Delta$}
{Section \ref{sec:prelims-for-showing-all}, Eq.~(\ref{eq:short-exact-sequence})}
\\
\nind{$\mu$}
{ -- the map $\mu \colon G \times H \rightarrow \mathbb{C}^{\times}$, $(g,h) \mapsto ghg^{-1}h^{-1}$}
{Section \ref{sec:prelims-for-showing-all}, Eq.~(\ref{al:mu-G-H})}
\\
\nind{}
{ -- the map $\mu \colon \Gamma \times \Delta \rightarrow \mathbb{C}^{\times}$ induced by $\mu \colon G \times H \rightarrow \mathbb{C}^{\times}$}
{Section \ref{sec:prelims-for-showing-all}, Eq.~(\ref{eq:mu-gamma-delta})}
\\      
\nind{$\hat{\Gamma}$}
{the character group of $\Gamma$, naturally isomorphic to $\Delta$}
{Theorem \ref{thm:dual-finite-ab-gps}}
\\
\nind{$G_d$, $H_d$}
{the subgroups of $G$, $H$ defining distinguished automorphisms of $G^{\circ}$, $H^{\circ}$}
{Section \ref{sec:prelims-for-showing-all}}
\\
\nind{$s_G$, $s_H$}
{the 2-cocycles $s_G \colon \Gamma \times \Gamma \rightarrow Z(G^{\circ})$, $s_H \colon \hat{\Gamma} \times \hat{\Gamma} \rightarrow Z(H^{\circ})$ defining the extensions in Eq.~(\ref{eq:pairs-of-extensions})}
{Section \ref{sec:prelims-for-showing-all}, Eq.~(\ref{eq:s_G-and-s_H})}
\\
\nind{$\wtilde{(\cdot)}$}
{(depending on context) the lift $\wtilde{(\cdot)} \colon \Gamma \rightarrow G$ satisfying $\wtilde{\gamma} \wtilde{\gamma'} = s_G (\gamma , \gamma') \wtilde{\gamma \gamma'}$ or the lift $\wtilde{(\cdot)} \colon \hat{\Gamma} \rightarrow H$ satisfying $\wtilde{\delta} \, \wtilde{\delta'} = s_H (\delta , \delta') \wtilde{\delta \delta'}$}
{Subsec.~\ref{subsec:choose-Gamma-reps}, Eq.~(\ref{eq:def-tilde-lift})}
\\    
\nind{$(\prescript{h}{}{\varphi_j}, \prescript{h}{}{F_j})$}
{a $G$-irrep $(\varphi_j, F_j)$ twisted by $h \in H$: $\prescript{h}{}{\varphi_j}(g) := \varphi_j (hgh^{-1})$}
{Subsec.~\ref{subsec:hat-Gamma-action-on-G-irreps}, Eq.~(\ref{eq:phi_i-twisted-by-h})} 
\\
\nind{$(\prescript{\delta}{}{\varphi_j}, \prescript{\delta}{}{F_j})$}
{a $G$-irrep $(\varphi_j, F_j)$ twisted by $\delta \in \hat{\Gamma}$: $\prescript{\delta}{}{\varphi_j}(\cdot) := \varphi_j (\cdot) \otimes \mu (\cdot , \delta)^{-1}$}
{Subsec.~\ref{subsec:hat-Gamma-action-on-G-irreps}, Eq.~(\ref{eq:phi-twisted-by-delta})}  
\\
\nind{$I_H(\varphi_j)$}
{the inertia group $I_H(\varphi_j) := \{ h \in H \, : \, \prescript{h}{}{\varphi_j} \simeq \varphi_j \}$ of a $G$-irrep $\varphi_j$}
{Subsec.~\ref{subsec:hat-Gamma-action-on-G-irreps}} 
\\
\nind{$(\varphi , F)$}
{a fixed irreducible subrepresentation of the embedding of $G$ in $U$}
{Subsec.~\ref{subsec:G-irred-reps}} 
\\
\nind{$\mathcal{J}$}
{the quotient group $\mathcal{J}:= H/I_H(\varphi)$}
{Subsec.~\ref{subsec:G-irred-reps}} 
\\
\nind{$\{ \delta_j \}_{j \in \mathcal{J}}$}
{coset representatives for $I_H(\varphi)/H^{\circ}$ in $\hat{\Gamma}$ (and a lift $\mathcal{J} \rightarrow \hat{\Gamma}$ given by $j \mapsto \delta_j$)}
{Subsec.~\ref{subsec:G-irred-reps}} 
\\
\nind{$(\prescript{g}{}{\beta_i}, \prescript{g}{}{B_i})$}
{a $G^{\circ}$-irrep $(\beta_i, B_i)$ twisted by $g \in G$: $\prescript{g}{}{\beta_i}(g_0) := \beta_i (gg_0g^{-1})$}
{Subsec.~\ref{subsec:Gamma-action-on-G-circ}, Eq.~(\ref{eq:beta-twisted-by-g})}  
\\
\nind{$I_G(\beta_i)$}
{the inertia group $I_G(\beta_i) := \{ g \in G \, : \, \prescript{g}{}{\beta_i} \simeq \beta_i \}$ of a $G^{\circ}$-irrep $\beta_i$}
{Subsec.~\ref{subsec:Gamma-action-on-G-circ}} 
\\
\nind{$(\beta , B)$}
{a fixed $G^{\circ}$-irreducible subrepresentation of $\Res^G_{G^{\circ}} (F)$}
{Subsec.~\ref{subsec:Gamma-action-on-G-circ}} 
\\
\nind{$\mathcal{K}$}
{the quotient group $\mathcal{K}:= G/I_G(\beta)$}
{Subsec.~\ref{subsec:Gamma-action-on-G-circ}} 
\\
\nind{$\{ \gamma_k \}_{k \in \mathcal{K}}$}
{coset representatives for $I_G(\beta)/G^{\circ}$ in $\Gamma$ (and a lift $\mathcal{K} \rightarrow \Gamma$ given by $k \mapsto \gamma_k$)}
{Subsec.~\ref{subsec:Gamma-action-on-G-circ}} 
\\
\nind{$E_{\delta}$}
{the multiplicity space $E_{\delta} := \Hom_G(\prescript{\delta}{}{F}, U)$ for $\delta \in \hat{\Gamma}$}
{Subsec.~\ref{subsec:H0-irreps}, Eq.~(\ref{eq:E_delta})} 
\\ 
\nind{$\varepsilon_{\delta_j}$}
{the $H^{\circ}$-irreducible $\varepsilon_{\delta_j} \colon H^{\circ} \rightarrow GL(E_{\delta_j})$ defined as $h_0 \in H^{\circ} \mapsto (h_0)\vert_{E_{\delta_j}}$ for $j \in \mathcal{J}$}
{Subsec.~\ref{subsec:H0-irreps}} 
\\
\nind{$(\varepsilon, E)$}
{the $H^{\circ}$-irreducible $\varepsilon := \varepsilon_1 = \varepsilon_{\delta_0}$ with representation space $E:= E_1 = E_{\delta_0}$}
{Subsec.~\ref{subsec:H0-irreps}} 
\\
\nind{$A_{\wtilde{\gamma}}$}
{the multiplicity space $A_{\wtilde{\gamma}} := \Hom_{G^{\circ}}(\prescript{\wtilde{\gamma}}{}{B}, U)$ for $\gamma \in \Gamma$}
{Subsec.~\ref{subsec:define-H-irreps}, Eq.~(\ref{eq:A_gamma})} 
\\ 
\nind{$\alpha_{\wtilde{\gamma_k}}$}
{the $H$-irreducible $\alpha_{\wtilde{\gamma_k}} \colon H \rightarrow GL(A_{\wtilde{\gamma_k}})$ defined as $h \in H \mapsto (h)\vert_{A_{\wtilde{\gamma_k}}}$ for $k \in \mathcal{K}$}
{Subsec.~\ref{subsec:define-H-irreps}} 
\\
\nind{$(\alpha, A)$}
{the $H$-irreducible $\alpha := \alpha_{\wtilde{1}} = \alpha_{\wtilde{\gamma_0}}$ with representation space $A:= A_{\wtilde{1}} = A_{\wtilde{\gamma_0}}$}
{Subsec.~\ref{subsec:define-H-irreps}} 
\\
\nind{$(\prescript{g}{}{\alpha_i}, \prescript{g}{}{A_i})$}
{an $H$-irrep $(\alpha_i, A_i)$ twisted by $g \in G$: $\prescript{g}{}{\alpha_i}(h) := \alpha_i (ghg^{-1})$}
{Subsec.~\ref{subsec:Gamma-action-on-H-irreps}, Eq.~(\ref{eq:alpha-twisted-by-g})} 
\\ 
\nind{$(\prescript{\gamma}{}{\alpha_i}, \prescript{\gamma}{}{A_i})$}
{an $H$-irrep $(\alpha_i, A_i)$ twisted by $\gamma \in \Gamma$: $\prescript{\gamma}{}{\alpha_i}(\cdot) := \alpha_i (\cdot) \otimes \mu (\gamma , \cdot)$}
{Subsec.~\ref{subsec:Gamma-action-on-H-irreps}, Eq.~(\ref{eq:alpha-twisted-by-gamma})} 
\\
\nind{$I_G(\alpha_i)$}
{the inertia group $I_G(\alpha_i) := \{ g \in G \, : \, \prescript{g}{}{\alpha_i} \simeq \alpha_i \}$ of an $H$-irrep $\alpha_i$}
{Subsec.~\ref{subsec:Gamma-action-on-H-irreps}} 
\\
\nind{$I_G$}
{the inertia group $I_G := I_G(\beta) = I_G(\alpha)$}
{Subsec.~\ref{subsec:Gamma-action-on-H-irreps}, Eq.~(\ref{eq:I_G})} 
\\
\nind{$(\prescript{h}{}{\varepsilon_j}, \prescript{h}{}{E_j})$}
{an $H^{\circ}$-irrep $(\varepsilon_j, E_j)$ twisted by $h \in H$: $\prescript{h}{}{\varepsilon_j}(h_0) := \varepsilon_j (hh_0h^{-1})$}
{Subsec.~\ref{subsec:defining-hat-Gamma-action}, Eq.~(\ref{eq:epsilon-twisted-by-h})}  
\\
\nind{$I_H(\varepsilon_j)$}
{the inertia group $I_H(\varepsilon_j) := \{ h \in H \, : \, \prescript{h}{}{\varepsilon_j} \simeq \varepsilon_j \}$ of an $H^{\circ}$-irrep $\varepsilon_j$}
{Subsec.~\ref{subsec:defining-hat-Gamma-action}} 
\\
\nind{$I_H$}
{the inertia group $I_H := I_H(\varphi) = I_H(\varepsilon)$}
{Subsec.~\ref{subsec:defining-hat-Gamma-action}, Eq.~(\ref{eq:I_H})} 
\\
\nind{$L$}
{the vector space $L := \Hom_{G^{\circ}} (B,F)$, naturally isomorphic to $\Hom_{H^{\circ}}(E,A)$}
{Section \ref{sec:reps-of-I_G-and-I_H}} 
\\
\nind{$L_{j,k}$}
{the vector space $L_{j,k} := \Hom_{G^{\circ}} ( \prescript{\wtilde{\gamma_k}}{}{B}, \prescript{\delta_j}{}{F} ) \simeq_{\text{naturally}} \Hom_{H^{\circ}} ( \prescript{\wtilde{\delta_j}}{}{E}, \prescript{\gamma_k}{}{A} )$, naturally isomorphic to $L$ for all $j \in \mathcal{J}$, $k\in \mathcal{K}$}
{Section \ref{sec:reps-of-I_G-and-I_H}} 
\\
\nind{$(\rho_{j,k}, R_{j,k})$}
{the $I_G$-irreducible satisfying $\rho_{j,k} \preceq \Ind_{G^{\circ}}^{I_G} (\prescript{\wtilde{\gamma_k}}{}{\beta})$ and $\Ind_{I_G}^G (\rho_{j,k}) = \prescript{\delta_j}{}{\varphi}$}
{Subsec.~\ref{subsec:define-I_G-and-I_H-irreps}, Eq.~(\ref{eq:rho-jk-def})} 
\\
\nind{$(\psi_{j,k}, P_{j,k})$}
{the $I_H$-irreducible satisfying $\psi_{j,k} \preceq \Ind_{H^{\circ}}^{I_H} (\prescript{\wtilde{\delta_j}}{}{\varepsilon})$ and $\Ind_{I_H}^H (\psi_{j,k}) = \prescript{\gamma_k}{}{\alpha}$}
{Subsec.~\ref{subsec:define-I_G-and-I_H-irreps}, Eq.~(\ref{eq:psi-jk-def})} 
\\
\nind{$\Omega_{j,k}$, $\Lambda_{j,k}$}
{the quotient groups $\Omega_{j,k} := \rho_{j,k}(I_G)/GL(B) = (I_G)_d \vert_{L_{j,k}}/\mathbb{C}^{\times}$ and $\Lambda_{j,k} := \psi_{j,k}(I_H)/GL(E) = (I_H)_d\vert_{L_{j,k}}/\mathbb{C}^{\times}$}
{Subsec.~\ref{subsec:PGL(L)-dual-pairs}} 
\\
\nind{$\hat{\Omega}_{j,k}$}
{the character group of $\Omega_{j,k}$, naturally isomorphic to $\Lambda_{j,k}$}
{Lemma \ref{lem:dual-inertia-comp}} 
\\
\nind{$\Omega$}
{the group $\Omega := \Omega_{0,0}$, naturally isomorphic to $\Omega_{j,k}$ for all $j \in \mathcal{J}$, $k \in \mathcal{K}$}
{Subsec.~\ref{subsec:omegas-are-all-the-same}}
\\   
\end{tabularx}

\end{document}